\documentclass[reqno]{amsart}

\usepackage{amssymb,mathrsfs,bm}
\usepackage{graphicx}
\usepackage{enumitem}
\usepackage{xcolor}
\usepackage{microtype}
\usepackage[
  colorlinks=true,
  citecolor=blue,
  linkcolor=blue,
  urlcolor=black
]{hyperref}
\hypersetup{
  pdftitle={Quantitative asymptotic error analysis and Richardson extrapolation for exponential Euler approximations of stochastic wave equations},
  pdfauthor={Jianbo Cui, Diancong Jin, Huagui Liu, and Derui Sheng}
}

\newif\ifshowrevisions
\showrevisionsfalse
\ifshowrevisions\else
  \colorlet{green}{black}
\fi

\issueinfo{}{}{}{}
\PII{}
\copyrightinfo{}{}
\date{}

\theoremstyle{definition}
\newtheorem{definition}{Definition}[section]
\newtheorem{assumption}{Assumption}
\newtheorem{example}[definition]{Example}
\theoremstyle{plain}
\newtheorem{lemma}[definition]{Lemma}

\newtheorem{corollary}[definition]{Corollary}
\newtheorem{theorem}[definition]{Theorem}
\newtheorem{proposition}[definition]{Proposition}

\theoremstyle{remark}
\newtheorem{remark}[definition]{Remark}
\newcommand{\cH}{\mathcal{H}}
\newcommand{\ud}{\mathrm d}

\newcommand{\ttau}{\underline{t}}
\newcommand{\R}{\mathbb{R}}
\newcommand{\E}{\mathbb{E}}

\numberwithin{equation}{section}
\allowdisplaybreaks[4]
\begin{document}
\title[Weak error expansions for stochastic wave equations]{Quantitative asymptotic error analysis
and Richardson extrapolation for exponential Euler
approximations of stochastic wave equations}
\author[J. Cui]{Jianbo Cui}
\address{Department of Applied Mathematics, The Hong Kong Polytechnic
University, Hung Hom, Kowloon, Hong Kong}
\email{jianbo.cui@polyu.edu.hk}
\author[D. Jin]{Diancong Jin}
\address{Hubei Key Laboratory of Engineering Modeling and Scientific Computing, Huazhong University of Science and
Technology, Wuhan 430074, China}
\email{jindc@hust.edu.cn}
\author[H. Liu]{Huagui Liu}
\address{Department of Applied Mathematics, The Hong Kong Polytechnic
University, Hung Hom, Kowloon, Hong Kong}
\email{huaguliu@polyu.edu.hk}
\author[D. Sheng]{Derui Sheng}
\address{Department of Applied Mathematics, The Hong Kong Polytechnic
University, Hung Hom, Kowloon, Hong Kong}
\email{sdr@lsec.cc.ac.cn (Corresponding author)}
\thanks{This work is supported by MOST National Key R\&D Program No. 2024FA1015900, the Hong Kong Research Grant Council GRF grant 15302823, NSFC/RGC Joint Research Scheme N$\_$PolyU5141/24, NSFC grant 12301526, internal funds (P0041274, P0045336) from Hong Kong Polytechnic University, and the
CAS AMSS-PolyU Joint Laboratory of Applied Mathematics.
}
\begin{abstract}
We establish quantitative weak convergence for the joint terminal law
of the solution to a stochastic wave equation and the normalized error
of its exponential Euler approximation.
For multiplicative trace-class noise and spatial regularity
$\alpha\in(1,2]$, the normalized error converges to its limiting
distribution at rate $O(m^{-(\alpha-1)})$,
where $m$ is the number of time steps.
This estimate identifies the leading weak bias and yields a first-order
weak-error expansion with remainder $O(m^{-\alpha})$.
Moreover, by canceling the leading weak-error term, we construct
a Richardson extrapolation scheme with weak convergence order $\alpha$
for $\alpha\in(1,2]$ and show that this rate is sharp.
To overcome the lack of parabolic smoothing and control the remainders
amplified by normalization, we combine dimension-uniform regularity
estimates for the joint Kolmogorov equation with a conditional weak
expansion and backward transport of coefficient directions.
\end{abstract}
\maketitle
\section{Introduction}
Stochastic wave equations provide a fundamental mathematical framework for describing wave propagation in media subject to random perturbations, where the stochastic terms can account for unpredictable environmental conditions, incomplete system information, and uncertainties caused by measurement noise \cite{BPTB02}.  Such equations have been used in a variety of applications, including the modeling of the primary current density vector field in the grey matter of the human brain \cite{Galka08}, the propagation of shock waves inside the sun \cite{DKMNX06}, and semiconductor laser linewidth theory \cite{Wenzel21}. 
Within the framework of stochastic evolution equations, the stochastic wave equation under consideration takes the following abstract form:
\begin{equation}\label{eq:SWE-intro}
\ud X_t=(AX_t+F(X_t))\ud t+\Sigma(X_t)\ud W_t,\quad t\in(0,T],
\end{equation}
{subject to the initial condition $X_0=x_0$,}
where $T>0$, $X_t=(u(t),\partial_tu(t))^\top$, and
$$
A=
\begin{pmatrix}
0&I\\
-\Lambda&0
\end{pmatrix},
\quad
F(X_t)=
\begin{pmatrix}
0\\
f(u(t))
\end{pmatrix},
\quad
\Sigma(X_t)=
\begin{pmatrix}
0\\
\sigma(u(t))
\end{pmatrix}.
$$
Here, $-\Lambda$ denotes the Dirichlet Laplacian on the
one-dimensional spatial domain $\mathcal O=[0,1]$ with homogeneous
Dirichlet boundary conditions, $f,\sigma:\mathbb R\to\mathbb R$, and
$W$ is a $Q$-Wiener process on $H=L^2(\mathcal O)$.

Since explicit solutions of \eqref{eq:SWE-intro} are generally unavailable,
the development of reliable numerical schemes and the associated error
analysis are essential for both theoretical understanding and practical
computation (see, e.g., \cite{KLS10,CQ16,CHS26Book}).
Existing studies include higher-order approximations
(see, e.g., \cite{FPP24}), structure-preserving schemes (see, e.g., \cite{CHS25}), and so on.
Related work also addresses weak approximation \cite{HM19,BHS18}
and the numerical approximation of invariant measures \cite{LBG25}.
A widely used time discretization for \eqref{eq:SWE-intro} is the
exponential Euler method, also known as the stochastic trigonometric
method \cite{CLS13,WX15,ACLW16}.
For $m\in\mathbb N_+$, setting $\tau=T/m$ and $t_n=n\tau$,
the exponential Euler approximation $\{X_{t_n}^m\}_{n=0}^m$ for \eqref{eq:SWE-intro} is defined recursively by
\begin{align}\label{eq:EEM-intro}
X^m_{t_{n+1}}
=
S_{\tau}(X^m_{t_n}
+\tau F(X^m_{t_n})
+\Sigma(X^m_{t_n})\Delta W_n),
\quad n=0,1,\ldots,m-1,
\end{align}
with the initial value $X_{0}^m={x_0}$,
where $\Delta W_n=W_{t_{n+1}}-W_{t_n}$ and $\{S_t=e^{tA}\}_{t\in\R}$ is the group generated by $A$. 
By integrating the linear wave dynamics exactly, the exponential
Euler method \eqref{eq:EEM-intro} captures the underlying linear
oscillations and avoids the CFL-type time-step restrictions
associated with spatial discretization \cite{CQ16}.
Strong convergence rates for this method have been established
in \cite{ACLW16,WX15}, while temporal weak convergence rates
have been analyzed in \cite{CJL24}.
In particular,
under Assumption~\ref{asp:F} and for deterministic initial data
$x_0\in\cH^1=H^1\times H$, the method \eqref{eq:EEM-intro} achieves first-order strong convergence
(see Lemma~\ref{lem:limit-basic}), namely,
\begin{equation}\label{eq:SO}
\E\left[\|X_T-X_T^m\|_{\cH}^p\right]
\le C(T,p,x_0)m^{-p},\quad p\ge2,
\end{equation}
where $\cH=H\times H^{-1}$. 
This estimate quantifies the strong accuracy of the numerical
approximation and provides a theoretical basis for the design
and complexity analysis of Monte Carlo and multilevel Monte Carlo
simulations. Nevertheless, \eqref{eq:SO} neither determines whether the first-order
rate is sharp nor describes the leading random fluctuations at that scale. 
To capture finer information, we study the asymptotic error distribution of \eqref{eq:EEM-intro},
namely, the limiting distribution of the normalized error
$$Y_T^m:=m(X_T^m-X_T)\quad \text{as }m\to\infty.$$

The asymptotic error distribution of stochastic numerical methods plays a role analogous to a central limit theorem for the numerical error. Fruitful results have been obtained in this direction for various stochastic systems. For stochastic ordinary differential equations, \cite{JP98} proved the convergence in distribution of the normalized error of the Euler--Maruyama method under Lipschitz continuous coefficients and identified the corresponding limiting error process. This result was later extended in \cite{PQM2020} to the case of locally Lipschitz continuous coefficients. Related asymptotic error distribution results have also been established for other classes of stochastic systems. See, e.g., \cite{FU2023,NS2023} for stochastic Volterra equations, \cite{HLN2016,NN2007,ZHL2023,UK25} for equations driven by fractional noise, and \cite{LWW23} for McKean--Vlasov stochastic differential equations. To the best of our knowledge, closely related results in the infinite-dimensional setting include \cite{HJWY24} and the preprint \cite{HJW25}, which investigated exponential integrators for stochastic heat equations driven by additive and multiplicative noises, respectively.
These works focus on establishing stable convergence of the
normalized error process and identifying its limit.
Beyond these qualitative results, we seek to quantify the rate
of convergence to the asymptotic error distribution.
Such quantitative estimates are crucial for controlling the remainder
in the weak-error expansion and establishing the convergence order
of the extrapolated approximation.

Establishing a weak convergence rate for the joint distribution
of $X_T$ and $Y_T^m$ presents several challenges.
First, Jacod's theory on the convergence in distribution of
conditionally Gaussian martingales \cite{Jacod97} is a classical and useful
tool for deriving asymptotic error distributions.
However, it only establishes convergence to the asymptotic error distribution and does not provide a quantitative convergence rate.
Second, multiplying the numerical error $X^m-X$ by $m$ amplifies
the local remainders arising from time discretization.
Controlling these amplified remainders is particularly delicate
for the stochastic wave equation, since the operator $A$ generates
a strongly continuous group rather than an analytic semigroup.
Thus, no parabolic smoothing is available to compensate for the low
regularity of the noise and the unboundedness of $A$.
Finally, the displacement and velocity components $u(t)$ and $\partial_t u(t)$ in  the stochastic wave equation have different spatial regularities,
requiring mixed estimates on the product phase space.

Our main result, Theorem \ref{tho:mr}, establishes the following
quantitative estimate for the joint terminal distribution: 
\begin{equation}\label{eq:mr-intro}
\left|\E[\varphi(X_T,Y_T^m)]
-\E[\varphi(X_T,Y_T)]\right|
\le C m^{-(\alpha-1)}
\end{equation}
for a suitable
class of smooth test functions $\varphi$,
where $\alpha\in(1,2]$ is the spatial regularity parameter specified by Assumption \ref{asp:alp}. 
The limiting error process $Y$ in \eqref{eq:mr-intro} is governed by the following linear equation:\begin{equation}\label{eq:ASE}
\ud Y_t
=
\{AY_t+\Phi_F(X_t,Y_t)\}\ud t
+\Phi_\Sigma(X_t,Y_t)\ud W_t
+\frac{T}{\sqrt{12}}\phi_\Sigma(X_t)\ud \widetilde W_t,
\quad t\in(0,T],
\end{equation}
with the initial value $Y_0=0$, where $\widetilde W=\{\widetilde W_t\}_{t\in[0,T]}$ is an independent copy of $W$. Here, for $G\in\{F,\Sigma\}$, the coefficients are defined by
  \begin{equation}\label{PhiG}
\phi_G(z)=AG(z)-DG(z)Az,\quad
\Phi_G(z,y)=DG(z)y+\frac{T}{2}\phi_G(z),~
z\in\cH^1, y\in\cH.
\end{equation}
The additional Wiener process $\widetilde W$ in \eqref{eq:ASE} originates from the centered fluctuation of the diffusion remainder within each time step. 

To prove Theorem~\ref{tho:mr}, we employ a Kolmogorov equation
approach \cite{TT90,HM19}, motivated by the weak error analysis
of numerical approximations for SPDEs. Using the Markov property of
the limiting joint process $(X,Y)$ and It\^o's formula, we first derive
a weak error representation for $(X_T,Y_T^m)$ that decomposes the error
into a remainder contribution $\mathcal E_{\rm rem}$ and a centered
contribution $\mathcal E_{\rm cent}$; see \eqref{eq:two-blocks}. For the
numerical solution $X_t^m$ itself, the weak error
representation involves derivatives up to order two of the solution to
the Kolmogorov equation associated with $X$. For the normalized
error $Y_t^m$, however, the discretization errors amplified by
normalization require a further weak expansion, which involves mixed
derivatives up to total order four of the solution to the
Kolmogorov equation \eqref{eq:Kolu} associated with $(X,Y)$. In Proposition~\ref{prop:varphiY}, we establish the
required mixed derivative estimates for the corresponding
finite-dimensional equation, with constants independent of the spectral
Galerkin dimension. For the remainder
contribution $\mathcal E_{\rm rem}$, the key is to exploit conditional
cancellation in the stochastic remainder terms, since direct moment
estimates yield only the rate $m^{-\min\{\alpha-1,\frac12\}}$, which is
insufficient for \eqref{eq:mr-intro} when $\alpha\in(3/2,2]$; see Remark \ref{rem:Erem} and
Proposition~\ref{prop:Erem}. To estimate the centered contribution
$\mathcal E_{\rm cent}$, we exploit the cancellation property \eqref{eq:centered-identities} of the
within-step weight $\vartheta_m(t)$ and its quadratic counterpart.
This reduces the control of $\mathcal E_{\rm cent}$ to weak
temporal increment estimates for coefficient-weighted derivatives of
the solution to the Kolmogorov equation \eqref{eq:Kolu}; see
Proposition~\ref{prop:one-step}. A direct treatment of these increments
would introduce an additional application of the unbounded operator
$A$ to coefficient directions with insufficient spatial regularity
(see Remark~\ref{rem:Aterm}). We avoid this loss of regularity by
transporting these directions backward along the wave group, so that
the operator contributions generated by the transport cancel the
corresponding terms arising from differentiation of the 
Kolmogorov equation. This makes it possible to establish the required
increment estimates without imposing $H^2$ spatial regularity on the
solution.

Beyond the quantitative convergence result in Theorem~\ref{tho:mr}, our main contributions are threefold:
\begin{enumerate}[label=\textup{(\arabic*)},leftmargin=*,itemsep=0.3ex,topsep=0.3ex]
\item Using the quantitative convergence estimate~\eqref{eq:mr-intro}, we derive a first-order weak-error expansion for the exponential Euler method~\eqref{eq:EEM-intro}. More precisely, for every $\psi\in\mathcal C_b^4(\cH,\mathbb R)$, we obtain
\begin{equation}\label{eq:psiXm}
	\E[\psi(X_T^m)]-\E[\psi(X_T)]
	=
	\frac1m\E[D\psi(X_T)Y_T]+O(m^{-\alpha}).
\end{equation}
In particular, we identify the leading-order coefficient of the weak error explicitly in terms of the limiting error process $Y_T$.
\item  Motivated by the temporal lower-bound question raised in
\cite{CJL24}, we give an admissible example in the present trace-class
setting for which
\[
\lim_{m\to\infty}m\bigl|\E[\psi(X_T^m)]-\E[\psi(X_T)]\bigr|>0.
\]
Together with the known upper bound \eqref{eq:SO}, this reveals that the
first-order weak rate cannot in general be improved for this method;
the Lipschitz continuity of $\psi$ also yields sharpness of the strong
rate; see Remark~\ref{rem:opt}.
\item Using the weak-error expansion \eqref{eq:psiXm}, we construct an
$\alpha$th-order weak approximation for \eqref{eq:SWE-intro} by Richardson extrapolation, namely, 
\[
2\E[\psi(X_T^{2m})]-\E[\psi(X_T^m)]
=\E[\psi(X_T)]+O(m^{-\alpha}).
\]
This quantitative rate cannot be inferred from the stable convergence
of $Y_T^m$ with respect to the $\sigma$-algebra generated by $X_T$ (see Definition \ref{def:stable}),
or equivalently from the joint weak convergence of $(X_T,Y_T^m)$.
Indeed, stable convergence alone yields only an $o(m^{-1})$
extrapolation error and does not specify any higher algebraic
convergence rate; see Remark~\ref{rem:stable-Richardson}.
Moreover, for every $\alpha\in(1,2]$, we construct examples showing
that the order $\alpha$ of the Richardson extrapolation is sharp;
see Proposition~\ref{prop:convergence-rates-sharpness} for $\alpha\in(1,2)$ and Remark \ref{rem:Rich} for $\alpha=2$.
This highlights an infinite-dimensional regularity barrier:
whereas second-order weak extrapolation is natural in finite-dimensional problems, the attainable order here is
capped by the spatial regularity $\alpha$ of the solutions.
\end{enumerate}

The rest of this paper is organized as follows. Section \ref{S:preliminaries} introduces the functional setting and assumptions, and states the main result, its application to Richardson extrapolation, and the sharpness of the convergence rates. In Section~\ref{S:mainresult}, we reduce the infinite-dimensional problem
to its spectral Galerkin approximations.
The proof of Theorem~\ref{tho:mr} is then reduced to a weak error estimate that is uniform in
the Galerkin dimension. Section \ref{S:proof} proves this dimension-uniform weak error estimate. Appendix \ref{S:auxiliary-proofs} contains the auxiliary regularity estimates for the Kolmogorov equation, while Appendix \ref{app:sharpness} gives the proof of Proposition \ref{prop:convergence-rates-sharpness}.
\section{Preliminaries and main results}
\label{S:preliminaries}
Let $\Lambda=-\Delta$ be the negative Laplacian operator endowed with
Dirichlet boundary conditions on $\mathcal O:=[0,1]$.  For $r\ge0$, we
define
\[
H^r:=\operatorname{Dom}(\Lambda^{r/2}),\qquad
\|z\|_{H^r}:=\|\Lambda^{r/2}z\|_H,
\]
and, for $r>0$, let $H^{-r}$ be the dual space of $H^r$ with respect to
the pivot space $H:=H^0=L^2(\mathcal O)$. For $z=(z_1,z_2)\in\cH^r:=H^r\times H^{r-1}$ with $r\in\mathbb R$, we use the product norm
\[
\|z\|_{\cH^r}^2=\|z_1\|_{H^r}^2+\|z_2\|_{H^{r-1}}^2.
\]
When viewed as an operator on $\cH:=\cH^0=H\times H^{-1}$, the wave operator is
$Az=(z_2,-\Lambda z_1)$ with domain
$\operatorname{Dom}(A)=\cH^1=H^1\times H$.
Given Hilbert spaces
$V_1,\ldots,V_n$, the space
$\mathscr L(V_1\times\cdots\times V_n,\R)$ denotes the space of bounded
$n$-linear forms on $\prod_{i=1}^nV_i$, endowed with the operator norm.
Furthermore, $\mathscr L_2(V_1,V_2)$ denotes the space of
Hilbert--Schmidt operators $A_1:V_1\to V_2$, endowed with the Hilbert--Schmidt norm
$
\|A_1\|_{\mathscr L_2(V_1,V_2)}
:=\left(\sum_{k=1}^\infty\|A_1v_k\|_{V_2}^2\right)^{1/2},
$
where $\{v_k\}_{k\ge1}$ is any complete orthonormal basis of $V_1$. Hereafter, for $h=(h_1,h_2)\in\cH$, let  $[h]_1:=h_1$ and $[h]_2:=h_2$ denote its first and second components, respectively.

 Recall that the group  $\{S_t=e^{tA}\}_{t\in\R}$ generated by $A$ is given by
\begin{equation}\label{eq:St}
S_t = \begin{pmatrix} 
\cos(\Lambda^{1/2}t) & \Lambda^{-1/2}\sin(\Lambda^{1/2}t) \\ 
-\Lambda^{1/2}\sin(\Lambda^{1/2}t) & \cos(\Lambda^{1/2}t) 
\end{pmatrix}.
\end{equation}
We have the following fundamental properties (see, e.g., \cite{ACLW16}):
\begin{gather}\label{eq:Abound}
\|Az\|_{\cH^{r-1}}
= \|z\|_{\cH^{r}},\quad
\|S_t z\|_{\cH^r}=\|z\|_{\cH^r},\quad z = (z_1, z_2)^\top \in \cH^{r},\quad t\in\R.
\end{gather}
For any $\beta \in [0,1]$, there exists $C:=C(\beta,T)>0$ such that
for any $t\in[0,T]$,
\begin{gather}\label{eq:sinop}
\max\{\|\Lambda^{-\beta/2}\sin(t\Lambda^{1/2})\|_{\mathscr{L}(H)},
\|(\cos(t\Lambda^{1/2}) - I)\Lambda^{-\beta/2}\|_{\mathscr{L}(H)} \}
\le C t^{\beta}.
\end{gather}
Consequently, for every $r\in\R$ and $\beta\in[0,1]$,
\begin{align}\label{eq:St-I}
\|(S_t-I)z\|_{\cH^r}
\le C t^\beta\|z\|_{\cH^{r+\beta}},
\qquad z\in\cH^{r+\beta},\quad t\in[0,T].
\end{align}

Let $Q:H\to H$ be a nonnegative, self-adjoint, trace-class
operator.  The driving process $W$ in \eqref{eq:SWE-intro} is a
$Q$-Wiener process with the Karhunen--Lo\`eve expansion
$W_t = \sum_{i=1}^\infty Q^{\frac{1}{2}}e_i \beta_i(t)$, where
$\{e_i\}_{i \in \mathbb{N}}$ is a complete orthonormal eigenbasis of $Q$ in
$H=L^2(\mathcal{O})$, and $\{\beta_i\}_{i \in \mathbb{N}}$ are
independent standard one-dimensional Brownian motions.
We denote by $H_Q:=Q^{1/2}(H)$ the reproducing kernel Hilbert space
associated with $W$, endowed with the inner product
$
\langle u,v\rangle_{H_Q}
:=\langle Q^{-1/2}u,Q^{-1/2}v\rangle_H$ for $u,v\in H_Q$,
where $Q^{-1/2}$ is the Moore--Penrose pseudoinverse of
$Q^{1/2}$.
For simplicity, we assume that the initial value
${x_0}\in\cH^1$ is deterministic.
Given two functions $f,\sigma:\R\to\R$, we formally define
the nonlinearities $F$ and $\Sigma$ in \eqref{eq:SWE-intro} by
\begin{align*}
[F(z)](\zeta)
=
\begin{pmatrix}
0\\
f(z_1(\zeta))
\end{pmatrix},
\qquad
[\Sigma(z)q](\zeta)
=
\begin{pmatrix}
0\\
\sigma(z_1(\zeta))q(\zeta)
\end{pmatrix},
\end{align*}
for $z=(z_1,z_2)\in\cH$, $\zeta\in\mathcal O$, and $q\in H_Q$. We impose the following assumption to ensure that $F$ and $\Sigma$ are well defined as mappings
from $\cH$ into $\cH^1$ and
$\mathscr L_2(H_Q,\cH^1)$, respectively.
\begin{assumption}\label{asp:F}
(1)
$f:\R\to\R$ and $\sigma:\R\to\R$ are globally Lipschitz continuous.\\
(2)
The noise covariance satisfies
\(
\sum_{i=1}^{\infty}
\|Q^{1/2}e_i\|_{L^\infty(\mathcal O)}^2<\infty.
\)
\end{assumption}
In fact,
under Assumption~\ref{asp:F}, the mappings
$F:\cH\to\cH^1$ and
$\Sigma:\cH\to\mathscr L_2(H_Q,\cH^1)$
are well defined and globally Lipschitz continuous.
More precisely, there exists $C>0$ such that, for all
$z,\xi\in\cH$,
\begin{align}\label{eq:Lip}
\|F(z)-F(\xi)\|_{\cH^1}
+\|\Sigma(z)-\Sigma(\xi)\|_{\mathscr L_2(H_Q,\cH^1)}
\le C\|z_1-\xi_1\|_H
\le C\|z-\xi\|_{\cH}.
\end{align}
In particular, there exists $C>0$ such that
\begin{align}\label{eq:lin}
\|F(z)\|_{\cH^1}
+\|\Sigma(z)\|_{\mathscr L_2(H_Q,\cH^1)}
\le C(1+\|z\|_{\cH}),
\qquad z\in\cH.
\end{align}
As a consequence of \eqref{eq:Lip} and \eqref{eq:lin}, Assumption~\ref{asp:F} guarantees the existence and
uniqueness of a mild solution to \eqref{eq:SWE-intro};
see, e.g., \cite{ACLW16,WX15}.

Recall that $\tau=T/m$, with $m\in\mathbb N_+$, denotes
the uniform time step size. A continuous version of the exponential Euler method \eqref{eq:EEM-intro} reads
\begin{align}\label{eq:EEM}
\ud X^m_t=(AX^m_t+S_{t-\ttau}F(X^m_{\ttau}))\ud t+S_{t-\ttau}\Sigma(X^m_{\ttau})\ud W_t, \quad t\in(0,T]
\end{align}
with the initial value $X_0^m={x_0}$,
where $\ttau=\max\{t_i:t_i\le t\}$.
It is proved in, e.g., \cite{WX15} that, under Assumption
\ref{asp:F} and for ${x_0}\in\cH^1$, the exponential Euler method
\eqref{eq:EEM} is first-order strongly convergent.  Consequently, the
normalized error $Y^m_t:=m(X^m_t-X_t)$ is bounded in
$L^p(\Omega,\cH)$, uniformly in $t\in[0,T]$, as shown in the following lemma.
\begin{lemma}\label{lem:limit-basic}
Let Assumption \ref{asp:F} hold and let
${x_0}\in\cH^1$ be deterministic.
Then, for every $p\ge2$, there exists
$C>0$, depending on $p,T$ and the coefficients $f$ and $\sigma$
in Assumption
\ref{asp:F}, but independent of $m$ and ${x_0}$, such that, for all
$s,t\in[0,T]$,
\begin{gather}
\E[\|X\|_{L^\infty(0,T;\cH^1)}^p]
+\E[\|X^m\|_{L^\infty(0,T;\cH^1)}^p]
\le C(1+\|{x_0}\|_{\cH^1}^p),
\label{eq:LpX}\\
\E[\|u(t)-u(s)\|_H^p]
\le C(1+\|{x_0}\|_{\cH^1}^p)|t-s|^p,
\label{Holderq}\\
\E[\|Y^m\|_{L^\infty(0,T;\cH)}^{p}]
\le C(1+\|{x_0}\|_{\cH^1}^p),
\label{Strongorder}
\end{gather}
where $u(t)=[X_t]_1$ is the displacement component of $X_t$.
\end{lemma}
\begin{proof}
The Lipschitz and linear-growth estimates \eqref{eq:Lip}--\eqref{eq:lin}
show that the coefficient conditions in
\cite[Assumption 1.1]{WX15} hold with the regularity parameter
$\delta=1$.  Moreover, the continuous-time extension used in
\cite{WX15} coincides with the interpolation \eqref{eq:EEM}.
Consequently, the uniform moment bound \eqref{eq:LpX},
the temporal estimate \eqref{Holderq}, and the strong-error estimate
\eqref{Strongorder} follow, respectively, from
\cite[Proposition 3.1, Lemma 3.3, and Theorem 3.1]{WX15}. In these estimates, keeping track of the initial value through the
linear-growth bound \eqref{eq:lin} yields the factor
$1+\|x_0\|_{\cH^1}^p$ in
\eqref{eq:LpX}--\eqref{Strongorder}.
\end{proof}
In this paper, we investigate the limiting distribution of the normalized error $Y_T^m$ as $m\to\infty$.
The rate of convergence of $Y_T^m$ will be characterised by the
spatial regularity of the coefficients and the noise
specified in the following assumption.
\begin{assumption}\label{asp:alp}
For some $\alpha\in(1,2]$, there exists a constant
\(C>0\) such that
	\begin{align*}
		{\|\Lambda^{\frac{\alpha-1}{2}}f(u(\cdot))\|_{H}^2 + \sum_{k=1}^\infty\|\Lambda^{\frac{\alpha-1}{2}}\sigma(u(\cdot))Q^{\frac12}e_k\|_H^2 \le C\bigl(1 + \|\Lambda^{\frac{\alpha}{2}}u\|_{H}^2\bigr),\quad u\in H^{\alpha}.}
	\end{align*}
\end{assumption}
\begin{remark}
(1) In the special case where the covariance operator $Q$ is diagonal
in the Dirichlet eigenbasis,
\begin{equation}\label{eq:Qeta}
\eta_i(\xi)=\sqrt2\sin(i\pi\xi),\qquad
Q\eta_i=q_i\eta_i,\qquad q_i\ge0,
\end{equation}
choosing $e_i=\eta_i$ gives
$
\sum_{i=1}^\infty
\|Q^{1/2}e_i\|_{L^\infty(\mathcal O)}^2
=
2\sum_{i=1}^\infty q_i
=
2\operatorname{Tr}Q.
$
Thus, within this class of covariance operators,
Assumption~\ref{asp:F}(2) is equivalent to the trace-class condition
$\operatorname{Tr}Q<\infty$.\\
(2)
When $\alpha=2$, Assumption~\ref{asp:alp} also imposes boundary
compatibility. More precisely, it requires
\[
f(u)\in H_0^1(0,1),\qquad
\sigma(u)Q^{1/2}e_k\in H_0^1(0,1),
\]
for every $u\in H^2(0,1)\cap H_0^1(0,1)$,
where
\(
H_0^1(0,1)
=
\{u\in H^1(0,1):u(0)=u(1)=0\}.
\)
Taking $u\equiv0$ gives $f(0)=0$. If $\sigma(0)\neq0$,
Assumption~\ref{asp:alp} further requires
\(
\sum_{k=1}^{\infty}
\|Q^{1/2}e_k\|_{H^1}^2<\infty.
\)
If, in addition, $Q$ is diagonal in the Dirichlet eigenbasis as in
\eqref{eq:Qeta}, this condition reduces to
\(
\sum_{k=1}^{\infty}\lambda_k q_k<\infty.
\)
\end{remark}
In view of Assumption \ref{asp:alp},
$F:\cH^\alpha\to\cH^\alpha$ and
$\Sigma:\cH^\alpha\to\mathscr{L}_2(H_Q,\cH^\alpha)$ grow at most
linearly. More precisely, for any $z\in \cH^\alpha$,
\begin{align}\label{eq:Fz}
\|F(z)\|_{\cH^{\alpha}}+\|\Sigma(z)\|_{\mathscr{L}_2(H_Q,\cH^\alpha)}\le C(1+\|z\|_{\cH^{\alpha}}).
\end{align}
\begin{lemma}
Let Assumptions~\ref{asp:F}--\ref{asp:alp} hold and let
$x_0\in\cH^\alpha$ be deterministic. Then, for every $p\ge2$,
there exists $C>0$, independent of $m$, such that
\begin{equation}\label{eq:Xalp}
\|X\|_{L^\infty(0,T;L^p(\Omega,\cH^\alpha))}
+\|X^m\|_{L^\infty(0,T;L^p(\Omega,\cH^\alpha))}
\le C(1+\|x_0\|_{\cH^\alpha}).
\end{equation}
\end{lemma}
\begin{proof}
Fix \(p\ge2\). Define \(Z_t^{[0]}=S_tx_0\) and, for \(n\ge0\),
\[
Z_t^{[n+1]}
=
S_tx_0+\int_0^tS_{t-s}F(Z_s^{[n]})\,\ud s
+\int_0^tS_{t-s}\Sigma(Z_s^{[n]})\,\ud W_s,\quad t\in[0,T].
\]
The isometry of \(S_t\), \eqref{eq:Fz}, H\"older's inequality,
and the Burkholder--Davis--Gundy inequality yield, by iteration,
\[
\sup_{n\ge0}\sup_{t\in[0,T]}
\E\|Z_t^{[n]}\|_{\cH^\alpha}^p
\le C_{p,T}\bigl(1+\|x_0\|_{\cH^\alpha}^p\bigr).
\]
Since \(Z_t^{[n]}\to X_t\) in \(L^p(\Omega;\cH^1)\) by
\eqref{eq:Lip}, weak compactness of unit balls in \(L^p(\Omega;\cH^\alpha)\)
and weak lower semicontinuity give the required bound \eqref{eq:Xalp} for \(X_t\).
The corresponding estimate for the numerical solution \(X^m\)
is proved similarly, and we omit the details.
\end{proof}

The following assumption strengthens 
Assumption \ref{asp:F}(1).
\begin{assumption}\label{asp:fg3}
The functions $f$ and $\sigma$ belong to
$\mathcal C^4(\R,\R)$, and their derivatives of orders one through
four are bounded. 
\end{assumption}

Denote by \(\mathscr T^4\) the space of
functions \(\varphi:\cH^1\times\cH\to\R\) whose continuous
derivatives \(D_x^aD_y^b\varphi\) exist for all \(a,b\ge0\) satisfying
\(a+b\le4\) and \(a\le3\), and for which there exist \(q\ge0\) and
\(C_\varphi>0\) such that for every such pair \((a,b)\),
\[
\|D_x^aD_y^b\varphi(x,y)\|_{\mathscr L((\cH^1)^{\otimes a}\times\cH^{\otimes b},\R)}
\le C_\varphi(1+\|x\|_{\cH^1}+\|y\|_{\cH})^q.
\]
Under Assumptions~\ref{asp:F} and \ref{asp:fg3},
equation~\eqref{eq:ASE} is affine in $Y$, with bounded linear
coefficients $DF(X_t)$ and $D\Sigma(X_t)$. Moreover, according to Lemma \ref{lem:limit-basic}, for every
$p\ge2$,
\[
\sup_{t\in[0,T]}
\E\left[
\|\phi_F(X_t)\|_{\cH}^p
+
\|\phi_\Sigma(X_t)\|_{\mathscr L_2(H_Q,\cH)}^p
\right]
<\infty.
\]
Hence, by the standard existence and uniqueness theory for stochastic
evolution equations with Lipschitz coefficients (see, e.g.,
\cite[Section~7.1, Theorem~7.14]{DPZ14}), \eqref{eq:ASE} admits a unique adapted mild solution
\(
Y\in L^p\bigl(\Omega;C([0,T];\cH)\bigr)\) for every
$p\ge2$.
\begin{definition}[Stable convergence]\label{def:stable}
Let $(\Omega,\mathcal F,\mathbb P)$ be a probability space,
$\mathcal G\subset\mathcal F$ a sub-$\sigma$-algebra, and $E$ a
Polish space. Let $\{Z_n\}_{n\ge1}$ be $E$-valued random variables
on $(\Omega,\mathcal F,\mathbb P)$, and let $Z$ be an $E$-valued
random variable defined on an extension
$(\widetilde\Omega,\widetilde{\mathcal F},\widetilde{\mathbb P})$
of the original probability space.
We say that $Z_n$ converges stably to $Z$ with respect to
$\mathcal G$, if, for every bounded $\mathcal G$-measurable random variable
$\xi$ and every $f\in C_b(E)$,
$$
\lim_{n\to\infty}\E[\xi f(Z_n)]
=
\widetilde\E[\xi f(Z)].
$$
Here $\widetilde\E$ denotes expectation on the extended probability
space.
\end{definition}
Recall that $\widetilde W$ denotes an independent copy of $W$
(on an extension of the probability space if necessary). \(\mathbb E\) denotes expectation with respect to the probability measure under
consideration.
Our main result is Theorem \ref{tho:mr} on a quantitative weak
error estimate with respect to test functions in \(\mathscr T^4\). Its proof is deferred to Sections~\ref{S:mainresult} and \ref{S:proof}.
  \begin{theorem}\label{tho:mr}
Let Assumptions \ref{asp:F}, \ref{asp:alp}, and \ref{asp:fg3} hold,
and let ${x_0}\in\cH^\alpha$.  Then, for every
{$\varphi\in\mathscr T^4$},
there exists $C>0$ depending on $\varphi$, but
independent of $m$, such that
\begin{align*}
{\left|\E[\varphi(X_T,Y_T^m)]
-\E[\varphi(X_T,Y_T)]\right|}
\le Cm^{-(\alpha-1)}.
\end{align*}
\end{theorem}

By Theorem \ref{tho:mr}, we can obtain the stable convergence of $Y_T^m$ to $Y_T$
with respect to the $\sigma$-algebra $\sigma(X_T)$ generated by $X_T$. To justify this we establish
the joint convergence in distribution of $X_T$ and $Y_T^m$.
Let $\mathscr{O}$ be an arbitrary open subset of $\cH^1\times\cH$.
By separability of $\cH^1\times\cH$, there exist balls $B(a_j,r_j)$, $j\ge1$,
such that
$\mathscr{O}=\cup_{j\ge1}B(a_j,r_j)$ and
$B(a_j,2r_j)\subset \mathscr{O}$.
Choose $\chi\in C^\infty([0,\infty))$ satisfying
$0\le\chi\le1$, $\chi=1$ on $[0,1]$, and
$\chi=0$ on $[4,\infty)$.
Define
\[
g_n(z)=1-\prod_{j=1}^n
\bigg(1-\chi\Big(
\frac{\|z-a_j\|_{\cH^1\times\cH}^2}{r_j^2}
\Big)\bigg),\quad n\ge1.
\]
Each $g_n$ is smooth with bounded derivatives and hence
belongs to $\mathscr T^4$. Moreover,
$0\le g_n\uparrow\mathbf1_\mathscr{O}$ pointwise.
By Theorem~\ref{tho:mr}, for each fixed $n$,
\begin{align*}
\liminf_{m\to\infty}
\mathbb P\big((X_T,Y_T^m)\in \mathscr{O}\big)
\ge \lim_{m\to\infty}
\mathbb E[g_n(X_T,Y_T^m)]=\mathbb E[g_n(X_T,Y_T)].
\end{align*}
Letting $n\to\infty$ and using monotone convergence yields
\[
\liminf_{m\to\infty}
\mathbb P\big\{(X_T,Y_T^m)\in\mathscr O\big\}
\ge
\mathbb P\big\{(X_T,Y_T)\in\mathscr O\big\}.
\]
The Portmanteau theorem implies that
$(X_T,Y_T^m)$ converges in distribution to
$(X_T,Y_T)$ in $\cH^1\times\cH$, and thus $Y_T^m$ converges stably to $Y_T$ with respect to $\sigma(X_T)$.
 Moreover, Theorem~\ref{tho:mr} yields the following
first-order weak-error expansion for the exponential Euler method.

\begin{corollary}
\label{thm:weak-expansion}

Let the assumptions of Theorem \ref{tho:mr} hold and let
\(\psi\in\mathcal C_b^4(\cH,\R)\). Then there exists a constant \(C>0\),
independent of \(m\), such that
\begin{align}
\big|\E[\psi(X_T^m)]-\E[\psi(X_T)]
-\frac{1}{m}\E[D\psi(X_T)Y_T]\big|
\le Cm^{-\alpha}.
\label{eq:weak-expansion}
\end{align}
In particular, the exponential Euler scheme~\eqref{eq:EEM} has weak order one.
\end{corollary}
\begin{proof}
Since \(\psi\in\mathcal C_b^4(\cH,\R)\), the mapping $(x,y)\mapsto D\psi(x)y$ belongs to $\mathscr T^4$.
Applying Theorem \ref{tho:mr} therefore implies
\begin{align}
\left|\E[D\psi(X_T)Y_T^m]-\E[D\psi(X_T)Y_T]\right|
\le Cm^{-(\alpha-1)}.
\label{eq:linear-test-rate}
\end{align}
Recalling that \(X_T^m=X_T+m^{-1}Y_T^m\), Taylor's formula gives
\begin{align}\label{eq:psi}\notag
\psi(X_T^m)-\psi(X_T)
={}&m^{-1}D\psi(X_T)Y_T^m\\
&+m^{-2}\int_0^1(1-\theta)
D^2\psi(X_T+\theta m^{-1}Y_T^m)
(Y_T^m,Y_T^m)\,\ud\theta.
\end{align}
By the boundedness of $D^2\psi$ and \eqref{Strongorder},
\[
\E\left|\int_0^1(1-\theta)
D^2\psi(X_T+\theta m^{-1}Y_T^m)(Y_T^m,Y_T^m)\,\ud\theta\right|
\le C\E\|Y_T^m\|_{\cH}^2
\le C.
\] 
Finally, combining this estimate with
\eqref{eq:linear-test-rate} and \(1<\alpha\le2\) proves
\eqref{eq:weak-expansion}.
\end{proof}
\begin{remark}\label{rem:opt}
	The first-order strong and weak convergence rates of the exponential
Euler method are sharp. To see this, let
$
T=\frac12, x_0=0, f\equiv0, \sigma\equiv1
$, and choose
\begin{equation}\label{eq:psidefcos}
\psi_0(x)=\cos(\langle x_1,\eta_1\rangle_H),\quad x\in\cH.
\end{equation}
Let
$
\eta_1(y):=\sqrt{2}\sin(\pi y)$, $y\in[0,1]$,
be the normalized first Dirichlet eigenfunction of $\Lambda$, and set
$Q=\eta_1\otimes\eta_1$, that is,
\[
(Qz)(y)=\langle z,\eta_1\rangle_H\eta_1(y),
\qquad z\in H.
\]
Then the $Q$-Wiener process takes the form
$W_t=\eta_1\beta_t$, where $\beta=\{\beta_t\}_{t\ge0}$ is a
one-dimensional Brownian motion.
In this setting,
$
\E[D\psi_0(X_T)Y_T]
=-\E[V\sin(U)],
$
where the random variables $U$ and $V$ are
\[
U:=\langle[X_T]_1,\eta_1\rangle_H,
\qquad
V:=\langle[Y_T]_1,\eta_1\rangle_H.
\]
Projecting the mild forms of $X$ and $Y$ onto the first Dirichlet
mode and using \eqref{eq:St}, we obtain
\begin{align*}
U
&=\frac1\pi\int_0^{1/2}
\sin\left(\pi\left(\frac12-s\right)\right)\,\ud\beta_s,\\
V
&=\frac14\int_0^{1/2}
\cos\left(\pi\left(\frac12-s\right)\right)\,\ud\beta_s
+\frac1{2\sqrt{12}}\int_0^{1/2}
\cos\left(\pi\left(\frac12-s\right)\right)\,\ud\widetilde\beta_s,
\end{align*}
where $\widetilde\beta$ is independent of $\beta$. Hence, by It\^o's
isometry,
$
\operatorname{Var}(U)=\frac1{4\pi^2},
$ and
$\operatorname{Cov}(U,V)=\frac1{8\pi^2}.
$
Since $(U,V)$ is centered jointly Gaussian,
the integration by parts yields
\[
\E[V\sin(U)]
=
\operatorname{Cov}(U,V)\E[\cos(U)]
=
\frac1{8\pi^2}e^{-1/(8\pi^2)}
\ne0.
\]
It therefore follows from \eqref{eq:weak-expansion} that
\[
\lim_{m\to\infty}
m\left|
\E[\psi_0(X_T^m)]-\E[\psi_0(X_T)]
\right|
=
\frac1{8\pi^2}e^{-1/(8\pi^2)}>0.
\]
In particular, there exist $c_1>0$ and $m_0\in\mathbb N_+$ such that
\begin{equation}\label{eq:psi0}
\bigl|\E[\psi_0(X_T^m)]-\E[\psi_0(X_T)]\bigr|
\ge
c_1m^{-1},
\qquad m\ge m_0.
\end{equation}
Notice that the function $\psi_0:\cH\to\R$ is bounded by $1$ and
globally Lipschitz continuous with $\rm{Lip}(\psi_0)\le1$.
Hence, by \eqref{eq:psi0}, for every $m\ge m_0$,
\begin{align*}
&\mathcal W_1(\mathbb P\circ(X_T^m)^{-1},
\mathbb P\circ(X_T)^{-1})=
\sup_{\psi:\cH\to\R,\ \rm{Lip}(\psi)\le1}
\bigl|\E[\psi(X_T^m)]-\E[\psi(X_T)]\bigr|
\ge c_1m^{-1},
\end{align*}
where $\mathcal W_1$ denotes the
Wasserstein-$1$ distance.
On the other hand, since $\rm{Lip}(\psi_0)\le1$,
$$
\E\|X_T^m-X_T\|_{\cH}
\ge
\bigl|\E[\psi_0(X_T^m)]-\E[\psi_0(X_T)]\bigr|
\ge c_1m^{-1},
\quad m\ge m_0.
$$
Together with the first-order strong error estimate, these bounds show
that both the first-order strong convergence rate and the
Wasserstein-$1$ convergence rate are sharp.\end{remark}
Richardson extrapolation is an acceleration procedure based on
an asymptotic error expansion. For weak approximations of stochastic
differential equations, see, e.g., Talay and Tubaro~\cite{TT90}.
More precisely, suppose that
approximations \(\mathcal{Q}_m\) of a quantity \(\mathcal{Q}\)
satisfy
\[
\mathcal{Q}_m=\mathcal{Q}+c_1m^{-r}+o(m^{-r})
\]
for some $r>0$.
Then the extrapolated approximation
\[
\mathcal{Q}_m^{\mathrm R}:=\frac{2^r\mathcal{Q}_{2m}-\mathcal{Q}_m}{2^r-1}
\]
eliminates the leading error term \(c_1m^{-r}\)
and reduces the approximation error to
$o(m^{-r})$.
In the present setting, we set
\(\mathcal{Q}_m=\E[\psi(X_T^m)]\) and \(r=1\).
\begin{corollary}[Richardson extrapolation]
\label{cor:Richardson}
Let the assumptions of Corollary~\ref{thm:weak-expansion} hold and let
\(\psi\in\mathcal C_b^4(\cH,\R)\).  Then
\begin{align}
\left|2\E[\psi(X_T^{2m})]-\E[\psi(X_T^m)]
-\E[\psi(X_T)]\right|
&\le Cm^{-\alpha}.
\label{eq:Richardson}
\end{align}
\end{corollary}
\begin{proof}
Applying \eqref{eq:weak-expansion} with \(m\) and \(2m\) cancels the
term \(m^{-1}\E[D\psi(X_T)Y_T]\) and proves \eqref{eq:Richardson}.
\end{proof}
\begin{remark}\label{rem:stable-Richardson}
Let $\psi\in \mathcal{C}_b^4(\cH,\R)$. By \eqref{Strongorder} and the
boundedness of $D\psi$,
\[
\sup_m\E|D\psi(X_T)Y_T^m|^2
\le C\sup_m\E\|Y_T^m\|_{\cH}^2<\infty,
\]
so the family $\{D\psi(X_T)Y_T^m\}_{m\in\mathbb{N}_+}$ is uniformly
integrable. If one only knew that \(Y_T^m\) converges stably to \(Y_T\)
with respect to the $\sigma$-algebra generated by $X_T$, then,
together with this uniform integrability, one could only conclude that
$$
\lim_{m\to\infty}
\E[D\psi(X_T)Y_T^m]
=\E[D\psi(X_T)Y_T],
$$
without obtaining a convergence rate. This is weaker than
\eqref{eq:linear-test-rate}. Combining this qualitative convergence
with Taylor's formula \eqref{eq:psi} yields
\[
\lim_{m\to\infty}
m\left(
2\E[\psi(X_T^{2m})]
-\E[\psi(X_T^m)]
-\E[\psi(X_T)]
\right)=0.
\]
Hence, the Richardson extrapolation error is \(o(m^{-1})\), but this
argument does not provide any strictly higher algebraic convergence
rate. In contrast, \eqref{eq:linear-test-rate} upgrades this qualitative
$o(m^{-1})$ estimate to the explicit order $m^{-\alpha}$
in Corollary~\ref{cor:Richardson}.
\end{remark}
To test the sharpness of the algebraic rates for
$\alpha\in(1,2)$, we consider the following example.
\begin{example}\label{ex:sharpness}
Fix $\alpha\in(1,2)$. Let
$
T=1,\ x_0=0,\ f\equiv0$, and $\sigma\equiv1$.
Let
$
\eta_k(y):=\sqrt2\sin(k\pi y)$ for $k\in\mathbb N_+$.
Define the covariance operator $Q:H\to H$ by
$
Qz:=\langle z,g\rangle_Hg$, $z\in H$,
where
\begin{equation}\label{eq:gdef}
g
:=
\eta_1+
\sum_{j=1}^{\infty}
\frac{(2\pi\kappa_j)^{1-\alpha}}{j^2}
\eta_{2\kappa_j},
\quad \text{where }\kappa_j:=2^{j^2-1}.
\end{equation}
Next, define
$
\rho:=\sum_{j=1}^{\infty}\frac1j\eta_{2\kappa_j}\in H
$
and the test functions
\begin{align}
\varphi(x,y)
&:=
\langle x_1,\eta_1\rangle_H
\langle\Lambda^{-1/2}y_2,\rho\rangle_H,
\label{eq:varphi}\\
\psi(x)
&:=
\cos\big(
\langle x_1,\eta_1\rangle_H
+
\langle\Lambda^{-1/2}x_2,\rho\rangle_H
\big).
\label{eq:psidef}
\end{align}
\end{example}
Proposition \ref{prop:convergence-rates-sharpness} and Remark \ref{rem:Rich} establish the sharpness of the algebraic
exponents in Theorem~\ref{tho:mr} and Corollary~\ref{cor:Richardson}
through the example above. More precisely, it shows that the exponent
$\alpha-1$ in Theorem~\ref{tho:mr} and the exponent $\alpha$ in
Corollary~\ref{cor:Richardson} cannot, in general, be improved.
\begin{proposition}
\label{prop:convergence-rates-sharpness}
Fix $\alpha\in(1,2)$ and consider the setting of
Example~\ref{ex:sharpness}.
Then Assumptions~\ref{asp:F}, \ref{asp:alp}, and \ref{asp:fg3}
are satisfied, and
\begin{equation}\label{eq:phipsi}
\varphi\in\mathscr T^4,
\qquad
\psi\in\mathcal C_b^4(\cH,\R).
\end{equation}
Moreover, for every $\varepsilon>0$,
\begin{gather}
\limsup_{m\to\infty}
m^{\alpha-1+\varepsilon}
\left|
\E[\varphi(X_1,Y_1^m)]
-
\E[\varphi(X_1,Y_1)]
\right|
=\infty,
\label{eq:shp1}\\
\limsup_{m\to\infty}
m^{\alpha+\varepsilon}
\left|
2\E[\psi(X_1^{2m})]
-
\E[\psi(X_1^m)]
-
\E[\psi(X_1)]
\right|
=\infty.
\label{eq:shp2}
\end{gather}
\end{proposition}
The proof of Proposition~\ref{prop:convergence-rates-sharpness} is given in Appendix~\ref{app:sharpness}.
\begin{remark}\label{rem:Rich}
At the endpoint $\alpha=2$, the second-order Richardson rate remains
sharp. Indeed, consider the example in Remark~\ref{rem:opt} and set
$c=1/(8\pi^2)$. Let
\(
U_m:=\langle [X_{1/2}^m]_1,\eta_1\rangle_H.
\)
A direct computation gives
\[
\operatorname{Var}(U_m)
=
\frac{1}{2m\pi^2}
\sum_{k=1}^m
\sin^2\left(\frac{k\pi}{2m}\right)
=
\frac{1}{4\pi^2}\left(1+\frac1m\right).
\]
Since $U_m$ is centered Gaussian and
$\psi_0(x)=\cos(\langle x_1,\eta_1\rangle_H)$, we have
\[
\E[\psi_0(X_{1/2}^m)]
=
\exp\left(-\frac12\operatorname{Var}(U_m)\right)
=
e^{-c}e^{-c/m}.
\]
Together with
$\E[\psi_0(X_{1/2})]=e^{-c}$ from Remark~\ref{rem:opt}, this gives
\begin{align*}
2\E[\psi_0(X_{1/2}^{2m})]
-\E[\psi_0(X_{1/2}^m)]
-\E[\psi_0(X_{1/2})]
=
-e^{-c}(1-e^{-c/(2m)})^2.
\end{align*}
Multiplying by $m^2$ and letting $m\to\infty$ yields
\[
\lim_{m\to\infty}
m^2\left(
2\E[\psi_0(X_{1/2}^{2m})]
-\E[\psi_0(X_{1/2}^m)]
-\E[\psi_0(X_{1/2})]
\right)
=
-\frac{e^{-c}c^2}{4}\ne0.
\]
Hence, the exponent $2$ in Corollary~\ref{cor:Richardson} cannot, in
general, be improved.
\end{remark}
\section{Reduction to a finite-dimensional estimate}
\label{S:mainresult}
In this section, we first introduce the Galerkin approximations associated
with the stochastic processes \(X\), \(X^m\), \(Y^m\), and \(Y\), and collect the estimates that
are uniform in the projection parameter \(N\). Their terminal
convergence as \(N\to\infty\) then reduces the original assertion (i.e.,  Theorem~\ref{tho:mr}) to
the finite-dimensional weak error estimate stated at the end of this
section.

Let $\{(\lambda_k,\eta_k)\}_{k\ge1}$ be the eigenpairs of $\Lambda$
and define
\begin{align}
P_Nz:=\sum_{k=1}^N\langle z,\eta_k\rangle_H\eta_k,\qquad
\mathcal P_N:=\begin{pmatrix}P_N&0\\0&P_N\end{pmatrix},\qquad
\cH_N:=\mathcal P_N\cH
\label{eq:Galerkin-projections}
\end{align}
{The projection \(P_N\) defined in
\eqref{eq:Galerkin-projections}} extends spectrally to every
$H^r$, $r\in\R$.
Thus $\mathcal P_N$ is a contraction on every $\cH^r$ and
$\mathcal P_Nz\to z$ in $\cH^r$ as $N\to\infty$.  When
$\mathcal P_N$ is applied to an operator
$B\in\mathscr L_2(H_Q,\cH^r)$, it acts on the range of $B$, namely,
$
(\mathcal P_NB)q:=\mathcal P_N(Bq)$ for $q\in H_Q.$
The contractivity of $\mathcal P_N:\cH^r\to\cH^r$ implies that $\mathcal{P}_N:\mathscr L_2(H_Q,\cH^r)\to \mathscr L_2(H_Q,\cH^r)$ is also a contraction. 
Put
\begin{align}
A^{(N)}&:=\mathcal P_NA|_{\cH_N},&
S_t^{(N)}&:=e^{tA^{(N)}}=\mathcal P_NS_t|_{\cH_N},\notag\\
F^{(N)}(z)&:=\mathcal P_NF(z),&
\Sigma^{(N)}(z)&:=\mathcal P_N\Sigma(z),\qquad z\in\cH_N.
\label{eq:Galerkin-coefficients}
\end{align}
Furthermore, $S_t^{(N)}=\mathcal P_NS_t|_{\cH_N}$ is an isometry on
$\cH_N$ with respect to the ambient $\cH^r$-norm, and the shifted
estimate \eqref{eq:St-I} holds with
$S_t^{(N)}$ because
\begin{equation}\label{eq:St-IN}
\|(S_t^{(N)}-I)z\|_{\cH^r}
=\|\mathcal P_N(S_t-I)z\|_{\cH^r}
\le C t^\beta\|z\|_{\cH^{r+\beta}},
\qquad z\in\cH_N.
\end{equation}

In the following lemma, we collect the dimension-uniform estimates
for the projected coefficients. In the sequel, an empty product, as in
\eqref{eq:DjG0} when $j=2$, is understood to be one.
\begin{lemma}\label{lem:FSigN}
(i) {Under Assumption \ref{asp:F},} for any $z,\xi\in\cH_N$,
\begin{align}
&\|F^{(N)}(z)-F^{(N)}(\xi)\|_{\cH^1}
+\|\Sigma^{(N)}(z)-\Sigma^{(N)}(\xi)\|_{\mathscr L_2(H_Q,\cH^1)}
\le C\|z-\xi\|_{\cH},
\label{eq:Galerkin-Lip}\\
&\|F^{(N)}(z)\|_{\cH^1}
+\|\Sigma^{(N)}(z)\|_{\mathscr L_2(H_Q,\cH^1)}
\le C(1+\|z\|_{\cH}).
\label{eq:Galerkin-lin}
\end{align}
(ii) {Under Assumptions \ref{asp:F} and \ref{asp:alp}, for any $z\in\cH_N$,}
\begin{align}
\|F^{(N)}(z)\|_{\cH^\alpha}
+\|\Sigma^{(N)}(z)\|_{\mathscr L_2(H_Q,\cH^\alpha)}
\le C(1+\|z\|_{\cH^\alpha}).
\label{eq:Galerkin-Fz}
\end{align}
(iii) Under Assumptions \ref{asp:F} and \ref{asp:fg3}, for $z,h^1\in\cH_N$,
\begin{align}
&\|DF^{(N)}(z)h^1\|_{\cH^s}
+\|D\Sigma^{(N)}(z)h^1\|_{\mathscr L_2(H_Q,\cH^s)}
\le C\|h_1^1\|_H,\qquad s=0,1,
\label{eq:DG1}
\end{align}
and, for $j\in\{2,3,4\}$ and $z,h^1,\ldots,h^j\in\cH_N$,
\begin{equation}
\begin{aligned}
&\|D^jF^{(N)}(z)(h^1,\ldots,h^j)\|_{\cH}
+\|D^j\Sigma^{(N)}(z)(h^1,\ldots,h^j)\|_
 {\mathscr L_2(H_Q,\cH)}\\
&\quad\le C\prod_{\mu=1}^2\|h_1^\mu\|_H
\prod_{\ell=3}^j\|h_1^\ell\|_{H^1}.
\label{eq:DjG0}
\end{aligned}
\end{equation}
Moreover, \eqref{eq:DjG0} remains valid after any permutation
of the direction arguments. 
\end{lemma}
\begin{proof}
Recall that $\mathcal P_N$ is contractive on $\cH^r$ and
on $\mathscr L_2(H_Q,\cH^r)$.
The estimates \eqref{eq:Galerkin-Lip} and
\eqref{eq:Galerkin-lin} follow from \eqref{eq:Lip} and
\eqref{eq:lin}, respectively.  Under Assumption \ref{asp:alp},
\eqref{eq:Galerkin-Fz} follows from \eqref{eq:Fz}.

It remains to prove the derivative estimates  \eqref{eq:DG1}  and \eqref{eq:DjG0} under
Assumption~\ref{asp:fg3}.
For \(G\in\{F,\Sigma\}\), \(j\in\{1,2,3,4\}\), and
\(z,h^1,\ldots,h^j\in\cH_N\), the projected derivative satisfies
\begin{equation}\label{eq:projected-DjG}
D^jG^{(N)}(z)(h^1,\ldots,h^j)
=\mathcal P_ND^jG(z)(h^1,\ldots,h^j).
\end{equation}
Notice that, for $j=1,2,3,4$ and $h^1,\ldots,h^j\in\cH_N$,
\begin{equation}\label{eq:projected-Nemytskii-derivatives}
\begin{aligned}
D^jF(z)(h^1,\ldots,h^j)
&=\begin{pmatrix}0\\
f^{(j)}(z_1)h_1^1\cdots h_1^j\end{pmatrix},\\
D^j\Sigma(z)(h^1,\ldots,h^j)Q^{1/2}e_i
&=\begin{pmatrix}0\\
\sigma^{(j)}(z_1)h_1^1\cdots h_1^jQ^{1/2}e_i\end{pmatrix}.
\end{aligned}
\end{equation}
For $j=1$, the boundedness of $f'$ and $\sigma'$, and Assumption
\ref{asp:F}(2) {imply}
\begin{equation*}
\begin{aligned}
&\|DF(z)h^1\|_{\cH^1}^2+\|D\Sigma(z)h^1\|_{\mathscr{L}_2(H_Q,\cH^1)}^2\\
&\quad\le \|f'(z_1)h_1^1\|_{H}^2+\sum_{i=1}^\infty\|\sigma'(z_1)h_1^1Q^{1/2}e_i\|_{H}^2\le C\|h_1^1\|_H^2,
\end{aligned}
\end{equation*}
which, along with $\cH^1\hookrightarrow\cH$ and the contractivity of $\mathcal{P}_N$, implies \eqref{eq:DG1}.

Fix $j\in\{2,3,4\}$.  By the embedding
$H^1(\mathcal O)\hookrightarrow L^\infty(\mathcal O)$ and H\"older's
inequality,
\begin{align}\label{eq:projected-product-bound}
\Big\|\prod_{\ell=1}^j h_1^\ell\Big\|_{L^1}
&\le \|h_1^1\|_H\|h_1^2\|_H
\prod_{\ell=3}^j\|h_1^\ell\|_{L^\infty}
\le C\|h_1^1\|_H\|h_1^2\|_H
\prod_{\ell=3}^j\|h_1^\ell\|_{H^1}.
\end{align} 
Using \eqref{eq:projected-Nemytskii-derivatives}, $\|\Lambda^{-1/2}w\|_{L^\infty}\le C\|w\|_H$, the duality between
$H^{-1}$ and $H^1$, and the boundedness of $f^{(j)}$, we obtain
\begin{align*}
&\left\|D^jF(z)(h^1,\ldots,h^j)\right\|_{\cH}\\
&\quad=\sup_{\|w\|_H\le1}
\Big|\Big\langle f^{(j)}(z_1)\prod_{\ell=1}^jh_1^\ell,
\Lambda^{-1/2}w\Big\rangle_H\Big|\le C\|h_1^1\|_H\|h_1^2\|_H
\prod_{\ell=3}^j\|h_1^\ell\|_{H^1}.
\end{align*}
Similarly, the boundedness of $\sigma^{(j)}$,
Assumption \ref{asp:F}(2), and
\eqref{eq:projected-product-bound} imply
\begin{align*}
&\|D^j\Sigma(z)(h^1,\ldots,h^j)\|_
 {\mathscr L_2(H_Q,\cH)}^2=\sum_{i=1}^\infty
\Big\|\sigma^{(j)}(z_1)\prod_{\ell=1}^jh_1^\ell
Q^{1/2}e_i\Big\|_{H^{-1}}^2\\
&\le C\Big\|\prod_{\ell=1}^jh_1^\ell\Big\|_{L^1}^2
\sum_{i=1}^\infty\|Q^{1/2}e_i\|_{L^\infty}^2\le C\|h_1^1\|_H^2\|h_1^2\|_H^2
\prod_{\ell=3}^j\|h_1^\ell\|_{H^1}^2.
\end{align*}
Combining these two estimates with the contractivity of
$\mathcal P_N$, and \eqref{eq:projected-DjG} proves
\eqref{eq:DjG0}.
The formula
\eqref{eq:projected-Nemytskii-derivatives} also shows that the
multilinear derivatives are symmetric in their direction arguments, namely, for every permutation \(\pi\) of
\(\{1,\ldots,j\}\),
\[
D^jG^{(N)}(z)(h^1,\ldots,h^j)
=D^jG^{(N)}(z)(h^{\pi(1)},\ldots,h^{\pi(j)}),\qquad G\in\{F,\Sigma\}.
\]
Hence, \eqref{eq:DjG0} remains valid after any permutation
of the direction arguments. 
\end{proof}
\begin{remark}\label{rem:DFS}
For each $z\in\cH^1$, the dimension-uniform estimates in
Lemma~\ref{lem:FSigN}(iii) and the density of
$\cup_{N\ge1}\cH_N$ in both $\cH$ and $\cH^1$
yield bounded linear extensions
$DF(z):\cH\to\cH^1$ and
$D\Sigma(z):\cH\to\mathscr L_2(H_Q,\cH^1)$.
For $2\le j\le4$, the same reasoning gives continuous multilinear
extensions of $D^jF(z)$ and $D^j\Sigma(z)$ to directions
$h^1,h^2\in\cH$ and $h^3,\ldots,h^j\in\cH^1$,
with values in $\cH$ and $\mathscr L_2(H_Q,\cH)$, respectively. 
In particular, \eqref{eq:DG1} and \eqref{eq:DjG0} remain valid
with $F$ and $\Sigma$ in place of $F^{(N)}$ and $\Sigma^{(N)}$,
respectively.
\end{remark}
For $G\in\{F,\Sigma\}$, $z\in\cH_N$, and
$\xi\in\cH_N$, define
\begin{align}
\phi_G^{(N)}(z)&:=A^{(N)}G^{(N)}(z)-DG^{(N)}(z)A^{(N)}z,\notag\\
\Phi_G^{(N)}(z,\xi)&:=DG^{(N)}(z)\xi+\frac T2\phi_G^{(N)}(z).
\label{eq:Galerkin-Phi}
\end{align}
Since $S_t^{(N)}\mathcal P_N=\mathcal P_NS_t$ and
$A^{(N)}\mathcal P_N=\mathcal P_NA$ on the projected domain $\cH_N$, it follows from
\eqref{eq:projected-DjG} that, for $G\in\{F,\Sigma\}$,
\begin{align}\label{eq:projected-Phi-identities}
\phi_G^{(N)}(z)=\mathcal P_N\phi_G(z),\qquad
\Phi_G^{(N)}(z,\xi)=\mathcal P_N\Phi_G(z,\xi),
\qquad z,\xi\in\cH_N.
\end{align}
We next introduce an auxiliary process
\((\mathbb X^{(N)},\mathbb Y^{m,(N)},\mathbb Y^{(N)})\) on
\([0,T]\):
\begin{align}
\ud\mathbb X_t^{(N)}
&=\big(A^{(N)}\mathbb X_t^{(N)}+F^{(N)}(\mathbb X_t^{(N)})\big)\ud t
+\Sigma^{(N)}(\mathbb X_t^{(N)})\ud W_t,
\label{eq:Galerkin-systemX}\\
\ud\mathbb Y_t^{m,(N)}
&=\Big[A^{(N)}\mathbb Y_t^{m,(N)}
+m\big\{S_{t-\ttau}^{(N)}F^{(N)}(
m^{-1}\mathbb Y_{\ttau}^{m,(N)}+\mathbb X_{\ttau}^{(N)})
-F^{(N)}(\mathbb X_t^{(N)})\big\}\Big]\ud t\notag\\
&\quad
+m\big\{S_{t-\ttau}^{(N)}\Sigma^{(N)}(
m^{-1}\mathbb Y_{\ttau}^{m,(N)}+\mathbb X_{\ttau}^{(N)})
-\Sigma^{(N)}(\mathbb X_t^{(N)})\big\}\ud W_t,
\label{eq:Galerkin-systemYm}\\
\ud\mathbb Y_t^{(N)}
&=\big(A^{(N)}\mathbb Y_t^{(N)}
+\Phi_F^{(N)}(\mathbb X_t^{(N)},\mathbb Y_t^{(N)})\big)\ud t\notag\\
&\quad+\Phi_{\Sigma}^{(N)}(\mathbb X_t^{(N)},\mathbb Y_t^{(N)})
\ud W_t
+\frac T{\sqrt{12}}\phi_{\Sigma}^{(N)}(\mathbb X_t^{(N)})
\ud\widetilde W_t,\quad t\in[0,T]
\label{eq:Galerkin-systemY}
\end{align}
with the initial values
\[
\mathbb X_0^{(N)}=x\in\cH_N,\qquad
\mathbb Y_0^{m,(N)}=\mathbb Y_0^{(N)}=y\in\cH_N.
\]
For each fixed $N$, the drift and diffusion coefficients in
\eqref{eq:Galerkin-systemY} depend affinely on
$\mathbb Y^{(N)}$. Their linear parts are bounded, and their
inhomogeneous parts have at most linear growth in
$\mathbb X^{(N)}$, which ensures the existence of a
unique solution to \eqref{eq:Galerkin-systemY}. 
The parenthesized superscript $(N)$ records the dependence
on the Galerkin parameter \(N\), while the superscript \(m\) records the
dependence on the temporal discretization parameter \(m\). For brevity, we omit the
explicit dependence on \(x\) and \(y\) in the notation of
\((\mathbb X^{(N)},\mathbb Y^{m,(N)},\mathbb Y^{(N)})\).
Throughout this section, generic constants $C$ are independent
of $N$ and $m$, with the dependence on the initial values
$x$ and $y$ made explicit in the estimates.
For \(x=\mathcal P_N{x_0}\) and \(y=0\), we denote by
\((X^{(N)},Y^{m,(N)},Y^{(N)})\) the corresponding solution to
\eqref{eq:Galerkin-systemX}--\eqref{eq:Galerkin-systemY}.
Accordingly, the vector
$(X^{(N)},X^{m,(N)},Y^{m,(N)},Y^{(N)})$ is exactly the spectral
Galerkin approximation of $(X,X^m,Y^m,Y)$, where $X_t^{m,(N)}:=m^{-1} Y_t^{m,(N)}+ X_t^{(N)}$.
\begin{lemma}\label{lem:W}
Let Assumptions \ref{asp:F} and \ref{asp:alp} hold and let
\({x_0}\in\cH^\alpha\).  Then for every $p\ge2$, there exists \(C>0\), independent of \(m\)
and \(N\), such that
\begin{gather}\label{eq:LpX-Galerkin}
\E\|X^{(N)}\|_{L^\infty(0,T;\cH^1)}^p
+\E\|X^{m,(N)}\|_{L^\infty(0,T;\cH^1)}^p
\le C(1+\|{x_0}\|_{\cH^1}^p),\\\label{Holderq-Galerkin}
\E\|[X_t^{(N)}-X_s^{(N)}]_1\|_H^p
\le C(t-s)^p,\qquad 0\le s\le t\le T,\\\label{Strongorder-Galerkin}
\E\|Y^{m,(N)}\|_{L^\infty(0,T;\cH)}^p
\le C,\\\label{eq:Xalp-Galerkin}
\|X^{(N)}\|_{L^\infty(0,T;L^p(\Omega,\cH^\alpha))}
+
\|X^{m,(N)}\|_{L^\infty(0,T;L^p(\Omega,\cH^\alpha))}
\le C(1+\|{x_0}\|_{\cH^\alpha}).
\end{gather}
\end{lemma}
\begin{proof}
The contractivity of $\mathcal P_N$, Lemma \ref{lem:FSigN},
\eqref{eq:St-IN}, and \eqref{eq:projected-Phi-identities} show that
all coefficient and group estimates used in the proof of Lemma
\ref{lem:limit-basic} and in the derivation of \eqref{eq:Xalp}
remain valid for the spectral Galerkin approximation processes, with
the same constants uniformly in $m$ and $N$. {Consequently,
\eqref{eq:LpX-Galerkin}--\eqref{eq:Xalp-Galerkin} hold.}
\end{proof}
We will frequently use the following estimate.
Fix \(p\ge2\).
Let $b$ and $B$ be predictable processes with values in
$\cH_N$ and $\mathscr L_2(H_Q,\cH_N)$, respectively.  Since
$S_t^{(N)}$ is an isometry in every $\cH^r$-norm,
H\"older's inequality and the
Burkholder--Davis--Gundy inequality imply
\begin{align}\label{eq:convolution-bounds}
&{\E\sup_{0\le u\le t}
\left\|\int_0^uS_{u-v}^{(N)}b_v\,\ud v\right\|_{\cH^r}^p+
\E\sup_{0\le u\le t}
\left\|\int_0^uS_{u-v}^{(N)}B_v\,\ud W_v\right\|_{\cH^r}^p}\notag\\
&\qquad{\le C\int_0^t\E\left[
\|b_v\|_{\cH^r}^p+
\|B_v\|_{\mathscr L_2(H_Q,\cH^r)}^p\right]\ud v.}
\end{align}
In fact,
for the stochastic term in \eqref{eq:convolution-bounds}, one may write
$S_{u-v}^{(N)}=S_u^{(N)}S_{-v}^{(N)}$ and apply the martingale maximal
inequality. 
The same argument applies to stochastic integrals with respect to
$\widetilde W$.
\begin{lemma}
\label{lem:Galerkin-convergence}
Let Assumptions \ref{asp:F}, \ref{asp:alp}, and \ref{asp:fg3} hold,
let \({x_0}\in\cH^\alpha\), and fix \(m\in\mathbb N_+\). Then
\[
\|Y_T^{(N)}-Y_T\|_{L^2(\Omega,\cH)}
+\|Y_T^{m,(N)}-Y_T^m\|_{L^2(\Omega,\cH)}
\rightarrow0
\qquad\text{as }N\to\infty.
\]
\end{lemma}
\begin{proof}
{For \(G\in\{F,\Sigma\}\), differentiating
\eqref{PhiG} yields}
\begin{gather*}
D\phi_G(z)q
=ADG(z)q-D^2G(z)(q,Az)-DG(z)Aq,
\qquad z,q\in\cH^1,\\
D_z\Phi_G(z,\xi)q
=D^2G(z)(q,\xi)+\frac T2D\phi_G(z)q,\\
D_\xi\Phi_G(z,\xi)k=DG(z)k,
\qquad \xi,k\in\cH.
\end{gather*}
Recall that \eqref{eq:DG1} and \eqref{eq:DjG0} remain valid
with $F$ and $\Sigma$ in place of $F^{(N)}$ and $\Sigma^{(N)}$,
respectively; see Remark \ref{rem:DFS}.
Together with \eqref{eq:Abound}, this implies
\begin{equation}
\begin{aligned}
&\|D\phi_F(z)q\|_{\cH}
+\|D\phi_\Sigma(z)q\|_{\mathscr L_2(H_Q,\cH)}
\le C(1+\|z\|_{\cH^1})\|q\|_{\cH^1},\\
&\|D_z\Phi_F(z,\xi)q\|_{\cH}
+\|D_z\Phi_\Sigma(z,\xi)q\|_{\mathscr L_2(H_Q,\cH)}
\le C(1+\|z\|_{\cH^1}+\|\xi\|_{\cH})\|q\|_{\cH^1},\\
&\|D_\xi\Phi_F(z,\xi)k\|_{\cH}
+\|D_\xi\Phi_\Sigma(z,\xi)k\|_{\mathscr L_2(H_Q,\cH)}
\le C\|k\|_{\cH}.
\end{aligned}
\label{eq:phi-Phi-derivative-bounds}
\end{equation}
The mean-value formula therefore yields, for
\(z,\bar z\in\cH^1\) and \(\xi,\bar\xi\in\cH\),
\begin{equation}
\begin{aligned}\label{eq:Phi-limit-Lipschitz}
&\|\phi_F(z)-\phi_F(\bar z)\|_{\cH}
+\|\phi_\Sigma(z)-\phi_\Sigma(\bar z)\|_
{\mathscr L_2(H_Q,\cH)}\\
&\qquad\le
C(1+\|z\|_{\cH^1}+\|\bar z\|_{\cH^1})
\|z-\bar z\|_{\cH^1},\\
&\|\Phi_F(z,\xi)-\Phi_F(\bar z,\bar\xi)\|_{\cH}
+\|\Phi_\Sigma(z,\xi)-\Phi_\Sigma(\bar z,\bar\xi)\|_
{\mathscr L_2(H_Q,\cH)}\\
&\qquad\le C\|\xi-\bar\xi\|_{\cH}
+C(1+\|z\|_{\cH^1}+\|\bar z\|_{\cH^1}+\|\bar\xi\|_{\cH})
\|z-\bar z\|_{\cH^1}.
\end{aligned}
\end{equation}
By \eqref{eq:projected-Phi-identities} and the contractivity of
\(\mathcal P_N\), the same estimates hold for
\(\phi_G^{(N)}\) and \(\Phi_G^{(N)}\) when all their arguments belong
to \(\cH_N\).
Let \(e_Y^{(N)}(t):=Y_t^{(N)}-\mathcal P_NY_t\). Subtracting the
projected mild equation for \(Y\) from that for \(Y^{(N)}\), we use
\eqref{eq:projected-Phi-identities} in the form
\[
\begin{aligned}
\phi_G^{(N)}(X_s^{(N)})-\mathcal P_N\phi_G(X_s)
&=\mathcal P_N\bigl(\phi_G(X_s^{(N)})-\phi_G(X_s)\bigr),\\
\Phi_G^{(N)}(X_s^{(N)},Y_s^{(N)})
-\mathcal P_N\Phi_G(X_s,Y_s)
&=\mathcal P_N\bigl(
\Phi_G(X_s^{(N)},Y_s^{(N)})-\Phi_G(X_s,Y_s)\bigr).
\end{aligned}
\]
Thus the unprojected estimates
\eqref{eq:phi-Phi-derivative-bounds} and
\eqref{eq:Phi-limit-Lipschitz}, followed by the contractivity of
{\(\mathcal P_N\), imply}
\begin{align}\label{eq:eYN}
\E\|e_Y^{(N)}(t)\|_{\cH}^2
&\le C\int_0^t\E\|e_Y^{(N)}(s)\|_{\cH}^2\ud s
+C\int_0^t\E\|(I-\mathcal P_N)Y_s\|_{\cH}^2\ud s\\\notag
&\quad+C\int_0^t
\E\!\left[
K_N(s)^2\|X_s^{(N)}-X_s\|_{\cH^1}^2\right]\ud s,
\end{align}
where
\(
K_N(s):=1+\|X_s^{(N)}\|_{\cH^1}
+\|X_s\|_{\cH^1}+\|Y_s\|_{\cH}.
\)

{
The spectral representation of \(\mathcal P_N\) and
\eqref{eq:Xalp} imply
\[
\sup_{t\in[0,T]}
\|(I-\mathcal P_N)X_t\|_{L^2(\Omega,\cH^1)}
\le
\lambda_{N+1}^{-\frac{\alpha-1}{2}}
\sup_{t\in[0,T]}\|X_t\|_{L^2(\Omega,\cH^\alpha)}
\rightarrow0
\qquad\text{as }N\to\infty.
\]
To justify the perturbation estimate directly in $\cH^1$, put
$e_X^{(N)}(t):=X_t^{(N)}-\mathcal P_NX_t$. Subtracting the projected
mild equations and using \eqref{eq:Lip},
\eqref{eq:convolution-bounds} with $r=1$ and $p=2$, and
$\cH^1\hookrightarrow\cH$, we obtain
\[
\E\|e_X^{(N)}(t)\|_{\cH^1}^2
\le C\int_0^t\left(
\E\|e_X^{(N)}(s)\|_{\cH^1}^2
+\E\|(I-\mathcal P_N)X_s\|_{\cH^1}^2
\right)\ud s,
\]
with the initial value $e_X^{(N)}(0)=0$. 
Applying Gronwall's inequality and then using the
triangle inequality gives a constant
independent of $N$ such that
\[
\sup_{t\in[0,T]}
\|X_t^{(N)}-X_t\|_{L^2(\Omega,\cH^1)}
\le C\sup_{t\in[0,T]}
\|(I-\mathcal P_N)X_t\|_{L^2(\Omega,\cH^1)}\rightarrow0
\qquad\text{as }N\to\infty.
\]
}

	Moreover, the estimates \eqref{eq:LpX} and
	\eqref{eq:LpX-Galerkin} with  \(p=8\)  imply $X_t^{(N)}-X_t$ is uniformly bounded in $L^8(\Omega,\cH^1)$ with respect to $t\in[0,T]$ and $N\ge1$.	Therefore, interpolation between \(L^2(\Omega)\) and
	\(L^8(\Omega)\) yields
	\begin{equation}\label{eq:X-Galerkin-standard}
	\begin{aligned}
	\sup_{t\in[0,T]}
	\|X_t^{(N)}-X_t\|_{L^4(\Omega,\cH^1)}
	\rightarrow0
		\qquad\text{as }N\to\infty.
		\end{aligned}
		\end{equation}
For \(G\in\{F,\Sigma\}\),
\eqref{eq:Galerkin-Phi}, \eqref{eq:Abound}, \eqref{eq:lin}, and
{\eqref{eq:DG1} imply, uniformly in \(N\),}
\begin{equation}\label{eq:Phi-growth}
\begin{aligned}
&\|\phi_F^{(N)}(z)\|_{\cH}
+\|\phi_\Sigma^{(N)}(z)\|_{\mathscr L_2(H_Q,\cH)}
\le C(1+\|z\|_{\cH^1}),\\
&\|\Phi_F^{(N)}(z,\xi)\|_{\cH}
+\|\Phi_\Sigma^{(N)}(z,\xi)\|_{\mathscr L_2(H_Q,\cH)}
\le C(1+\|z\|_{\cH^1}+\|\xi\|_{\cH}),
\end{aligned}
\end{equation}
where $z,\xi\in\cH_N$. The mild form of \eqref{eq:Galerkin-systemY}, the
isometry of \(S_t^{(N)}\), Hölder's inequality, and the
Burkholder--Davis--Gundy inequality imply, for every \(p\ge2\) and
\(t\in[0,T]\),
\[
\E\sup_{0\le r\le t}\|Y_r^{(N)}\|_{\cH}^p
\le C\int_0^t\left(
1+\E\|X_s^{(N)}\|_{\cH^1}^p
+\E\sup_{0\le u\le s}\|Y_u^{(N)}\|_{\cH}^p
\right)\ud s.
\]
Using \eqref{eq:LpX-Galerkin} and then Gronwall's inequality, we
obtain
\begin{equation}\label{eq:LpY-Galerkin}
\E\|Y^{(N)}\|_{L^\infty(0,T;\cH)}^p
\le C(1+\|{x_0}\|_{\cH^1}^p).
\end{equation}
The linear equation \eqref{eq:ASE} is treated in the same way, and for every $p\ge2$,
\begin{equation}\label{eq:LpY}
\E\|Y\|_{L^\infty(0,T;\cH)}^p
\le C(1+\|{x_0}\|_{\cH^1}^p).
\end{equation}
The \(p=4\) cases of \eqref{eq:LpX},
\eqref{eq:LpX-Galerkin}, \eqref{eq:LpY-Galerkin}, and
\eqref{eq:LpY} yield a constant \(C>0\), independent of \(N\), such
that, for \(t\in[0,T]\) and \(N\ge1\),
\begin{align}\label{eq:limit-uniform-moments}
\|X_t^{(N)}\|_{L^4(\Omega,\cH^1)}
+\|Y_t^{(N)}\|_{L^4(\Omega,\cH)}
+\|X_t\|_{L^4(\Omega,\cH^1)}
+\|Y_t\|_{L^4(\Omega,\cH)}
\le C.
\end{align}
Hence,
the last two integrals on the right hand side of \eqref{eq:eYN} converge to zero by the strong convergence of
\(\mathcal P_N\), \eqref{eq:X-Galerkin-standard}, H\"older's
inequality, and \eqref{eq:limit-uniform-moments}. Applying Gronwall's
inequality to \eqref{eq:eYN} results in 
$
\|e_Y^{(N)}(T)\|_{L^2(\Omega,\cH)}\rightarrow0$
as $N\to\infty$.
By
the strong convergence of \(\mathcal P_N\) and the dominated
convergence theorem,
\[
\|Y_T^{(N)}-Y_T\|_{L^2(\Omega,\cH)}\le\|e_Y^{(N)}(T)\|_{L^2(\Omega,\cH)}+\|(I-\mathcal P_N)Y_T\|_{L^2(\Omega,\cH)} \rightarrow0
\qquad\text{as }N\to\infty.
\]

For each fixed $m$, \cite[Corollary~3.8]{CJL24} implies convergence
of the spectral Galerkin approximations of the exponential Euler
process in $L^2(\Omega,\cH)$, uniformly in time. In other words, for
every fixed \(m\),
\begin{equation}\label{eq:XtmN}
\sup_{t\in[0,T]}
\|X_t^{m,(N)}-X_t^m\|_{L^2(\Omega,\cH)}
\rightarrow0
\qquad\text{as }N\to\infty.
\end{equation}
{By the triangle inequality,
\[
\|Y_T^{m,(N)}-Y_T^m\|_{L^2(\Omega,\cH)}
\le m\left(
\|X_T^{m,(N)}-X_T^m\|_{L^2(\Omega,\cH)}
+\|X_T^{(N)}-X_T\|_{L^2(\Omega,\cH)}
\right).
\]
Therefore, \eqref{eq:XtmN} and \eqref{eq:X-Galerkin-standard} prove
the convergence of \(Y_T^{m,(N)}\) to \(Y_T^m\), as $N\to \infty$.}
\end{proof}
By Lemma \ref{lem:Galerkin-convergence},
it suffices to prove the following dimension-uniform Galerkin estimate.
\begin{theorem}
\label{tho:Galerkin-main}
Let Assumptions \ref{asp:F}, \ref{asp:alp}, and \ref{asp:fg3} hold,
and let \({x_0}\in\cH^\alpha\). Then for every
{\(\varphi\in\mathscr T^4\)},
there exists a constant \(C>0\),
independent of \(m\) and \(N\), such that
\begin{equation}\label{eq:Galerkin-main-statement}
{\big|\E[\varphi(X_T^{(N)},Y_T^{m,(N)})]
-\E[\varphi(X_T^{(N)},Y_T^{(N)})]\big|}
\le Cm^{-(\alpha-1)}.
\end{equation}
\end{theorem}
The proof of Theorem~\ref{tho:Galerkin-main} is deferred to
Section~\ref{S:proof}.
\begin{proof}[Proof of Theorem \ref{tho:mr}]
Fix \(m\in\mathbb N_+\).
Since \(\varphi\in\mathscr T^4\), there exist
\(q\ge0\) and \(C>0\) such that, for
\(x,\bar x\in\cH^1\) and \(y,\bar y\in\cH\),
\[
|\varphi(x,y)-\varphi(\bar x,\bar y)|
\le C\bigl(1+\|x\|_{\cH^1}^q+\|\bar x\|_{\cH^1}^q
+\|y\|_{\cH}^q+\|\bar y\|_{\cH}^q\bigr)
\bigl(\|x-\bar x\|_{\cH^1}+\|y-\bar y\|_{\cH}\bigr).
\]
First, we apply this inequality with
$(x,y)=(X_T^{(N)},Y_T^{m,(N)})$ and
$(\bar x,\bar y)=(X_T,Y_T^m)$. The moment estimates
\eqref{eq:LpX}, \eqref{Strongorder}, \eqref{eq:LpX-Galerkin},
and \eqref{Strongorder-Galerkin}, applied at moment order
\(\max\{2,2q\}\), bound the polynomial factor
in \(L^2(\Omega)\), uniformly in \(N\).
The Cauchy--Schwarz inequality,
\eqref{eq:X-Galerkin-standard}, and
the convergence of \(Y_T^{m,(N)}\) in
Lemma~\ref{lem:Galerkin-convergence}, yield
\begin{align*}
\lim_{N\to\infty}\E\big|\varphi(X_T^{(N)},Y_T^{m,(N)})
-\varphi(X_T,Y_T^m)\big|=0.
\end{align*}

Next, we take $(x,y)=(X_T^{(N)},Y_T^{(N)})$ and
$(\bar x,\bar y)=(X_T,Y_T)$, and use the moment estimates
\eqref{eq:LpX}, \eqref{eq:LpY}, \eqref{eq:LpX-Galerkin}, and
\eqref{eq:LpY-Galerkin}, again at moment order
\(\max\{2,2q\}\).
Another application of the Cauchy--Schwarz inequality, together with
\eqref{eq:X-Galerkin-standard} and
the convergence of \(Y_T^{(N)}\) in
Lemma~\ref{lem:Galerkin-convergence}, gives
\begin{align*}
\lim_{N\to\infty}\E\big|\varphi(X_T^{(N)},Y_T^{(N)})
-\varphi(X_T,Y_T)\big|=0.
\end{align*}
Letting \(N\to\infty\) in \eqref{eq:Galerkin-main-statement}
proves Theorem~\ref{tho:mr}, since the constant \(C\) is independent
of \(N\) and \(m\).
\end{proof}
\section{Dimension-uniform weak error analysis}\label{S:proof}

We first prove Theorem~\ref{tho:Galerkin-main} in subsection \ref{S:4.1} using a weak error
representation and estimates for the remainder and centered
contributions.
The estimates for the remainder and centered
contributions are established in
Sections~\ref{S:remainder-proofs} and \ref{S:one-step-proof},
respectively.
To simplify the notation, we suppress the
superscript \((N)\) in this section, and write
\begin{align*}
(A,S,F,\Sigma,\phi_G,\Phi_G)
&:=(A^{(N)},S^{(N)},F^{(N)},\Sigma^{(N)},
\phi_G^{(N)},\Phi_G^{(N)}),\\
(\mathbb X,\mathbb Y,X,X^m,Y^m,Y)
&:=(\mathbb X^{(N)},\mathbb Y^{(N)},X^{(N)},X^{m,(N)},Y^{m,(N)},Y^{(N)}).
\end{align*}
We note that generic constants \(C\)
throughout this section are independent of \(N\). 
\subsection{Proof of Theorem \ref{tho:Galerkin-main}}\label{S:4.1}
Under Assumptions \ref{asp:F} and \ref{asp:fg3}, for
{\(\varphi\in\mathscr T^4\)}, define
\begin{equation}\label{eq:utx}
{u(t,x,y):=\E[\varphi(\mathbb X_t,\mathbb Y_t)]},\qquad
t\in[0,T],
\end{equation}
where $(x,y)\in\cH_N\times\cH_N$ denote the initial value of $(\mathbb X_t,\mathbb Y_t)$.
For generic initial values $(x,y)\in\cH_N\times\cH_N$, the
moment estimates \eqref{eq:LpX-Galerkin} and \eqref{eq:LpY-Galerkin}
give, for every $p\ge2$,
\begin{align}\label{eq:generic-flow-moments}
\E\sup_{0\le t\le T}\bigl(
\|\mathbb X_t\|_{\cH^1}^p+\|\mathbb Y_t\|_{\cH}^p\bigr)
\le C_p(1+\|x\|_{\cH^1}+\|y\|_{\cH})^p.
\end{align}

Since $\varphi$ has polynomial growth, \eqref{eq:generic-flow-moments}
ensures that $u(t,x,y)$ is well defined.
\begin{proposition}\label{prop:varphiY}
Assume that Assumptions \ref{asp:F} and \ref{asp:fg3} hold and let
{\(\varphi\in\mathscr T^4\)}.
Then there exist \(C>0\) and \(r\ge0\), independent of \(N\), such that, for
\(t\in[0,T]\) and \(x,y\in\cH_N\),
\begin{align}\label{eq:partialab}
\|\partial_x^a\partial_y^b u(t,x,y)\|_
 {\mathscr L((\cH^1)^{\otimes a}\times\cH^{\otimes b},\R)}
\le C(1+\|x\|_{\cH^1}+\|y\|_{\cH})^r
\end{align}
for all integers \(a,b\ge0\) satisfying
\(1\le a+b\le4\) and \(a\le3\).
\end{proposition}
Since the proof of Proposition \ref{prop:varphiY} is lengthy, we
defer it to Section \ref{S:auxiliary-proofs} to keep the main argument
uninterrupted. We next introduce the backward Kolmogorov equation
associated with \((\mathbb X,\mathbb Y)\).
For Hilbert spaces $V_1,V_2$, a bounded bilinear form
$B\in\mathscr L(V_1\times V_2,\R)$, and
$\Sigma_j\in\mathscr L_2(H_Q,V_j)$, define
\begin{equation}\label{eq:trace-contraction}
\mathrm{tr}_Q(\Sigma_1^*B\Sigma_2)
:=\sum_{i=1}^\infty B(\Sigma_1Q^{1/2}e_i,\Sigma_2Q^{1/2}e_i).
\end{equation}
The series is absolutely convergent, and the Cauchy--Schwarz inequality yields
\begin{align}\label{eq:trace}
|\mathrm{tr}_Q(\Sigma_1^*B\Sigma_2)|
\le\|B\|_{\mathscr L(V_1\times V_2,\R)}
\|\Sigma_1\|_{\mathscr L_2(H_Q,V_1)}
\|\Sigma_2\|_{\mathscr L_2(H_Q,V_2)}.
\end{align}
For a sufficiently smooth function
\(\Theta:\cH_N\times\cH_N\to\mathbb R\), define
\begin{align}
\mathcal L_x\Theta(x,y)
&:=\langle Ax+F(x),\partial_x\Theta(x,y)\rangle
+\frac12\mathrm{tr}_Q\left(
\Sigma(x)^*\partial_{xx}\Theta(x,y)\Sigma(x)\right),
\label{eq:Lx}\\
\mathcal L_y\Theta(x,y)
&:=\langle Ay+\Phi_F(x,y),\partial_y\Theta(x,y)\rangle
+\frac12\mathrm{tr}_Q\left(
\Phi_\Sigma(x,y)^*\partial_{yy}\Theta(x,y)
\Phi_\Sigma(x,y)\right)\notag\\
&\quad+\frac{T^2}{24}\mathrm{tr}_Q\left(
\phi_\Sigma(x)^*\partial_{yy}\Theta(x,y)
\phi_\Sigma(x)\right)
+\mathrm{tr}_Q\left(
\Sigma(x)^*\partial_{xy}\Theta(x,y)
\Phi_\Sigma(x,y)\right).
\label{eq:Ly}
\end{align}
For each fixed \(N\), standard finite-dimensional Kolmogorov equation theory,
together with differentiability of the stochastic flow associated with
the Galerkin system (see, e.g., \cite{Kunita90,Krylov80}), shows that
the function \(u\) is a classical solution of the backward Kolmogorov
equation
\begin{equation}\label{eq:Kolu}
\left\{
\begin{aligned}
&\partial_t u(t,x,y)
=(\mathcal L_x+\mathcal L_y)u(t,x,y),
\qquad t\in(0,T],x,y\in\cH_N,\\
&{u(0,x,y)=\varphi(x,y)}.
\end{aligned}
\right.
\end{equation}
Differentiating the integral form of \eqref{eq:Kolu} under the
integral sign also shows that \(\partial_yu,\partial_{xy}u,\partial_{yy}u\) are
continuously differentiable in time.
\begin{lemma}\label{lem:L-L}
Let Assumptions \ref{asp:F}, \ref{asp:alp}, and \ref{asp:fg3} hold,
let \({x_0}\in\cH^\alpha\), and let
{\(\varphi\in\mathscr T^4\)}.
Put
\[
v(t,x,y):=u(T-t,x,y),\qquad t\in[0,T],
\]
where \(u\) is the solution of \eqref{eq:Kolu} with initial value
\(\varphi\).
Then, for any $t\in[0,T]$,
 \begin{align}\label{eq:vXY}
\E[v(t,X_t,Y_t^m)]-v(0,\mathcal P_N{x_0},0)
&=\E\int_{0}^{t}\Big[
\big\langle\mathcal R_F(s),
\partial_yv(s,X_s,Y_s^{m})\big\rangle\\\notag
&+\mathrm{tr}_Q\big(
\Sigma(X_s)^*
\partial_{xy}v(s,X_s,Y_s^{m})
\mathcal R_{\Sigma}(s)\big)+\mathcal E_v(s)
\Big]\ud s,
\end{align}
where the remainders are given by 
\begin{align}
\mathcal R_G(t)
&:=m\left(S_{t-\ttau}G(X_{\ttau}^m)
-G(X_t)\right)-\Phi_G(X_t,Y_t^m),\quad G\in\{F,\Sigma\},
\label{eq:RG}\\
\mathcal E_v(t)
&:=\frac12m^2\mathrm{tr}_Q\Big(
\big[S_{t-\ttau}\Sigma(X_{\ttau}^m)
-\Sigma(X_t)\big]^*
\partial_{yy}v(t,X_t,Y_t^m)
\big[S_{t-\ttau}\Sigma(X_{\ttau}^m)
-\Sigma(X_t)\big]\Big)\notag\\
&\quad-\frac12\mathrm{tr}_Q\Big(
\Phi_\Sigma(X_t,Y_t^m)^*
\partial_{yy}v(t,X_t,Y_t^m)
\Phi_\Sigma(X_t,Y_t^m)\Big)\notag\\
&\quad-\frac{T^2}{24}\mathrm{tr}_Q\Big(
\phi_\Sigma(X_t)^*
\partial_{yy}v(t,X_t,Y_t^m)
\phi_\Sigma(X_t)\Big).
\label{eq:EV}
\end{align} 
\end{lemma}
\begin{proof}
Applying It\^o's formula to \(v(t,X_t,Y_t^m)\), taking
expectations, and using \eqref{eq:Kolu} gives the identity \eqref{eq:vXY}.
The required integrability follows from Lemma~\ref{lem:W}
and Proposition~\ref{prop:varphiY}.
\end{proof}
To write the weak error representation compactly, for
\(h\in\cH_N\) and
\(U,V\in\mathscr L_2(H_Q,\cH_N)\), define
\begin{equation}\label{eq:ABCfunctionals}
\left\{
\begin{aligned}
&\mathscr A_t(h)
:=\langle h,\partial_yv(t,X_t,Y_t^m)\rangle,\\
&\mathscr C_t(U)
:=\mathrm{tr}_Q\big(
\Sigma(X_t)^*\partial_{xy}v(t,X_t,Y_t^m)U\big),\\
&\mathscr B_t(U,V)
:=\mathrm{tr}_Q\big(
U^*\partial_{yy}v(t,X_t,Y_t^m)V\big).
\end{aligned}
\right.
\end{equation}
Note that $\mathscr B_t(U,V)=\mathscr B_t(V,U).$
By \eqref{eq:RG},
\begin{equation}\label{eq:mS}
m\big\{S_{t-\ttau}\Sigma(X_{\ttau}^m)-\Sigma(X_t)\big\}
=\mathcal R_\Sigma(t)+\Phi_\Sigma(X_t,Y_t^m).
\end{equation}
Since {\(v(T,x,y)=\varphi(x,y)\)} and
{\(v(0,\mathcal P_N{x_0},0)
=\E[\varphi(X_T,Y_T)]\)}, taking \(t=T\) in
Lemma \ref{lem:L-L} and substituting
\eqref{eq:RG}, \eqref{eq:EV}, and \eqref{eq:mS} yield
\begin{align}\label{eq:delatvarphi-ABC}
&{\E[\varphi(X_T,Y_T^m)]-\E[\varphi(X_T,Y_T)]}\notag\\
&=\E\int_0^T\Big[
\mathscr A_t\big(\mathcal R_F(t)\big)
+\mathscr C_t\big(\mathcal R_\Sigma(t)\big)
+\mathscr B_t\big(\mathcal R_\Sigma(t),
\Phi_\Sigma(X_t,Y_t^m)\big)\notag\\
&\quad+\frac12\mathscr B_t\big(
\mathcal R_\Sigma(t),\mathcal R_\Sigma(t)\big)
-\frac{T^2}{24}\mathscr B_t\big(
\phi_\Sigma(X_t),\phi_\Sigma(X_t)\big)\Big]\ud t.
\end{align}
For \(G\in\{F,\Sigma\}\), by \eqref{eq:RG}, we rewrite
\begin{align}\label{eq:mathRG}\notag
\mathcal R_G(t)
&=m\left(S_{t-\ttau}G(X_{\ttau}^m)-G(X_t)\right)
-DG(X_t)Y_t^m
-\frac T2\left(AG(X_t)-DG(X_t)AX_t\right)\\\notag
&=m\big(S_{t-\ttau}-I-(t-\ttau)A\big)
G(X_{\ttau}^m)\\\notag
&\quad+m\left(t-\ttau-\frac{T}{2m}\right)
AG(X_{\ttau}^m)
+\frac T2A\big(G(X_{\ttau}^m)-G(X_t)\big)\\
&\quad+m\big(G(X_{\ttau}^m)-G(X_t)\big)
-DG(X_t)\left(Y_t^m-\frac T2AX_t\right).
\end{align}
We have the following relation
\begin{equation*}
\begin{aligned}
&m(X_t-X_{\ttau})-\frac T2AX_{\ttau}\\
&\quad=m\big(S_{t-\ttau}-I-(t-\ttau)A\big)X_{\ttau}+m\left(t-\ttau-\frac{T}{2m}\right)AX_{\ttau}
+m\big(X_t-S_{t-\ttau}X_{\ttau}\big).
\end{aligned}
\end{equation*}
Substituting this into the preceding expression \eqref{eq:mathRG} leads to
\begin{align}\label{eq:compactRG}
\mathcal R_G(t)
&=\vartheta_m(t)\phi_G(X_{\ttau})
+\widetilde{\mathcal R}_G(t),
\end{align}
where $\vartheta_m(t):=m\left(t-\ttau-\frac{T}{2m}\right)$  and
\begin{align}\label{eq:tildeRG}\notag
\widetilde{\mathcal R}_G(t)
&=m\big(S_{t-\ttau}-I-(t-\ttau)A\big)
G(X_{\ttau}^m)+\vartheta_m(t)A
\big(G(X_{\ttau}^m)-G(X_{\ttau})\big)
\notag\\
&\quad+\frac T2A\big(G(X_{\ttau}^m)-G(X_t)\big)\notag\\\notag
&\quad+\int_0^1\Big(
DG\big(\theta X_{\ttau}^m+(1-\theta)X_{\ttau}\big)
-DG(X_t)\Big)Y_{\ttau}^m\ud\theta\\\notag
&\quad+DG(X_t)(Y_{\ttau}^m-Y_t^m)\notag\\
&\quad+m\int_0^1\Big(
DG(X_t)-DG\big(\theta X_{\ttau}+(1-\theta)X_t\big)\Big)
(X_t-X_{\ttau})\ud\theta\notag\\\notag
&\quad+\big(DG(X_{\ttau})-DG(X_t)\big)
\Big(m\left(t-\ttau-\frac{T}{2m}\right)AX_{\ttau}\Big)+\frac T2DG(X_t)A(X_t-X_{\ttau})\\
&\quad-DG(X_t)\Big[
m\big(S_{t-\ttau}-I-(t-\ttau)A\big)X_{\ttau}
+m\big(X_t-S_{t-\ttau}X_{\ttau}\big)\Big].
\end{align}
The function $\vartheta_m$ defined above satisfies the
two cancellations
\begin{equation}\label{eq:centered-identities}
\int_{t_i}^{t_{i+1}}\vartheta_m(t)\ud t=0,\qquad
\int_{t_i}^{t_{i+1}}
\left(\frac12\vartheta_m(t)^2-\frac{T^2}{24}\right)\ud t=0.
\end{equation}

Substituting \eqref{eq:compactRG} into
\eqref{eq:delatvarphi-ABC}, and using
\eqref{eq:centered-identities}, we arrive at
\begin{equation}\label{eq:two-blocks}
{\E[\varphi(X_T,Y_T^m)]-\E[\varphi(X_T,Y_T)]}
=\mathcal E_{\rm rem}+\mathcal E_{\rm cent},
\end{equation}
where the remainder contributions $\mathcal{E}_{\rm rem}$ and the centered contributions $\mathcal{E}_{\rm cent}$ are given by, respectively, 
\begin{align}\label{eq:Erem}
\mathcal E_{\rm rem}
&:=\E\int_0^T\Big[
\mathscr A_t(\widetilde{\mathcal R}_F(t))
+\mathscr C_t(\widetilde{\mathcal R}_\Sigma(t))
+\mathscr B_t\big(\widetilde{\mathcal R}_\Sigma(t),
\Phi_\Sigma(X_t,Y_t^m)\big)\notag\\
&\quad\qquad
+\vartheta_m(t)\mathscr B_t\big(
\phi_\Sigma(X_{\ttau}),
\widetilde{\mathcal R}_\Sigma(t)\big)
+\frac12\mathscr B_t\big(
\widetilde{\mathcal R}_\Sigma(t),
\widetilde{\mathcal R}_\Sigma(t)\big)
\Big]\ud t,\\\label{eq:Ecent}
\mathcal E_{\rm cent}
&:=\E\sum_{i=0}^{m-1}\int_{t_i}^{t_{i+1}}\bigg[
\vartheta_m(t)\big\{
\mathscr A_t(\phi_F(X_{t_i}))
-\mathscr A_{t_i}(\phi_F(X_{t_i}))\big\}\notag\\
&\quad\qquad+\vartheta_m(t)\big\{\mathscr C_t(\phi_\Sigma(X_{t_i}))
-\mathscr C_{t_i}(\phi_\Sigma(X_{t_i}))\big\}\notag\\
&\quad\qquad+\vartheta_m(t)\big\{\mathscr B_t\big(\phi_\Sigma(X_{t_i}),
\Phi_\Sigma(X_t,Y_t^m)\big)-\mathscr B_{t_i}\big(
\phi_\Sigma(X_{t_i}),
\Phi_\Sigma(X_{t_i},Y_{t_i}^m)\big)\big\}\notag\\
&\quad\qquad+\frac12\vartheta_m(t)^2
\big\{\mathscr B_t\big(\phi_\Sigma(X_{t_i}),\phi_\Sigma(X_{t_i})\big)
-\mathscr B_{t_i}\big(
\phi_\Sigma(X_{t_i}),\phi_\Sigma(X_{t_i})\big)\big\}\notag\\
&\quad\qquad-\frac{T^2}{24}
\big\{\mathscr B_t\big(\phi_\Sigma(X_t),\phi_\Sigma(X_t)\big)
-\mathscr B_{t_i}\big(
\phi_\Sigma(X_{t_i}),\phi_\Sigma(X_{t_i})\big)\big\}
\bigg]\ud t.
\end{align}
The following proposition provides the required estimate for
$\mathcal E_{\rm cent}$.
\begin{proposition}\label{prop:one-step}
Let Assumptions \ref{asp:F}, \ref{asp:alp}, and \ref{asp:fg3} hold,
let ${x_0}\in\cH^\alpha$, and let \({\varphi}\in \mathscr T^4\).
Then there exists
$C>0$, independent of \(m\) and \(N\), such that, for
\(t\in[t_i,t_{i+1}]\) and \(i=0,1,\ldots,m-1\),
\begin{equation}\label{eq:one-step}
\begin{split}
&\left|\E\left[\mathscr A_t(\phi_F(X_{t_i}))
-\mathscr A_{t_i}(\phi_F(X_{t_i}))\right]\right|
\le C(t-t_i)^{\alpha-1},\\
&\left|\E\left[\mathscr C_t(\phi_{\Sigma}(X_{t_i}))
-\mathscr C_{t_i}(\phi_{\Sigma}(X_{t_i}))\right]\right|
\le C(t-t_i)^{\alpha-1},\\
&\left|\E\left[
\mathscr B_t\big(\phi_{\Sigma}(X_{t_i}),
\Phi_{\Sigma}(X_t,Y_t^m)\big)
-\mathscr B_{t_i}\big(\phi_{\Sigma}(X_{t_i}),
\Phi_{\Sigma}(X_{t_i},Y_{t_i}^m)\big)
\right]\right|
\le C(t-t_i)^{\alpha-1},\\
&\left|\E\left[
\mathscr B_t\big(\phi_{\Sigma}(X_{t_i}),
\phi_{\Sigma}(X_{t_i})\big)
-\mathscr B_{t_i}\big(\phi_{\Sigma}(X_{t_i}),
\phi_{\Sigma}(X_{t_i})\big)\right]\right|
\le C(t-t_i)^{\alpha-1},\\
&\left|\E\left[
\mathscr B_t\big(\phi_{\Sigma}(X_t),\phi_{\Sigma}(X_t)\big)
-\mathscr B_{t_i}\big(\phi_{\Sigma}(X_{t_i}),
\phi_{\Sigma}(X_{t_i})\big)
\right]\right|
\le C(t-t_i)^{\alpha-1},
\end{split}
\end{equation}
where the functionals
\(\mathscr A_t\), \(\mathscr B_t\), and \(\mathscr C_t\) are defined in
\eqref{eq:ABCfunctionals}.
\end{proposition}

We defer the lengthy proof of Proposition~\ref{prop:one-step}
to Section~\ref{S:one-step-proof}.
To estimate $\mathcal E_{\rm rem}$ defined in \eqref{eq:Erem},
we further separate the stochastic integral terms on the right hand side of \eqref{eq:tildeRG} coming from
$Y_{\ttau}^m-Y_t^m$ and $X_t-X_{\ttau}$.  Define
\begin{gather}\label{eq:Zi}
\mathcal Z(t):=
-m\int_{\ttau}^tS_{t-{\ttau}}\Sigma(X_{\ttau}^m)\ud W_s
+\frac T2\int_{\ttau}^tAS_{t-s}\Sigma(X_s)\ud W_s,\quad t\in[0,T]
\\\label{eq:Rst}
\mathcal R_{G,\mathrm{st}}(t):=DG(X_t)\mathcal Z(t),\qquad \mathcal R_{G,\mathrm{reg}}(t):=
\widetilde{\mathcal R}_G(t)-\mathcal R_{G,\mathrm{st}}(t),
\end{gather} 
where $G\in\{F,\Sigma\}$.
This decomposition separates the regular part from the explicit
stochastic-integral part, which have different convergence orders, as
shown in the following lemma.
\begin{lemma}\label{lem:RGXY}
Let Assumptions \ref{asp:F}, \ref{asp:alp}, and
\ref{asp:fg3} hold and let \({x_0}\in\cH^\alpha\).
Then, for every \(p\ge2\), there exists
\(C>0\), independent of \(m\) and \(N\), such that, for
\(t\in[0,T]\),
\begin{equation}\label{eq:Rreg-Rst-bounds}
\begin{aligned}
&\|\mathcal R_{F,\mathrm{reg}}(t)\|_{L^p(\Omega,\cH)}
+\|\mathcal R_{\Sigma,\mathrm{reg}}(t)\|_
 {L^p(\Omega,\mathscr L_2(H_Q,\cH))}
\le Cm^{-(\alpha-1)},\\
&\|\mathcal R_{F,\mathrm{st}}(t)\|_{L^p(\Omega,\cH)}
+\|\mathcal R_{\Sigma,\mathrm{st}}(t)\|_
 {L^p(\Omega,\mathscr L_2(H_Q,\cH))}
\le Cm^{-\frac12}.
\end{aligned}
\end{equation}
\end{lemma}
The proof of Lemma \ref{lem:RGXY} can be found in
Section \ref{S:remainder-proofs}. 
\begin{remark}\label{rem:Erem}

Separating the stochastic remainder terms, as in \eqref{eq:Rst}, is essential for proving Theorem~\ref{tho:Galerkin-main}.
In fact, direct use of the moment estimates in
Lemma~\ref{lem:RGXY} for $\widetilde{\mathcal R}_F(t)$ and
$\widetilde{\mathcal R}_\Sigma(t)$ yields only an
$O(m^{-\min\{\alpha-1,\frac12\}})$ bound for $\mathcal E_{\rm rem}$,
which is insufficient to establish
$|\mathcal E_{\rm rem}|\le Cm^{-(\alpha-1)}$ when $\alpha>3/2$.

\end{remark}
By \(\widetilde{\mathcal R}_G
=\mathcal R_{G,\mathrm{reg}}+\mathcal R_{G,\mathrm{st}}\), we decompose $
\mathcal E_{\rm rem}
=\mathcal E_{\rm rem}^{(\mathrm{reg})}
+\mathcal E_{\rm rem}^{({\rm st})}
+\mathcal E_{\rm rem}^{(2)}$,
where
\begin{align}\label{eq:Erem-reg}
\mathcal E_{\rm rem}^{(\mathrm{reg})}
&:=\E\int_0^T\Big[
\mathscr A_t(\mathcal R_{F,\mathrm{reg}}(t))
+\mathscr C_t(\mathcal R_{\Sigma,\mathrm{reg}}(t))\notag\\
&\quad\qquad+\mathscr B_t\Big(
\mathcal R_{\Sigma,\mathrm{reg}}(t),
\Phi_\Sigma(X_t,Y_t^m)
+\vartheta_m(t)\phi_\Sigma(X_{\ttau})\Big)
\Big]\ud t,\\
\label{eq:Erem-st}
\mathcal E_{\rm rem}^{({\rm st})}
&:=\E\int_0^T\Big[
\mathscr A_t(\mathcal R_{F,\mathrm{st}}(t))
+\mathscr C_t(\mathcal R_{\Sigma,\mathrm{st}}(t))\notag\\
&\quad\qquad+\mathscr B_t\Big(
\mathcal R_{\Sigma,\mathrm{st}}(t),
\Phi_\Sigma(X_t,Y_t^m)
+\vartheta_m(t)\phi_\Sigma(X_{\ttau})\Big)
\Big]\ud t,\\
\label{eq:Erem-second}
\mathcal E_{\rm rem}^{(2)}
&:=\frac12\E\int_0^T
\mathscr B_t\big(\widetilde{\mathcal R}_\Sigma(t),
\widetilde{\mathcal R}_\Sigma(t)\big)\ud t.
\end{align}
The following proposition provides the estimate for
$\mathcal E_{\rm rem}^{(\mathrm{st})}$.
\begin{proposition}\label{prop:Erem}
Let Assumptions \ref{asp:F}, \ref{asp:alp}, and \ref{asp:fg3} hold,
let ${x_0}\in\cH^\alpha$, and let \({\varphi}\in \mathscr T^4\).
Then there exists
$C>0$, independent of \(m\) and \(N\), such that, for
\(t\in[t_i,t_{i+1}]\) and \(i=0,1,\ldots,m-1\),
\begin{align}\label{eq:Erem-bound}\notag
&|\E[\mathscr A_t(\mathcal R_{F,\mathrm{st}}(t))]|
+|\E[\mathscr C_t(\mathcal R_{\Sigma,\mathrm{st}}(t))]|\\
&\quad+\left|\E\left[\mathscr B_t\Big(
\mathcal R_{\Sigma,\mathrm{st}}(t),
\Phi_\Sigma(X_t,Y_t^m)
+\vartheta_m(t)\phi_\Sigma(X_{t_i})\Big)\right]\right|
\le Cm^{-(\alpha-1)},
\end{align}
where the functionals
\(\mathscr A_t\), \(\mathscr B_t\), and \(\mathscr C_t\) are defined in
\eqref{eq:ABCfunctionals}, and $\mathcal R_{\cdot,\mathrm{st}}(t)$  is defined in \eqref{eq:Rst}.
\end{proposition}
The proof of Proposition \ref{prop:Erem} is given in
Section \ref{S:remainder-proofs}.
We now combine the preceding estimates to prove
Theorem~\ref{tho:Galerkin-main}.

\begin{proof}[Proof of Theorem \ref{tho:Galerkin-main}]
For \(t\in[0,T]\), Proposition \ref{prop:varphiY},
\eqref{eq:Galerkin-lin}, \eqref{eq:Phi-growth},
\eqref{eq:LpX-Galerkin}, and
\eqref{Strongorder-Galerkin} imply
\begin{align}\label{eq:partvt}
&\left\|\partial_yv(t,X_t,Y_t^m)\right\|_
 {L^4(\Omega,\mathscr L(\cH,\R))}
+\left\|\partial_{xy}v(t,X_t,Y_t^m)\right\|_
 {L^4(\Omega,\mathscr L(\cH^1\times\cH,\R))}\notag\\
&\quad
+\left\|\partial_{yy}v(t,X_t,Y_t^m)\right\|_
 {L^4(\Omega,\mathscr L(\cH\times\cH,\R))}
+\left\|\Sigma(X_t)\right\|_
 {L^4(\Omega,\mathscr L_2(H_Q,\cH^1))}\notag\\
&\quad
+\left\|\Phi_\Sigma(X_t,Y_t^m)\right\|_
 {L^4(\Omega,\mathscr L_2(H_Q,\cH))}
+\left\|\phi_\Sigma(X_{\ttau})\right\|_
 {L^4(\Omega,\mathscr L_2(H_Q,\cH))}
\le C .
\end{align}
Furthermore, by \eqref{eq:ABCfunctionals} and the Cauchy--Schwarz inequality, we have 
\begin{align*}
\left|\mathscr A_t(h)\right|&\le \|\partial_yv(t,X_t,Y_t^m)\|_{\mathscr L(\cH,\R)}\|h\|_{\cH},
\\
\left|\mathscr C_t(U)\right|
&\le
\left\|\partial_{xy}v(t,X_t,Y_t^m)\right\|_
 {\mathscr L(\cH^1\times\cH,\R)}
\left\|\Sigma(X_t)\right\|_{\mathscr L_2(H_Q,\cH^1)}
\|U\|_{\mathscr L_2(H_Q,\cH)},\\
\left|\mathscr B_t(U,V)\right|&\le \left\|\partial_{yy}v(t,X_t,Y_t^m)\right\|_
 {\mathscr L(\cH\times\cH,\R)}
\left\|U\right\|_{\mathscr L_2(H_Q,\cH)}
\|V\|_{\mathscr L_2(H_Q,\cH)}.
\end{align*}
Consequently, Hölder's inequality, \eqref{eq:partvt},  and the
preceding bound yield
\begin{align*}
&\E\left|\mathscr A_t(\mathcal R_{F,\mathrm{reg}}(t))\right|
\le C\|\mathcal R_{F,\mathrm{reg}}(t)\|_{L^2(\Omega,\cH)},\\
&\E\left|\mathscr C_t(\mathcal R_{\Sigma,\mathrm{reg}}(t))\right|
\le C\|\mathcal R_{\Sigma,\mathrm{reg}}(t)\|_
 {L^2(\Omega,\mathscr L_2(H_Q,\cH))},\\
&\E\left|\mathscr B_t\Big(
\mathcal R_{\Sigma,\mathrm{reg}}(t),
\Phi_\Sigma(X_t,Y_t^m)
+\vartheta_m(t)\phi_\Sigma(X_{\ttau})\Big)\right|
\le C\|\mathcal R_{\Sigma,\mathrm{reg}}(t)\|_
 {L^2(\Omega,\mathscr L_2(H_Q,\cH))},
\end{align*}
Here the last inequality also uses
\(|\vartheta_m(t)|\le T/2\).  Inserting
\eqref{eq:Rreg-Rst-bounds} into these three estimates, we deduce from
\eqref{eq:Erem-reg} that
\begin{equation}\label{eq:Erem-reg-bound}
\big|\mathcal E_{\rm rem}^{(\mathrm{reg})}\big|
\le Cm^{-(\alpha-1)}.
\end{equation}
Moreover,
\(\widetilde{\mathcal R}_\Sigma
=\mathcal R_{\Sigma,\mathrm{reg}}
+\mathcal R_{\Sigma,\mathrm{st}}\), so
Hölder's inequality, \eqref{eq:partvt},
\eqref{eq:trace}, and Lemma \ref{lem:RGXY} yield that for \(t\in[0,T]\),
\begin{align}\label{eq:Erem-second-bound}
\frac12\E\big|\mathscr B_t\big(
\widetilde{\mathcal R}_\Sigma(t),
\widetilde{\mathcal R}_\Sigma(t)\big)\big|
&\le C(
m^{-2(\alpha-1)}
+m^{-(\alpha-\frac12)}
+m^{-1})
\le Cm^{-(\alpha-1)}.
\end{align}
By Proposition \ref{prop:Erem} and
\eqref{eq:Erem-st},
\(
\big|\mathcal E_{\rm rem}^{({\rm st})}\big|
\le Cm^{-(\alpha-1)}.
\)
Similarly, \eqref{eq:Erem-second-bound} and
\eqref{eq:Erem-second} imply
\(
\big|\mathcal E_{\rm rem}^{(2)}\big|
\le Cm^{-(\alpha-1)}.
\)
Consequently, the relation $
\mathcal E_{\rm rem}
=\mathcal E_{\rm rem}^{(\mathrm{reg})}
+\mathcal E_{\rm rem}^{({\rm st})}
+\mathcal E_{\rm rem}^{(2)}$ and
\eqref{eq:Erem-reg-bound} yield
$
|\mathcal E_{\rm rem}|
\le Cm^{-(\alpha-1)}.
$
By Proposition \ref{prop:one-step}, \eqref{eq:Ecent}, and
\(|\vartheta_m(t)|^2\le T^2/4\), we have
\(
|\mathcal E_{\rm cent}|
\le Cm^{-(\alpha-1)}.
\)
Combining the estimates for \(\mathcal E_{\rm rem}\) and
\(\mathcal E_{\rm cent}\) with \eqref{eq:two-blocks} proves the required
\eqref{eq:Galerkin-main-statement}. 
\end{proof}
\subsection{Proofs of Lemma \ref{lem:RGXY} and Proposition \ref{prop:Erem}}
\label{S:remainder-proofs}
In this subsection, we prove the remainder estimates in
Lemma \ref{lem:RGXY} and Proposition \ref{prop:Erem}.  We first derive
uniform moment and one-step increment estimates for the spectral Galerkin
normalized error \(Y^m\), and then use them to prove the two stated
remainder estimates.
\begin{lemma}
\label{lem:ACLW}
Let Assumptions \ref{asp:F} and \ref{asp:alp} hold and let
${x_0}\in\cH^\alpha$.  Then for every $p\ge2$, there exists
$C>0$, independent of $m$ and $N$, such that for any $t\in[0,T]$,
\begin{gather}\label{eq:ACLW-Galerkin}
\| Y_t^m\|_{L^p(\Omega,\cH^1)}
\le Cm^{2-\alpha},\\\label{lem:XYreg}
			\|X_t-X_{\ttau}\|_{L^p(\Omega,\cH^1)}
	+\|Y_t^m-Y_{\ttau}^m\|_{L^p(\Omega,\cH)}
	\le C m^{-\min\{\alpha-1,\frac12\}}.
\end{gather} 
\end{lemma}
\begin{proof}
The estimate \eqref{eq:ACLW-Galerkin} can be proved by applying
\eqref{eq:convolution-bounds} in $\cH^1$ to the mild formulation
of $Y^m$, together with \eqref{eq:St-IN},
\eqref{eq:Galerkin-Fz}, \eqref{eq:Xalp-Galerkin},
and Lemma~\ref{lem:W}.
The estimate \eqref{lem:XYreg} follows from the mild
formulations of $X$ and $Y^m$, the group estimate
\eqref{eq:St-IN}, \eqref{eq:ACLW-Galerkin},
and the Burkholder--Davis--Gundy inequality. We omit further tedious details.
\end{proof}

\begin{proof}[Proof of Lemma \ref{lem:RGXY}]
The properties \eqref{eq:Abound},
\eqref{eq:Galerkin-lin}, Lemma \ref{lem:W}, and the
Burkholder--Davis--Gundy inequality yield
\begin{align}\label{eq:Zi-first}
\|[\mathcal Z(t)]_1\|_{L^p(\Omega,H)}
\le C\left(
m(t-\ttau)^{\frac32}+(t-\ttau)^{\frac12}\right)
\le C(t-\ttau)^{\frac12}.
\end{align}
To treat the drift and diffusion remainders with the same notation,
set $\mathscr G_F=\cH$ and
$\mathscr G_\Sigma=\mathscr L_2(H_Q,\cH)$.
Then \eqref{eq:Rst} and
\eqref{eq:DG1} imply
\begin{align*}
\|\mathcal R_{G,\mathrm{st}}(t)\|_
 {L^p(\Omega,\mathscr G_G)}
\le C\|[\mathcal Z(t)]_1\|_{L^p(\Omega,H)}
\le Cm^{-\frac12}.
\end{align*}
It remains to estimate $\mathcal R_{G,\mathrm{reg}}(t)$. 
Subtracting \eqref{eq:Rst} from \eqref{eq:tildeRG} yields that for $G\in\{F,\Sigma\}$,
\begin{align}\label{eq:Rreg-explicit}
\mathcal R_{G,\mathrm{reg}}(t)
&={}m(S_{t-\ttau}-I-(t-\ttau)A)
G(X_{\ttau}^m)+\vartheta_m(t)A
\big(G(X_{\ttau}^m)-G(X_{\ttau})\big)\notag\\
&+\frac T2A\big(G(X_{\ttau}^m)-G(X_t)\big)+\int_0^1\!\big(DG(\theta X_{\ttau}^m+(1-\theta)X_{\ttau})
-DG(X_t)\big)Y_{\ttau}^m\ud\theta\notag\\
&+m\int_0^1\!\big(DG(X_t)
-DG(\theta X_{\ttau}+(1-\theta)X_t)\big)
(X_t-X_{\ttau})\ud\theta\notag\\
&+\vartheta_m(t)
\big(DG(X_{\ttau})-DG(X_t)\big)AX_{\ttau}-DG(X_t)(S_{t-\ttau}-I)Y_{\ttau}^m\notag\\
&+\frac T2DG(X_t)A(S_{t-\ttau}-I)X_{\ttau}-mDG(X_t)
\big(S_{t-\ttau}-I-(t-\ttau)A\big)X_{\ttau}\notag\\
&+DG(X_t)\Big[
-m(t-\ttau)S_{t-\ttau}F(X_{\ttau}^m)
+\frac T2\int_{\ttau}^tAS_{t-s}F(X_s)\ud s\Big].
\end{align}
In view of  \eqref{eq:St-IN} and
\eqref{eq:Abound}, for any $z\in\cH_N$ and $\beta\in(0,1]$,
\begin{equation}\label{eq:Sreg}
\|S_{t-\ttau}z-z-(t-\ttau)Az\|_{\cH}
=\left\|\int_0^{t-\ttau}(S_r-I)Az\,\ud r
	\right\|_{\cH}
\le C(t-\ttau)^{\beta+1}\|z\|_{\cH^{\beta+1}}.
\end{equation}
Applying \eqref{eq:Sreg}, \eqref{eq:Galerkin-Fz},
\eqref{eq:Xalp-Galerkin}, \eqref{eq:DG1}, and
\eqref{eq:DjG0} term by term in
\eqref{eq:Rreg-explicit}, and using
$|\vartheta_m(t)|\le T/2$, we can deduce that
\begin{align}\label{eq:RGreg}
&\|\mathcal R_{G,\mathrm{reg}}(t)\|_{L^p(\Omega,\mathscr{G}_G)}\\\notag
&\le Cm(t-\ttau)^\alpha
+C\|[X_{\ttau}^m-X_{\ttau}]_1\|_{L^p(\Omega,H)}
+C\|[X_{\ttau}^m-X_t]_1\|_{L^p(\Omega,H)}
\\\notag
&\quad+C\Big(
\|[X_{\ttau}^m-X_{\ttau}]_1\|_{L^{2p}(\Omega,H)}
+\|[X_t-X_{\ttau}]_1\|_{L^{2p}(\Omega,H)}
\Big)
\|Y_{\ttau}^m\|_{L^{2p}(\Omega,\cH)}
\\\notag
&\quad+Cm\|[X_t-X_{\ttau}]_1\|_{L^{2p}(\Omega,H)}^2
+C\|[X_t-X_{\ttau}]_1\|_{L^{2p}(\Omega,H)}
\|X_{\ttau}\|_{L^{2p}(\Omega,\cH^1)}
\\\notag
&\quad+C\|(S_{t-\ttau}-I)Y_{\ttau}^m\|_{L^p(\Omega,\cH)}
+C\|A(S_{t-\ttau}-I)X_{\ttau}\|_{L^p(\Omega,\cH)}
\\\notag
&\quad+Cm(t-\ttau)^\alpha
+Cm(t-\ttau)
\|[S_{t-\ttau}F(X_{\ttau}^m)]_1\|_{L^p(\Omega,H)}\\\notag
&\quad+C\int_{\ttau}^t
\|[AS_{t-s}F(X_s)]_1\|_{L^p(\Omega,H)}\ud s.
\end{align}

Before estimating the group increments, we first control
the terms in \eqref{eq:RGreg} involving increments of \(X\) and
\(X^m\).  The identity
\(X_{\ttau}^m-X_{\ttau}=m^{-1}Y_{\ttau}^m\),
\eqref{Holderq-Galerkin}, and Lemma \ref{lem:W} imply
\begin{align*}
\|[X_{\ttau}^m-X_{\ttau}]_1\|_{L^p(\Omega,H)}
&\le Cm^{-1},\qquad
\|[X_t-X_{\ttau}]_1\|_{L^p(\Omega,H)}
\le C(t-\ttau),\\
\|[X_{\ttau}^m-X_t]_1\|_{L^p(\Omega,H)}
&\le C\big(m^{-1}+t-\ttau\big)\le Cm^{-1}.
\end{align*}
Together with Lemma \ref{lem:W}, these bounds show that
all terms in \eqref{eq:RGreg} containing
\(X_{\ttau}^m-X_{\ttau}\), \(X_{\ttau}^m-X_t\), or
\(X_t-X_{\ttau}\) are bounded by
\(Cm^{-(\alpha-1)}\), since \(1<\alpha\le2\).

Again by \eqref{eq:St-IN} and \eqref{eq:Abound}, together with Lemma \ref{lem:ACLW} and \eqref{eq:Xalp-Galerkin}, we arrive at
\begin{align}\label{eq:St1}
\|(S_{t-\ttau}-I)Y_{\ttau}^m\|_{L^p(\Omega,\cH)}
&\le C(t-\ttau)\|Y_{\ttau}^m\|_{L^p(\Omega,\cH^1)}
\le Cm^{-(\alpha-1)},\\\label{eq:St2}
\|A(S_{t-\ttau}-I)X_{\ttau}\|_{L^p(\Omega,\cH)}
&\le C(t-\ttau)^{\alpha-1}
\|X_{\ttau}\|_{L^p(\Omega,\cH^\alpha)}
\le Cm^{-(\alpha-1)}.
\end{align}
Finally, because $F(z)=(0,P_Nf(z_1))^\top$ for $z\in\cH_N$, it follows from \eqref{eq:St-IN}, \eqref{eq:Galerkin-lin} and Lemma \ref{lem:W} that
\begin{align}\label{eq:St3}
&\|[S_{t-\ttau}F(X_{\ttau}^m)]_1\|_{L^p(\Omega,H)}=\|[(S_{t-\ttau}-I)F(X_{\ttau}^m)]_1\|_{L^p(\Omega,H)}\\\notag
&\le C(t-\ttau)\|F(X_{\ttau}^m)\|_{L^p(\Omega,\cH^1)}\le C(t-\ttau)(1+\|X_{\ttau}^m\|_{L^p(\Omega,\cH)})\le Cm^{-1},\\\label{eq:St4}
&\|[AS_{t-s}F(X_s)]_1\|_{L^p(\Omega,H)}\le C\|F(X_s)\|_{L^p(\Omega,\cH^1)}\le C.
\end{align}
Substituting \eqref{eq:St1}, \eqref{eq:St2},
\eqref{eq:St3}, and \eqref{eq:St4} into \eqref{eq:RGreg} proves
\eqref{eq:Rreg-Rst-bounds}.
\end{proof}
\begin{proof}[Proof of Proposition \ref{prop:Erem}]
For fixed \(t\), define
\begin{align*}
\Theta_1(x,y,z)
&:=\left\langle DF(x)z,\partial_yv(t,x,y)\right\rangle,\\
\Theta_2(x,y,z)
&:=\mathrm{tr}_Q\big(
\Sigma(x)^*\partial_{xy}v(t,x,y)D\Sigma(x)z\big),\\
\Theta_3(x,y,z)
&:=\mathrm{tr}_Q\Big(
(D\Sigma(x)z)^*\partial_{yy}v(t,x,y)
\big[\Phi_\Sigma(x,y)
+\vartheta_m(t)\phi_\Sigma(X_{t_i})\big]\Big).
\end{align*}
Using
\(\mathcal R_{G,\mathrm{st}}(t)=DG(X_t)\mathcal Z(t)\),
we have
\begin{gather*}
\mathscr A_t(\mathcal R_{F,\mathrm{st}}(t))
=\Theta_1(X_t,Y_t^m,\mathcal Z(t)),\qquad
\mathscr C_t(\mathcal R_{\Sigma,\mathrm{st}}(t))
=\Theta_2(X_t,Y_t^m,\mathcal Z(t)),\\
\mathscr B_t\Big(
\mathcal R_{\Sigma,\mathrm{st}}(t),
\Phi_\Sigma(X_t,Y_t^m)
+\vartheta_m(t)\phi_\Sigma(X_{t_i})\Big)
=\Theta_3(X_t,Y_t^m,\mathcal Z(t)).
\end{gather*}
Since the coefficients of
\(\Theta_\ell(X_{t_i},Y_{t_i}^m,\cdot)\) are
\(\mathscr F_{t_i}\)-measurable, \eqref{eq:Zi} implies
\begin{equation}\label{eq:frozen-stochastic-terms}
\E\left[
\Theta_\ell(X_{t_i},Y_{t_i}^m,\mathcal Z(t))
\mid\mathscr F_{t_i}\right]=0,
\qquad \ell=1,\ldots,3.
\end{equation}
Taking account of \eqref{eq:frozen-stochastic-terms}, we arrive at
\begin{equation}\label{eq:DiffTheta}
\E\Theta_\ell(X_t,Y_t^m,\mathcal Z(t))=\E\big[
\Theta_\ell(X_t,Y_t^m,\mathcal Z(t))
-\Theta_\ell(X_{t_i},Y_{t_i}^m,\mathcal Z(t))
\big].
\end{equation}
Let \(l,h\in\cH_N\), where \(l\) is measured in the ambient
\(\cH^1\)-norm and \(h\) in the ambient \(\cH\)-norm.  For
\(\Theta_1\), the product rule yields
\begin{align*}
D_x\Theta_1(x,y,z)l
&=\left\langle D^2F(x)(l,z),\partial_yv(t,x,y)\right\rangle
+\partial_{xy}v(t,x,y)\big(l,DF(x)z\big),\\
D_y\Theta_1(x,y,z)h
&=\partial_{yy}v(t,x,y)\big(h,DF(x)z\big).
\end{align*}
By Lemma \ref{lem:FSigN} and Proposition \ref{prop:varphiY}, there exist
$C>0$ and $r\ge0$, independent of $N$, such that
\begin{align*}
|D_x\Theta_1(x,y,z)l|+|D_y\Theta_1(x,y,z)h|
\le C(1+\|x\|_{\cH^1}+\|y\|_{\cH})^r
(\|l\|_{\cH^1}+\|h\|_{\cH})\|z_1\|_H .
\end{align*}

For \(\Theta_2\), direct differentiation yields
\begin{align*}
D_x\Theta_2(x,y,z)l
&=\mathrm{tr}_Q\big(
(D\Sigma(x)l)^*\partial_{xy}v(t,x,y)D\Sigma(x)z\big)\\
&\quad+\mathrm{tr}_Q\big(
\Sigma(x)^*\partial_{xxy}v(t,x,y)
(l,D\Sigma(x)z)\big)\\
&\quad+\mathrm{tr}_Q\big(
\Sigma(x)^*\partial_{xy}v(t,x,y)
D^2\Sigma(x)(l,z)\big),\\
D_y\Theta_2(x,y,z)h
&=\mathrm{tr}_Q\big(
\Sigma(x)^*\partial_{xyy}v(t,x,y)
(h,D\Sigma(x)z)\big).
\end{align*}
Therefore, \eqref{eq:trace}, Lemma \ref{lem:FSigN}, and
Proposition \ref{prop:varphiY} imply
\begin{equation*}
|D_x\Theta_2(x,y,z)l|+|D_y\Theta_2(x,y,z)h|
\le C(1+\|x\|_{\cH^1}+\|y\|_{\cH})^r
(\|l\|_{\cH^1}+\|h\|_{\cH})\|z_1\|_H .
\end{equation*}

Similarly, differentiating the three factors in the definition of
\(\Theta_3\) and applying the product rule, and then using
\eqref{eq:Phi-growth},
\eqref{eq:phi-Phi-derivative-bounds}, \eqref{eq:trace}, Lemma
\ref{lem:FSigN}, Proposition \ref{prop:varphiY}, and
\(|\vartheta_m(t)|\le T/2\), we obtain
\begin{align*}
&|D_x\Theta_3(x,y,z)l|+|D_y\Theta_3(x,y,z)h|\\
&\quad\le
C(1+\|X_{t_i}\|_{\cH^1})
(1+\|x\|_{\cH^1}+\|y\|_{\cH})^r
(\|l\|_{\cH^1}+\|h\|_{\cH})\|z_1\|_H .
\end{align*}
Applying the mean-value formula to \eqref{eq:DiffTheta}, then using the moment bounds \eqref{eq:LpX-Galerkin} and \eqref{Strongorder-Galerkin},  we obtain, for $\ell=1,2,3$,
\begin{align*}
&\left|
\E\Theta_\ell(X_t,Y_t^m,\mathcal Z(t))\right|\\
&\le
C\big(
\|X_t-X_{t_i}\|_{L^p(\Omega,\cH^1)}
+\|Y_t^m-Y_{t_i}^m\|_{L^p(\Omega,\cH)}\big)
\|[\mathcal Z(t)]_1\|_{L^p(\Omega,H)},
\end{align*}
where $p$ is sufficiently large. Furthermore, by H\"older's inequality,  \eqref{lem:XYreg}, and \eqref{eq:Zi-first}, \begin{align*}
\left|
\E\Theta_\ell(X_t,Y_t^m,\mathcal Z(t))\right|
&\le
Cm^{-\min\{\alpha-1,\frac12\}}(t-t_i)^{1/2}
\le Cm^{-(\alpha-1)}.
\end{align*}
Here \(t-t_i\le T/m\) and
\(\min\{\alpha-1,\frac12\}+\frac12\ge\alpha-1\).
The proof is complete.
\end{proof}
\subsection{Proof of Proposition \ref{prop:one-step}}
\label{S:one-step-proof}
By  \eqref{eq:RG}, we can rewrite
\begin{align*}
\ud Y_s^m
&=(AY_s^m+f_s^m)\ud s+\sigma_s^m\ud W_s,
\end{align*}
where \(f_s^m\) and \(\sigma_s^m\) are defined by
\begin{align}\label{eq:fsigma-cell}
f_s^m:=\Phi_F(X_s,Y_s^m)+\mathcal R_F(s),\qquad
\sigma_s^m:=\Phi_\Sigma(X_s,Y_s^m)+\mathcal R_\Sigma(s).
\end{align}
For a smooth function \(\Theta\), define
\begin{align}\label{eq:local-Ito-operator}
(\mathscr L_{m,s}\Theta)(s,x,y)
&:=D_x\Theta(s,x,y)(Ax+F(x))
  +D_y\Theta(s,x,y)(Ay+f_s^m)\notag\\
&\quad+\frac12\mathrm{tr}_Q\!\left(
\Sigma(x)^*D_{xx}\Theta(s,x,y)\Sigma(x)\right)+\mathrm{tr}_Q\!\left(
\Sigma(x)^*D_{xy}\Theta(s,x,y)\sigma_s^m\right)\notag\\
&\quad
  +\frac12\mathrm{tr}_Q\!\left(
(\sigma_s^m)^*D_{yy}\Theta(s,x,y)\sigma_s^m\right).
\end{align}
It\^o's formula applied to \(\Theta(s,X_s,Y_s^m)\) reads
\begin{align}\label{eq:dTheta}
\ud\Theta(s,X_s,Y_s^m)
&=\big[(\partial_s+\mathscr L_{m,s})\Theta\big]
  (s,X_s,Y_s^m)\ud s\\\notag
&\quad+\Big[
D_x\Theta(s,X_s,Y_s^m)\Sigma(X_s)
+D_y\Theta(s,X_s,Y_s^m)\sigma_s^m
\Big]\ud W_s.
\end{align}
Hence \(\mathscr L_{m,s}\) is the local It\^o operator associated
with the semimartingale \((X_s,Y_s^m)\).

Fix \(i\in\{0,1,\ldots,m-1\}\) and
\(t\in[t_i,t_{i+1}]\).  For \(s\in[t_i,t]\) and
\(x,y\in\cH_N\), define
\begin{equation}\label{eq:frozen-functions}
\begin{split}
\Psi_i^{(1)}(s,x,y)
&:=\langle\phi_F(X_{t_i}),\partial_yv(s,x,y)\rangle,\\
\Psi_i^{(2)}(s,x,y)
&:=\mathrm{tr}_Q\big(
\Sigma(x)^*\partial_{xy}v(s,x,y)\phi_\Sigma(X_{t_i})\big),\\
\Psi_i^{(3)}(s,x,y)
&:=\mathrm{tr}_Q\big(
\phi_\Sigma(X_{t_i})^*\partial_{yy}v(s,x,y)
\Phi_\Sigma(x,y)\big),\\
\Psi_i^{(4)}(s,x,y)
&:=\mathrm{tr}_Q\big(
\phi_\Sigma(X_{t_i})^*\partial_{yy}v(s,x,y)
\phi_\Sigma(X_{t_i})\big),\\
\Psi_i^{(5)}(s,x,y)
&:=\mathrm{tr}_Q\big(
\phi_\Sigma(x)^*\partial_{yy}v(s,x,y)\phi_\Sigma(x)\big).
\end{split}
\end{equation}
The five quantities
\(\Psi_i^{(\ell)}(s,X_s,Y_s^m)\), \(\ell=1,\ldots,5\), are precisely
the five quantities appearing in \eqref{eq:one-step}. 
To prove Proposition \ref{prop:one-step}, it therefore
suffices to prove
\begin{equation}\label{eq:Psiiel}
\left|\E\!\left[
\Psi_i^{(\ell)}(t,X_t,Y_t^m)
-\Psi_i^{(\ell)}(t_i,X_{t_i},Y_{t_i}^m)\right]\right|
\le C(t-t_i)^{\alpha-1}.
\end{equation}
Recall that 
$
\phi_G(z)=AG(z)-DG(z)Az$ and
$\Phi_\Sigma(z,\xi)=D\Sigma(z)\xi+\frac T2\phi_\Sigma(z).$
  Thus, for every \(p\ge2\) and
\(s\in[t_i,t]\),
\begin{align}\label{eq:frozen-directions}
&\|\phi_F(X_s)\|_{L^p(\Omega,\cH^{\alpha-1})}
+\|\phi_\Sigma(X_s)\|_
 {L^p(\Omega,\mathscr L_2(H_Q,\cH^{\alpha-1}))}
+\|\Phi_F(X_s,Y_s^m)\|_{L^p(\Omega,\cH)}\notag\\
&\qquad
+\|\Phi_\Sigma(X_s,Y_s^m)\|_
 {L^p(\Omega,\mathscr L_2(H_Q,\cH))}
+\|\Sigma(X_s)\|_
 {L^p(\Omega,\mathscr L_2(H_Q,\cH^1))}
\le C .
\end{align}
Here, \eqref{eq:Abound}, \eqref{eq:Galerkin-Fz},
\eqref{eq:DG1}, and \eqref{eq:Xalp-Galerkin} yield the estimates of
\(\phi_F(X_s)\) and \(\phi_\Sigma(X_s)\) in
\(\cH^{\alpha-1}\).  The estimates of
\(\Phi_F(X_s,Y_s^m)\) and \(\Phi_\Sigma(X_s,Y_s^m)\) follow from
\eqref{eq:Phi-growth}, \eqref{eq:LpX-Galerkin}, and
\eqref{Strongorder-Galerkin}, while that of \(\Sigma(X_s)\) follows
from \eqref{eq:Galerkin-lin} and \eqref{eq:LpX-Galerkin}.
We next establish a moment estimate for the coefficients
in \eqref{eq:fsigma-cell}.  Recall from \eqref{eq:compactRG} that, for
\(G\in\{F,\Sigma\}\),
$
\mathcal R_G(s)
=\widetilde{\mathcal R}_G(s)
+\vartheta_m(s)\phi_G(X_{\underline s}).
$
The identity
\(\widetilde{\mathcal R}_G
=\mathcal R_{G,\mathrm{reg}}+\mathcal R_{G,\mathrm{st}}\)
and Lemma \ref{lem:RGXY}, along with \eqref{eq:frozen-directions} and
\(|\vartheta_m(s)|\le T/2\) lead to
\[
\|\mathcal R_F(s)\|_{L^p(\Omega,\cH)}
+\|\mathcal R_\Sigma(s)\|_
 {L^p(\Omega,\mathscr L_2(H_Q,\cH))}
\le C.
\]
Consequently, \eqref{eq:fsigma-cell} and
\eqref{eq:frozen-directions} imply, for every \(p\ge2\),
\begin{equation}\label{eq:fsigma-moment}
\E\left[
\|f_s^m\|_{\cH}^p+
\|\sigma_s^m\|_{\mathscr L_2(H_Q,\cH)}^p
\right]\le C(p),\qquad s\in[t_i,t].
\end{equation}
\begin{remark}\label{rem:Aterm}
To prove \eqref{eq:Psiiel}, one would apply It\^o's formula to
\(\Psi_i^{(\ell)}(s,X_s,Y_s^m)\), which, according to
\eqref{eq:dTheta}, would require an \(N\)-uniform bound for
$
\big[(\partial_s+\mathscr L_{m,s})\Psi_i^{(\ell)}\big]\allowbreak
(s,X_s,Y_s^m)$.
For example, estimating the time derivative and generator term
separately for \(\Psi_i^{(1)}\) produces the term
\[
\big\langle A\phi_F(X_{t_i}),
\partial_yv(s,X_s,Y_s^m)\big\rangle.
\]
Estimate \eqref{eq:frozen-directions} controls
\(\phi_F(X_{t_i})\) only in \(\cH^{\alpha-1}\), so, when
\(\alpha<2\), no \(N\)-uniform bound is available for
\(A\phi_F(X_{t_i})\) in \(\cH\).
Hence, the terms \(\partial_s\Psi_i^{(\ell)}\) and
\(\mathscr L_{m,s}\Psi_i^{(\ell)}\) cannot be estimated separately
with constants independent of \(N\).
\end{remark}
To address the difficulty in Remark~\ref{rem:Aterm}, we transport
the coefficient directions backward along the wave group to cancel
the terms arising from the additional application of $A$.
For \(t\in[t_i,t_{i+1}]\) and $s\in[t_i,t]$, 
 we introduce
backward-transported coefficient directions:
\begin{equation}\label{eq:transported-fields}
\left\{
\begin{aligned}
&\Sigma_{s,t}(x)
:=S_{s-t}\Sigma(S_{t-s}x),\\
&\phi_{\Sigma,s,t}(x)
:=S_{s-t}\phi_\Sigma(S_{t-s}x),\\
&\Phi_{\Sigma,s,t}(x,y)
:=S_{s-t}\Phi_\Sigma(S_{t-s}x,S_{t-s}y).
\end{aligned}
\right.
\end{equation}
For \(j\in\{1,2\}\), the chain rule yields
\begin{align*}
D^j\Sigma_{s,t}(x)(q_1,\ldots,q_j)
&=S_{s-t}D^j\Sigma(S_{t-s}x)
   (S_{t-s}q_1,\ldots,S_{t-s}q_j),\\
D^j\phi_{\Sigma,s,t}(x)(q_1,\ldots,q_j)
&=S_{s-t}D^j\phi_\Sigma(S_{t-s}x)
   (S_{t-s}q_1,\ldots,S_{t-s}q_j).
\end{align*}
Likewise, for \(a+b\le2\),
\begin{align*}
&D_x^aD_y^b\Phi_{\Sigma,s,t}(x,y)
 (q_1,\ldots,q_a;k_1,\ldots,k_b)\\
&\quad=S_{s-t}D_x^aD_y^b\Phi_\Sigma
 (S_{t-s}x,S_{t-s}y)
 (S_{t-s}q_1,\ldots,S_{t-s}q_a;
  S_{t-s}k_1,\ldots,S_{t-s}k_b).
\end{align*}
Since \(S_r\) is an isometry in every ambient space used here,
\eqref{eq:DG1}, \eqref{eq:DjG0}, \eqref{eq:DjG1},
\eqref{eq:phi-Phi-derivative-bounds}, \eqref{eq:D2phi-bound}, and
\eqref{eq:Phi-coordinate-bound} imply, for \(j=1,2\) and
\(1\le a+b\le2\),
\begin{align}\label{eq:transported-coefficient-bounds}
&\|D^j\Sigma_{s,t}(x)(q_1,\ldots,q_j)\|_
 {\mathscr L_2(H_Q,\cH^1)}
\le C\prod_{\nu=1}^j\|q_\nu\|_{\cH^1},\notag\\
&\|D^j\phi_{\Sigma,s,t}(x)(q_1,\ldots,q_j)\|_
 {\mathscr L_2(H_Q,\cH)}
\le C(1+\|x\|_{\cH^1})
\prod_{\nu=1}^j\|q_\nu\|_{\cH^1},\notag\\
&\|D_x^aD_y^b\Phi_{\Sigma,s,t}(x,y)
(q_1,\ldots,q_a;k_1,\ldots,k_b)\|_
 {\mathscr L_2(H_Q,\cH)}
\notag\\
&\qquad\le C(1+\|x\|_{\cH^1}+\|y\|_{\cH})^{
 \mathbf 1_{\{a\ge1\}}}
\prod_{\nu=1}^a\|q_\nu\|_{\cH^1}
\prod_{\mu=1}^b\|k_\mu\|_{\cH}.
\end{align}
The constant \(C\) is independent of \(N\), \(s\), and \(t\).
The zero-order diffusion direction also satisfies, by
\eqref{eq:Galerkin-lin} and the isometry of the group,
\[
\|\Sigma_{s,t}(x)\|_{\mathscr L_2(H_Q,\cH^1)}
\le C(1+\|x\|_{\cH}).
\]
To estimate the functions defined in
\eqref{eq:frozen-functions}, we introduce
the following auxiliary functions:
\begin{equation}\label{eq:transported-functions}
\left\{
\begin{aligned}
\widehat\Psi_i^{(1)}(s,x,y;t)
&:=\langle S_{s-t}\phi_F(X_{t_i}),
\partial_yv(s,x,y)\rangle,\\
\widehat\Psi_i^{(2)}(s,x,y;t)
&:=\mathrm{tr}_Q\big(
\Sigma_{s,t}(x)^*\partial_{xy}v(s,x,y)
S_{s-t}\phi_\Sigma(X_{t_i})\big),\\
\widehat\Psi_i^{(3)}(s,x,y;t)
&:=\mathrm{tr}_Q\big(
\big(S_{s-t}\phi_\Sigma(X_{t_i})\big)^*
\partial_{yy}v(s,x,y)
\Phi_{\Sigma,s,t}(x,y)\big),\\
\widehat\Psi_i^{(4)}(s,x,y;t)
&:=\mathrm{tr}_Q\big(
\big(S_{s-t}\phi_\Sigma(X_{t_i})\big)^*
\partial_{yy}v(s,x,y)
S_{s-t}\phi_\Sigma(X_{t_i})\big),\\
\widehat\Psi_i^{(5)}(s,x,y;t)
&:=\mathrm{tr}_Q\big(
\phi_{\Sigma,s,t}(x)^*\partial_{yy}v(s,x,y)
\phi_{\Sigma,s,t}(x)\big).
\end{aligned}
\right.
\end{equation}
At \(s=t\), these are exactly the five functions in
\eqref{eq:frozen-functions}.
Moreover, the difference $\widehat\Psi_i^{(\ell)}(t_i,X_{t_i},Y_{t_i}^m;t)-\Psi_i^{(\ell)}(t_i,X_{t_i},Y_{t_i}^m)
$ at \(s=t_i\)
is \(\mathscr F_{t_i}\)-measurable.  Adding and subtracting
\(\widehat\Psi_i^{(\ell)}(t_i,X_{t_i},Y_{t_i}^m;t)\) yields
\begin{align}\label{eq:transported-increment-decomposition}
&\E\!\left[
\Psi_i^{(\ell)}(t,X_t,Y_t^m)
-\Psi_i^{(\ell)}(t_i,X_{t_i},Y_{t_i}^m)
\mid\mathscr F_{t_i}\right]\notag\\
&=
\widehat\Psi_i^{(\ell)}(t_i,X_{t_i},Y_{t_i}^m;t)
-\Psi_i^{(\ell)}(t_i,X_{t_i},Y_{t_i}^m)\notag\\
&\quad+\E\!\left[
\widehat\Psi_i^{(\ell)}(t,X_t,Y_t^m;t)
-\widehat\Psi_i^{(\ell)}(t_i,X_{t_i},Y_{t_i}^m;t)
\mid\mathscr F_{t_i}\right].
\end{align}
\begin{lemma}\label{lem:transport-evolution}
Under the assumptions of Proposition \ref{prop:one-step}, for
\(\ell=1,\ldots,5\), {there exists \(C>0\), independent
of \(m\) and \(N\), such that, for
\(t_i\le s\le t\le t_{i+1}\) and \(i=0,1,\ldots,m-1\),}
\begin{align}\label{eq:transported-generator-bound}
\left\|
\big[(\partial_s+\mathscr L_{m,s})
\widehat\Psi_i^{(\ell)}\big](s,X_s,Y_s^m;t)
\right\|_{L^1(\Omega)}
\le C.
\end{align}
\end{lemma}

\begin{proof}
Set $\mathcal L^0=\mathcal L_x+\mathcal L_y-\mathcal L^A$, where
$
(\mathcal L^A\Theta)(x,y)
=D_x\Theta(x,y)Ax+D_y\Theta(x,y)Ay.
$
Then
\begin{equation}\label{eq:ps}
\partial_sv+\mathcal L^Av+\mathcal L^0v=0.
\end{equation}
Direct differentiation of
\eqref{eq:transported-fields} yields
\begin{align}
(\partial_s+\mathcal L^A)
\big(S_{s-t}\phi_G(X_{t_i})\big)
&=AS_{s-t}\phi_G(X_{t_i}),
\qquad G\in\{F,\Sigma\},\notag\\
(\partial_s+\mathcal L^A)\Sigma_{s,t}(x)
&=A\Sigma_{s,t}(x),\notag\\
(\partial_s+\mathcal L^A)\phi_{\Sigma,s,t}(x)
&=A\phi_{\Sigma,s,t}(x),\notag\\
(\partial_s+\mathcal L^A)\Phi_{\Sigma,s,t}(x,y)
&=A\Phi_{\Sigma,s,t}(x,y).
\label{eq:alternative-transport-identities}
\end{align}

For the first identity below take \(k\in\cH_N\). For the second take
\(q\in\cH_N\), measured in the ambient \(\cH^1\)-norm, and
\(k\in\cH_N\), measured in the ambient \(\cH\)-norm. For the third
take \(q,k\in\cH_N\), both measured in the ambient \(\cH\)-norm.
Differentiating the equation \eqref{eq:ps} in these
directions yields
\begin{align}
(\partial_s+\mathcal L^A)
\langle k,\partial_yv\rangle
&=-\langle Ak,\partial_yv\rangle
-\langle k,D_y(\mathcal L^0v)\rangle,\notag\\
(\partial_s+\mathcal L^A)
\partial_{xy}v(q,k)
&=-\partial_{xy}v(Aq,k)-\partial_{xy}v(q,Ak)
-D_{xy}(\mathcal L^0v)(q,k),\notag\\
(\partial_s+\mathcal L^A)
\partial_{yy}v(q,k)
&=-\partial_{yy}v(Aq,k)-\partial_{yy}v(q,Ak)
-D_{yy}(\mathcal L^0v)(q,k).
\label{eq:alternative-differentiated-Kolmogorov}
\end{align}
Here all functions in
\eqref{eq:alternative-differentiated-Kolmogorov} are evaluated at
\((s,x,y)\). Combining
\eqref{eq:alternative-transport-identities} with
\eqref{eq:alternative-differentiated-Kolmogorov} and substituting
the resulting identities into the five definitions in
\eqref{eq:transported-functions}, we obtain
\begin{align}
(\partial_s+\mathcal L^A)\widehat\Psi_i^{(1)}
&=-\left\langle S_{s-t}\phi_F(X_{t_i}),
D_y(\mathcal L^0v)\right\rangle,\notag\\
(\partial_s+\mathcal L^A)\widehat\Psi_i^{(2)}
&=-\mathrm{tr}_Q\!\left(
\Sigma_{s,t}(x)^*D_{xy}(\mathcal L^0v)
S_{s-t}\phi_\Sigma(X_{t_i})\right),\notag\\
(\partial_s+\mathcal L^A)\widehat\Psi_i^{(3)}
&=-\mathrm{tr}_Q\!\left(
(S_{s-t}\phi_\Sigma(X_{t_i}))^*
D_{yy}(\mathcal L^0v)\Phi_{\Sigma,s,t}(x,y)\right),\notag\\
(\partial_s+\mathcal L^A)\widehat\Psi_i^{(4)}
&=-\mathrm{tr}_Q\!\left(
(S_{s-t}\phi_\Sigma(X_{t_i}))^*
D_{yy}(\mathcal L^0v)S_{s-t}\phi_\Sigma(X_{t_i})\right),\notag\\
(\partial_s+\mathcal L^A)\widehat\Psi_i^{(5)}
&=-\mathrm{tr}_Q\!\left(
\phi_{\Sigma,s,t}(x)^*
D_{yy}(\mathcal L^0v)\phi_{\Sigma,s,t}(x)\right).
\label{eq:alternative-five-cancellations}
\end{align}
The arguments of
\(\widehat\Psi_i^{(\ell)}\) in
\eqref{eq:alternative-five-cancellations} are \((s,x,y;t)\), and
those of the derivatives of \(\mathcal L^0v\) are \((s,x,y)\).
By \eqref{eq:Lx} and \eqref{eq:Ly},
\begin{align}\label{eq:L0v}\notag
\mathcal L^0v
&=D_xv\,F(x)+D_yv\,\Phi_F(x,y)
+\frac12\mathrm{tr}_Q\big(
\Sigma(x)^*D_{xx}v\,\Sigma(x)\big)\\\notag
&\quad+\mathrm{tr}_Q\big(
\Sigma(x)^*D_{xy}v\,\Phi_\Sigma(x,y)\big)
+\frac12\mathrm{tr}_Q\big(
\Phi_\Sigma(x,y)^*D_{yy}v\,\Phi_\Sigma(x,y)\big)\\
&\quad+\frac{T^2}{24}\mathrm{tr}_Q\big(
\phi_\Sigma(x)^*D_{yy}v\,\phi_\Sigma(x)\big),
\end{align}
where every derivative of \(v\) is evaluated at \((s,x,y)\).
Applying $D_y$, $D_{xy}$, and $D_{yy}$ to
\eqref{eq:L0v} and using the product rule, the only
fourth-order derivatives of $v$ that arise are $
D_{xxxy}v, D_{xxyy}v,
D_{xyyy}v$, and $D_{yyyy}v$.
These derivatives of $v$ are controlled uniformly in $N$ by
Proposition~\ref{prop:varphiY}.
For example, for $q\in\cH_N$ measured in $\cH^1$
and $k\in\cH_N$ measured in $\cH$, we have
\begin{align*}
&D_{xy}\Big[
\frac12\mathrm{tr}_Q\big(
\Sigma(x)^*D_{xx}v\,\Sigma(x)\big)
\Big](q,k)\\
&=\frac12\sum_{n=1}^{\infty}
D_{xxxy}v(s,x,y)
\big(q,\Sigma(x)Q^{1/2}e_n,
\Sigma(x)Q^{1/2}e_n;k\big)\\
&\quad+\sum_{n=1}^{\infty}
D_{xxy}v(s,x,y)
\big((D\Sigma(x)q)Q^{1/2}e_n,
\Sigma(x)Q^{1/2}e_n;k\big).
\end{align*}
By \eqref{eq:Galerkin-lin} and \eqref{eq:DG1},
the coefficient directions
$\Sigma(x)Q^{1/2}e_n$ and
$(D\Sigma(x)q)Q^{1/2}e_n$
have square-summable $\cH^1$-norms.
Thus all $x$-directions in these contractions
are measured in $\cH^1$, while the $y$-direction
$k$ is measured in $\cH$.
The remaining terms have the same structure:
$F(x)$, $\Sigma(x)Q^{1/2}e_n$, and their
coefficient derivatives provide $\cH^1$-valued
$x$-directions, whereas $\Phi_F(x,y)$,
$\Phi_\Sigma(x,y)Q^{1/2}e_n$,
$\phi_\Sigma(x)Q^{1/2}e_n$, and their
coefficient derivatives provide $\cH$-valued
$y$-directions.
The coefficient differentiations involved here require at most
second derivatives of $\phi_G$ and $\Phi_G$, and hence derivatives
of $f$ and $\sigma$ of order at most three.

Combining Proposition~\ref{prop:varphiY} with
\eqref{eq:Galerkin-lin}, \eqref{eq:DG1}, \eqref{eq:DjG0},
\eqref{eq:Phi-growth}, \eqref{eq:phi-Phi-derivative-bounds},
\eqref{eq:D2phi-bound}, and \eqref{eq:Phi-coordinate-bound}
shows that there exist $C>0$ and $r\ge0$, independent of $N$,
such that
\begin{align*}
|\langle k,D_y(\mathcal L^0v)(s,x,y)\rangle|
&\le C(1+\|x\|_{\cH^1}+\|y\|_{\cH})^r\|k\|_{\cH},\\
|D_{xy}(\mathcal L^0v)(s,x,y)(q,k)|
&\le C(1+\|x\|_{\cH^1}+\|y\|_{\cH})^r
\|q\|_{\cH^1}\|k\|_{\cH},\\
|D_{yy}(\mathcal L^0v)(s,x,y)(q,k)|
&\le C(1+\|x\|_{\cH^1}+\|y\|_{\cH})^r
\|q\|_{\cH}\|k\|_{\cH}.
\end{align*}
Consequently, \eqref{eq:transported-coefficient-bounds},
\eqref{eq:Phi-growth}, \eqref{eq:Galerkin-lin},
\eqref{eq:frozen-directions}, \eqref{eq:LpX-Galerkin}, and
\eqref{Strongorder-Galerkin} yield
\begin{equation}\label{eq:alternative-linear-bound}
\big\|
\big[(\partial_s+\mathcal L^A)\widehat\Psi_i^{(\ell)}\big]
(s,X_s,Y_s^m;t)
\big\|_{L^1(\Omega)}
\le C.
\end{equation}

By \eqref{eq:local-Ito-operator} and
$
(\mathcal L^A\Theta)(x,y)
=D_x\Theta(x,y)Ax+D_y\Theta(x,y)Ay
$, we have
\begin{align*}
&\big((\mathscr L_{m,s}-\mathcal L^A)\Theta\big)(s,x,y)
=D_x\Theta(s,x,y)F(x)+D_y\Theta(s,x,y)f_s^m\\
&\quad+\frac12\mathrm{tr}_Q\!\left(
\Sigma(x)^*D_{xx}\Theta(s,x,y)\Sigma(x)\right)
+\mathrm{tr}_Q\!\left(
\Sigma(x)^*D_{xy}\Theta(s,x,y)\sigma_s^m\right)\\
&\quad+\frac12\mathrm{tr}_Q\!\left(
(\sigma_s^m)^*D_{yy}\Theta(s,x,y)\sigma_s^m\right).
\end{align*}
By \eqref{eq:LpX-Galerkin}, \eqref{Strongorder-Galerkin}, and
\eqref{eq:fsigma-moment}, for every \(p\ge2\),
\[
\|X_s\|_{L^p(\Omega,\cH^1)}
+\|Y_s^m\|_{L^p(\Omega,\cH)}
+\|f_s^m\|_{L^p(\Omega,\cH)}
+\|\sigma_s^m\|_{L^p(\Omega,\mathscr L_2(H_Q,\cH))}
\le C_p.
\]
Estimate \eqref{eq:transported-coefficient-bounds},
\eqref{eq:frozen-directions}, \eqref{eq:Galerkin-lin},
\eqref{eq:Phi-growth},
\eqref{eq:fsigma-moment}, and Proposition \ref{prop:varphiY}
therefore yield
\begin{equation}\label{eq:alternative-remainder-bound}
\big\|
\big[(\mathscr L_{m,s}-\mathcal L^A)
\widehat\Psi_i^{(\ell)}\big](s,X_s,Y_s^m;t)
\big\|_{L^1(\Omega)}
\le C.
\end{equation}
Finally,
$
\partial_s+\mathscr L_{m,s}
=(\partial_s+\mathcal L^A)
+(\mathscr L_{m,s}-\mathcal L^A),
$
so \eqref{eq:alternative-linear-bound} and
\eqref{eq:alternative-remainder-bound} prove
\eqref{eq:transported-generator-bound}.
\end{proof}
The following lemma compares the transported functions
\(\widehat\Psi_i^{(\ell)}\) with the original functions
\(\Psi_i^{(\ell)}\) at \(s=t_i\).
\begin{lemma}[Comparison at \(s=t_i\)]\label{lem:transport-comparison}
Under the assumptions of Proposition \ref{prop:one-step}, for
\(\ell=1,\ldots,5\),
\begin{align}\label{eq:transported-function-comparison}
\big\|\widehat\Psi_i^{(\ell)}(t_i,X_{t_i},
Y_{t_i}^m;t)
-\Psi_i^{(\ell)}(t_i,X_{t_i},Y_{t_i}^m)
\big\|_{L^1(\Omega)}
\le C(t-t_i)^{\alpha-1}.
\end{align}
The constant \(C\) is independent of \(i\), \(m\), and \(N\).
\end{lemma}

\begin{proof}
The estimate follows from \eqref{eq:St-IN},
\eqref{eq:frozen-directions},
\eqref{eq:phi-Phi-derivative-bounds},
Lemmas~\ref{lem:FSigN} and \ref{lem:ACLW},
and Proposition~\ref{prop:varphiY}.
We omit the details for brevity.
\end{proof}

\begin{proof}[Proof of Proposition \ref{prop:one-step}]
The estimates
\eqref{eq:transported-coefficient-bounds},
\eqref{eq:fsigma-moment}, \eqref{eq:frozen-directions}, and
Proposition \ref{prop:varphiY} show that the stochastic integral in
the following application of It\^o's formula is square integrable.
Conditionally on $\mathscr F_{t_i}$, the frozen directions appearing in $\widehat\Psi_i^{(\ell)}$ are deterministic parameters.  Hence the finite-dimensional It\^o formula applies.
Applying \eqref{eq:dTheta} conditionally to
\(\widehat\Psi_i^{(\ell)}(\cdot,\cdot,\cdot;t)\) gives
\begin{align*}
&\E\!\left[
\widehat\Psi_i^{(\ell)}(t,X_t,Y_t^m;t)
-\widehat\Psi_i^{(\ell)}(t_i,X_{t_i},Y_{t_i}^m;t)
\mid\mathscr F_{t_i}\right]\\
&\quad=\int_{t_i}^t
\E\!\left[
\big[(\partial_s+\mathscr L_{m,s})
\widehat\Psi_i^{(\ell)}\big](s,X_s,Y_s^m;t)
\mid\mathscr F_{t_i}
\right]\ud s.
\end{align*}
Taking expectations in this identity and using the tower property, it follows from 
Lemma \ref{lem:transport-evolution} that
\begin{align*}
&\left|\E\!\left[
\widehat\Psi_i^{(\ell)}(t,X_t,Y_t^m;t)
-\widehat\Psi_i^{(\ell)}(t_i,X_{t_i},Y_{t_i}^m;t)\right]\right|
\le C(t-t_i).
\end{align*}
Since
\(\widehat\Psi_i^{(\ell)}(t,\cdot,\cdot;t)
=\Psi_i^{(\ell)}(t,\cdot,\cdot)\) and $\alpha\in(1,2]$,
taking expectations in
\eqref{eq:transported-increment-decomposition}, and using the tower
property, we obtain that for each $\ell=1,\ldots,5$,
\begin{align*}
&\left|\E\left[
\Psi_i^{(\ell)}(t,X_t,Y_t^m)
-\Psi_i^{(\ell)}(t_i,X_{t_i},Y_{t_i}^m)\right]\right|\\
&\quad\le C(t-t_i)
+\left\|
\widehat\Psi_i^{(\ell)}(t_i,X_{t_i},Y_{t_i}^m;t)
-\Psi_i^{(\ell)}(t_i,X_{t_i},Y_{t_i}^m)
\right\|_{L^1(\Omega)}\le C(t-t_i)^{\alpha-1}.
\end{align*}
Applying Lemma~\ref{lem:transport-comparison} yields
\eqref{eq:Psiiel} and completes the proof of
Proposition~\ref{prop:one-step}.
\end{proof}

\section{Numerical experiment}\label{S:numerical-experiments}
We illustrate the first-order weak convergence and the improvement obtained
by Richardson extrapolation for a nonlinear stochastic wave equation.  To be specific, we consider the stochastic sine--Gordon equation
\begin{equation}\label{eq:numerical-sine-gordon-compact}
  \begin{aligned}
    \partial_{tt}u(t,x)
      &=\partial_{xx}u(t,x)-8\sin\bigl(u(t,x)\bigr)+\dot W^{Q}(t,x),\quad &&0<t\le T,\quad 0<x<1,\\
    u(t,0)&=u(t,1)=0,
        &&0\le t\le T,\\
    u(0,\cdot)&=0.5\,\eta_1,
        \partial_tu(0,\cdot)=0,\quad&& T=1.
  \end{aligned}
\end{equation}
Here
$\eta_k(x)=\sqrt{2}\sin(k\pi x)$ for $k\ge1$, and 
$W^{Q}(t,x)=\sum_{j=1}^{\infty}\sqrt{q_j}\,\beta_j(t)\eta_j(x)$ with 
  $q_j=j^{-4}$,
where \((\beta_j)_{j\ge1}\) are independent standard Brownian motions.
In this setting, Assumptions~\ref{asp:F}--\ref{asp:fg3} are fulfilled with
\(\alpha=2\). For \(z=(z_1,z_2)\), we use the bounded smooth functional
\[
  \psi(z)=\tanh\!\left(
    0.1\,\langle z_1,\eta_1\rangle_{L^2(0,1)}
  \right),
  \qquad
\]

We discretize the equation \eqref{eq:numerical-sine-gordon-compact} in space by the spectral Galerkin method
with \(N\) sine modes and truncate the noise to \(J\) modes.
The resulting system is integrated in time using the exponential
Euler scheme with step size \(\tau=T/m\).
To couple the simulations at different time resolutions, we generate
the Brownian increments for each noise mode on a common finest grid
with \(m_{\mathrm{ref}}=3072\) time steps.  The Brownian increments on each coarser time grid are obtained by
summing the fine-grid increments over the corresponding time intervals.
For a realization \(\omega\) of the truncated Wiener process
\(W^{Q,J}=\sum_{j=1}^J\sqrt{q_j}\,\beta_j(t)\eta_j(x)\), let \(\widetilde X_T^\ell(\omega)\) denote the numerical
solution at time \(T\) obtained with \(\ell\) time steps.  For this
realization, we use the fine-grid Richardson value
\[
  \Psi_{\mathrm{ref}}(\omega)
  =
  2\psi\bigl(\widetilde X_T^{m_{\mathrm{ref}}}(\omega)\bigr)
  -
  \psi\bigl(\widetilde X_T^{m_{\mathrm{ref}}/2}(\omega)\bigr)
\]
as the reference value.
The signed differences for the ordinary and Richardson weak approximations
are, respectively,
\[
  \begin{aligned}
    D_m^{\mathrm{ord}}(\omega)
      &=
      \psi\bigl(\widetilde X_T^{3m}(\omega)\bigr)
      -
      \Psi_{\mathrm{ref}}(\omega),\\
    D_m^{\mathrm{Rich}}(\omega)
      &=
      2\psi\bigl(\widetilde X_T^{2m}(\omega)\bigr)
      -
      \psi\bigl(\widetilde X_T^m(\omega)\bigr)
      -
      \Psi_{\mathrm{ref}}(\omega).
  \end{aligned}
\]
\begin{figure}[t]
\centering
\begin{minipage}[t]{0.49\textwidth}
    \centering
    \includegraphics[width=\linewidth]
    {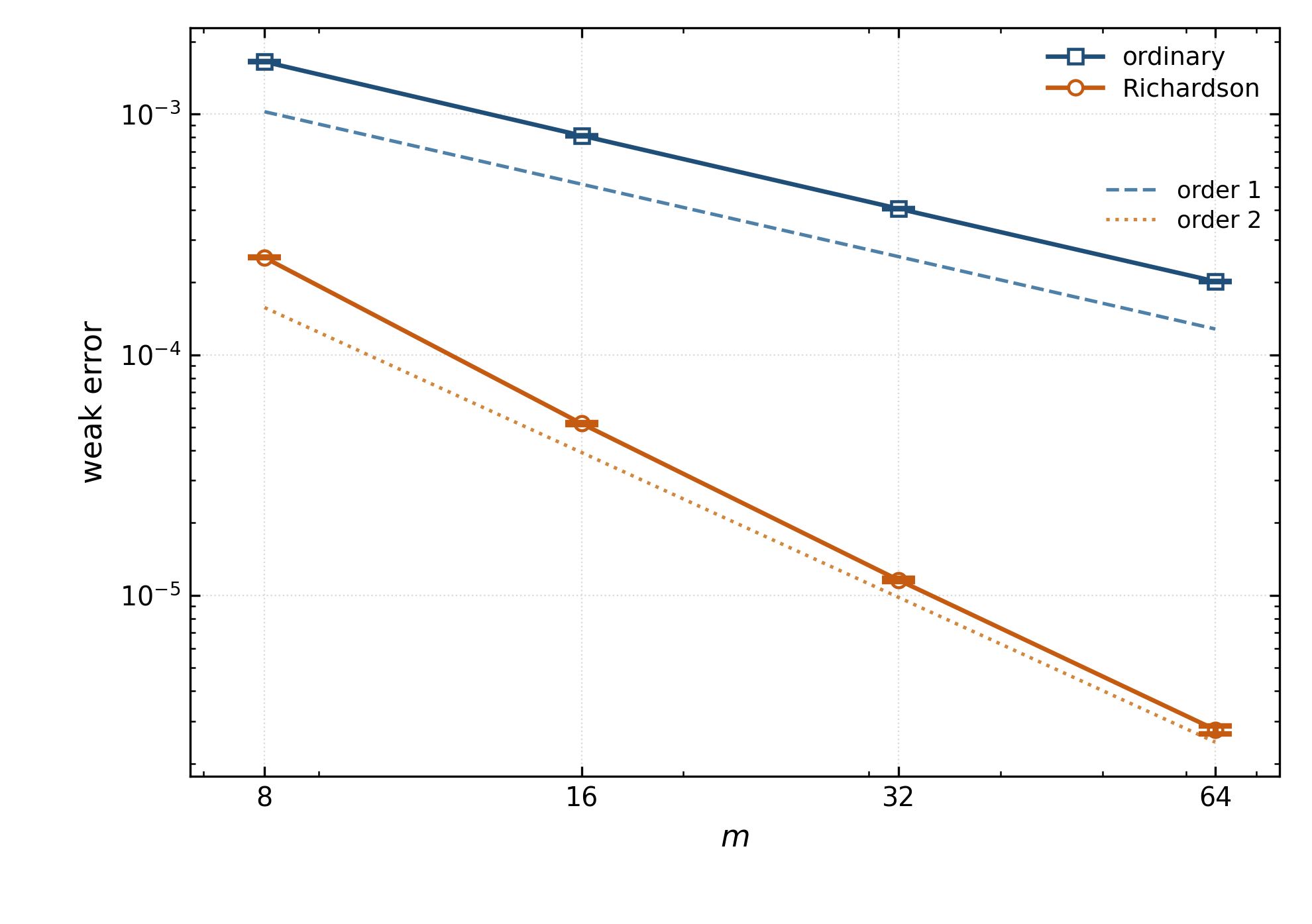}
\end{minipage}
\hfill
\begin{minipage}[t]{0.49\textwidth}
    \centering
    \includegraphics[width=\linewidth]
    {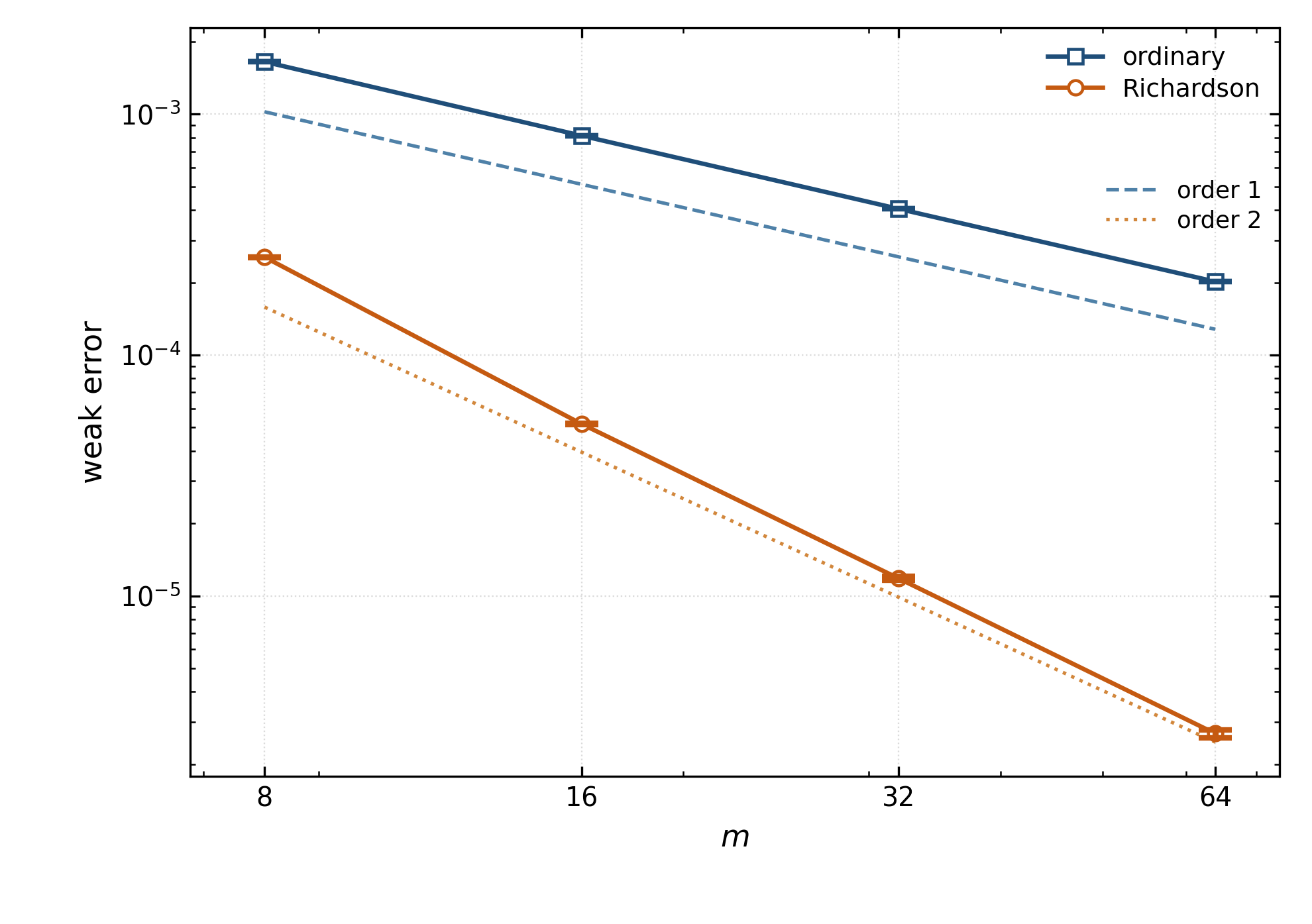}
\end{minipage}
\caption{Weak errors of the ordinary (blue squares) and
Richardson (orange circles) approximations for
$N=J=16$ (left) and $N=J=32$ (right).
Dashed and dotted lines indicate orders one and two,
respectively, while error bars represent 95\%
Monte Carlo confidence intervals.}
\label{fig:sine-gordon-ci-compact}
\end{figure}

To reduce the sampling variance, we use antithetic variates.  We generate
\(M_{\mathrm{pair}}=60{,}000\) independent realizations
\(\omega_i\) of the truncated Wiener process.  For each realization
\(\omega_i\), we evaluate the numerical scheme twice: once with the
original Brownian increments and once with their negatives, using the
same initial condition and all other parameters.  For
\(\star\in\{\mathrm{ord},\mathrm{Rich}\}\), the corresponding pair
observation is
\[
  Z_{m,i}^{\star}
  =
  \frac12\left[
    D_m^{\star}(\omega_i)
    +
    D_m^{\star}(-\omega_i)
  \right].
\]
Let \(b_m^\star\) and \(s_m^\star\) denote the sample mean and sample
standard deviation of the pair observations
\(\{Z_{m,i}^\star\}_{i=1}^{M_{\mathrm{pair}}}\), respectively.  Thus,
the estimated signed temporal error is \(b_m^\star\), and the
corresponding 95\% Monte Carlo confidence interval has half-width
\(
  1.96\,\frac{s_m^\star}{\sqrt{M_{\mathrm{pair}}}}.
\)
Figure~\ref{fig:sine-gordon-ci-compact} displays the absolute temporal
error estimates \(\lvert b_m^\star\rvert\) relative to the fine-grid
reference for
\(\star\in\{\mathrm{ord},\mathrm{Rich}\}\) and
\(m\in\{8,16,32,64\}\). Increasing the spatial and noise truncation levels
from $N=J=16$ to $N=J=32$ yields nearly identical
temporal weak-error estimates and observed convergence
rates.
Figure~\ref{fig:sine-gordon-ci-compact} confirms the expected first- and second-order weak convergence
rates of the ordinary and Richardson approximations, respectively, with
Richardson extrapolation yielding substantially smaller errors at the same time-stepping cost.

%
\appendix
\section{Proof of Proposition~\ref{prop:varphiY}}\label{S:auxiliary-proofs}

Throughout this section, we continue to fix \(N\) and use
the convention introduced at the beginning of Section
\ref{S:proof}; in particular, all Galerkin superscripts \((N)\) are
suppressed and all generic constants are independent of \(N\).
In this section, we prove Proposition~\ref{prop:varphiY},
which establishes regularity estimates
for the solution to the
finite-dimensional Kolmogorov equation \eqref{eq:Kolu}.

This solution of \eqref{eq:Kolu} is represented by \eqref{eq:utx} in terms of
the process $(\mathbb X,\mathbb Y)$. Recall that $(\mathbb X,\mathbb Y)\triangleq (\mathbb X^{(N)},\mathbb Y^{(N)})$
denotes the solution to the coupled equations
\eqref{eq:Galerkin-systemX} and \eqref{eq:Galerkin-systemY}
with arbitrary initial data
$(x,y)\in\cH_N\times\cH_N$.
The moment
bounds for \((\mathbb X,\mathbb Y)\) are recorded at the beginning of
the proof of Lemma \ref{lem:flow-variations}, which establishes the
variational estimates for
\((\mathbb X,\mathbb Y)\) with respect to the initial data.  The proof
uses projected linear stochastic evolution equations of the form
\eqref{eq:Z}, for which we first establish the following estimate.
\begin{lemma}\label{lem:Z}
Assume that
Assumption \ref{asp:F} holds and that the functions
$f$ and 
$\sigma$ are continuously differentiable with bounded derivatives.
Furthermore, assume that $\widetilde F$,
$\{\widetilde\Sigma^{i}\}_{i=1}^2$ are
predictable $\cH_N$- and
$\mathscr{L}_2(H_Q,\cH_N)$-valued processes, and let $Z$ solve
\begin{align}\label{eq:Z}
\ud Z_t=\big(AZ_t
+DF(\mathbb X_t)Z_t+\widetilde F_t\big)\ud t+\big(D\Sigma(\mathbb X_t)Z_t
+\widetilde\Sigma_t^{1}\big)\ud W_t
+\widetilde\Sigma_t^{2}\ud\widetilde W_t
\end{align}
for $t\in(0,T]$.
Fix \(r\in\{0,1\}\) and \(p\ge2\), and assume that
\(Z_0\in L^p(\Omega,\cH^r)\).  
Then there exists $C:=C(p,T)$, independent of
$N$, such that for any $t\in[0,T]$,
\begin{align*}
\E\Big[\|Z\|_{L^\infty(0,t;\mathcal{H}^r)}^p\Big]
\le C \E\|Z_0\|_{\mathcal{H}^r}^p+C\int_0^t\E\Big[
\|\widetilde F_s\|_{\cH^r}^p
+\sum_{i=1}^2\|\widetilde\Sigma_s^{i}\|_{\mathscr{L}_2(H_Q,\cH^r)}^p
\Big]\ud s,
\end{align*}
provided that the integral on the right hand side is finite.
\end{lemma}
\begin{proof}
By the mild form of \eqref{eq:Z}, \eqref{eq:convolution-bounds} and \eqref{eq:DG1}, we obtain that for $r=0,1$,
\begin{align*}
\E\sup_{0\le u\le t}\|Z_u\|_{\cH^r}^p
&\le C\E\|Z_0\|_{\cH^r}^p
+C\int_0^t
\E\sup_{0\le v\le s}\|Z_v\|_{\cH^r}^p\ud s\\
&\quad+C\int_0^t\E\Big[
\|\widetilde F_s\|_{\cH^r}^p
+\sum_{i=1}^2
\|\widetilde\Sigma_s^{i}\|_{\mathscr L_2(H_Q,\cH^r)}^p
\Big]\ud s .
\end{align*}
Finally, applying Gronwall's inequality yields the assertion.
\end{proof}

\begin{lemma}[Uniform variation estimates]\label{lem:flow-variations}
Let Assumptions \ref{asp:F} and \ref{asp:fg3} hold.  For every
\(p\ge2\), there exists \(C:=C(p,T)>0\), independent of \(N\), such
that, for \(x,y\in\cH_N\), \(1\le a\le3\), and
\(\ell_1,\ldots,\ell_a\in\cH_N\),
\begin{align}
\|D_x^a\mathbb X_t(\ell_1,\ldots,\ell_a)\|_
 {L^p(\Omega,\cH^1)}
&\le C\prod_{j=1}^a\|\ell_j\|_{\cH^1},\quad
t\in[0,T].
\label{eq:X-variation-lemma}
\end{align}
Moreover, if \(0\le a\le3\), \(b\in\{0,1\}\), and
\(1\le a+b\le4\), then, for
\(\ell_1,\ldots,\ell_a,h_1,\ldots,h_b\in\cH_N\) and
\(t\in[0,T]\),
\begin{equation}
\begin{aligned}
&\|D_x^aD_y^b\mathbb Y_t
(\ell_1,\ldots,\ell_a;h_1,\ldots,h_b)\|_
 {L^p(\Omega,\cH)}\\
&\quad\le C(1+\|x\|_{\cH^1}+\|y\|_{\cH})^a
\prod_{j=1}^a\|\ell_j\|_{\cH^1}
\prod_{j=1}^b\|h_j\|_{\cH}.
\label{eq:Y-variation-lemma}
\end{aligned}
\end{equation}
Every derivative of \(\mathbb X\) containing
a \(y\)-direction and every derivative of \(\mathbb Y\) containing
at least two \(y\)-directions vanish.
\end{lemma}
\begin{proof}
We first record the \(\cH^1\)-valued Nemytskii derivative estimate
needed below: for \(j\in\{2,3,4\}\) and
\(z,h^1,\ldots,h^j\in\cH_N\),
\begin{equation}\label{eq:DjG1}
\|D^jF(z)(h^1,\ldots,h^j)\|_{\cH^1}
+\|D^j\Sigma(z)(h^1,\ldots,h^j)\|_
 {\mathscr L_2(H_Q,\cH^1)}\le C\|h_1^1\|_H
\prod_{\ell=2}^j\|h_1^\ell\|_{H^1}.
\end{equation}
This inequality follows directly from \eqref{eq:projected-DjG},
\eqref{eq:projected-Nemytskii-derivatives},
\(H^1(\mathcal O)\hookrightarrow L^\infty(\mathcal O)\), and
Assumption \ref{asp:F}(2).
The estimate \eqref{eq:DjG1} remains valid after any permutation of the direction
arguments.
Fix \(G\in\{F,\Sigma\}\) and \(z,\xi\in\cH_N\). For directions
\(q_1,q_2,q_3,k\in\cH_N\), differentiating
\eqref{eq:projected-Phi-identities} and using
\eqref{eq:phi-Phi-derivative-bounds}, we obtain the first-derivative
identities and estimates. Differentiating
\(\phi_G\) once and twice more yields
\begin{align*}
D^2\phi_G(z)(q_1,q_2)
&=AD^2G(z)(q_1,q_2)-D^3G(z)(q_1,q_2,Az)\\
&\quad-D^2G(z)(Aq_1,q_2)-D^2G(z)(q_1,Aq_2),\\
D^3\phi_G(z)(q_1,q_2,q_3)
&=AD^3G(z)(q_1,q_2,q_3)
-D^4G(z)(q_1,q_2,q_3,Az)\\
&\quad-\sum_{\nu=1}^3
D^3G(z)(q_1,\ldots,Aq_\nu,\ldots,q_3).
\end{align*}
By \eqref{eq:Abound}, \eqref{eq:DG1}, \eqref{eq:DjG0}, and
\eqref{eq:DjG1}, these identities imply
\begin{align}
&\|D^2\phi_F(z)(q_1,q_2)\|_{\cH}
+\|D^2\phi_\Sigma(z)(q_1,q_2)\|_{\mathscr L_2(H_Q,\cH)}\notag\\
&\quad\le C(1+\|z\|_{\cH^1})
\prod_{\nu=1}^2\|q_\nu\|_{\cH^1},
\label{eq:D2phi-bound}\\
&\|D^3\phi_F(z)(q_1,q_2,q_3)\|_{\cH}
+\|D^3\phi_\Sigma(z)(q_1,q_2,q_3)\|_{\mathscr L_2(H_Q,\cH)}\notag\\
&\quad\le C(1+\|z\|_{\cH^1})
\prod_{\nu=1}^3\|q_\nu\|_{\cH^1}.
\label{eq:D3phi-bound}
\end{align}
Since the function
\(\Phi_G(z,\xi)=DG(z)\xi+\frac T2\phi_G(z)\) is affine in $\xi$, we have
\begin{gather*}
D_{zz}\Phi_G(z,\xi)(q_1,q_2)
=D^3G(z)(q_1,q_2,\xi)
+\frac T2D^2\phi_G(z)(q_1,q_2),\\
D_{z\xi}\Phi_G(z,\xi)(q_1,k)
=D^2G(z)(q_1,k),\qquad D_{\xi\xi}\Phi_G(z,\xi)=0,
\\
D_{zzz}\Phi_G(z,\xi)(q_1,q_2,q_3)
=D^4G(z)(q_1,q_2,q_3,\xi)
+\frac T2D^3\phi_G(z)(q_1,q_2,q_3),
\\
D_{zz\xi}\Phi_G(z,\xi)(q_1,q_2,k)
=D^3G(z)(q_1,q_2,k),
\\
D_{zzz\xi}\Phi_G(z,\xi)(q_1,q_2,q_3,k)
=D^4G(z)(q_1,q_2,q_3,k).
\end{gather*}
By symmetry of the Fréchet derivatives, the above mixed formulas
remain valid after any permutation of the direction
arguments.
All derivatives containing two or more \(\xi\)-directions vanish.
Consequently, if \(0\le a\le3\), \(b\in\{0,1\}\), and
\(1\le a+b\le4\), then
\begin{align}\label{eq:Phi-coordinate-bound}
&\|D_z^aD_\xi^b\Phi_F(z,\xi)
 (q_1,\ldots,q_a;k_1,\ldots,k_b)\|_{\cH}\notag\\
&\quad+\|D_z^aD_\xi^b\Phi_\Sigma(z,\xi)
 (q_1,\ldots,q_a;k_1,\ldots,k_b)\|_
 {\mathscr L_2(H_Q,\cH)}\notag\\
&\qquad\le
C(1+\|z\|_{\cH^1}+\|\xi\|_{\cH})^{\mathbf 1_{\{a\ge1\}}}
\prod_{\mu=1}^a\|q_\mu\|_{\cH^1}
\prod_{\mu=1}^b\|k_\mu\|_{\cH}.
\end{align}
For \(a=0\), we have \(b=1\) and \(D_\xi\Phi_G=DG\), which is
uniformly bounded. For \(a\ge1\), the displayed bound \eqref{eq:Phi-coordinate-bound} allows
linear growth in \((z,\xi)\).
Moreover, by
\eqref{eq:phi-Phi-derivative-bounds}, \eqref{eq:D2phi-bound}, and
\eqref{eq:D3phi-bound}, for
\(1\le a\le3\),
\begin{align}\label{eq:phi-coordinate-bound}
&\|D^a\phi_\Sigma(z)(q_1,\ldots,q_a)\|_
 {\mathscr L_2(H_Q,\cH)}
\le C(1+\|z\|_{\cH^1})
\prod_{\mu=1}^a\|q_\mu\|_{\cH^1}.
\end{align}

To prove \eqref{eq:X-variation-lemma} and
\eqref{eq:Y-variation-lemma}, we next introduce some notation.
Recall that \((x,y)\in\cH_N\times\cH_N\)
denotes the initial data for
\((\mathbb X,\mathbb Y)\).
Let \(I=I_x\mathbin{\cup}I_y\) be a finite collection of directions
in \(\cH_N\), with repeated occurrences of the same direction
treated as distinct entries.
Each \(\ell\in I_x\) specifies differentiation of the flow
with respect to \(x\) in the direction \(\ell\), keeping \(y\) fixed;
each \(h\in I_y\) specifies differentiation with respect to \(y\)
in the direction \(h\), keeping \(x\) fixed.
Assume \(|I_x|\le3\) and \(|I|\le4\), and put
\[
a(I):=|I_x|,\qquad b(I):=|I_y|.
\]For every nonempty \(B\subseteq I\), denote the corresponding
variations by
\[
\mathbb X_t^B:=D_B\mathbb X_t,\qquad
\mathbb Y_t^B:=D_B\mathbb Y_t.
\]
Here \(D_B\) differentiates the flow at \((x,y)\) once in each
labelled direction in \(B\).
Let \(\Pi(I)\) be the set of all partitions of \(I\).
This notation is used in the variational equations below;
see, e.g., \eqref{eq:XtI}.
Denote the product of the direction norms by
\[
\mathcal D_I:=
\prod_{\ell\in I_x}\|\ell\|_{\cH^1}
\prod_{h\in I_y}\|h\|_{\cH}.
\]
Recall that \(\pi\in\Pi(I)\) consists of pairwise disjoint nonempty
blocks \(B\) whose union is \(I\), and that \(|\pi|\) denotes the
number of blocks.
The \(\mathbb X\)-equation \eqref{eq:Galerkin-systemX} is independent of \(y\), whereas the
\(\mathbb Y\)-equation \eqref{eq:Galerkin-systemY} is affine in its initial value \(y\).
Consequently,
\[
\mathbb X_t^I=0\quad\text{if }b(I)\ge1,\qquad
\mathbb Y_t^I=0\quad\text{if }b(I)\ge2.
\]
Here \(b(I)\) counts the directions acting on the initial value \(y\).
For a singleton \(I=\{\ell\}\) containing an \(x\)-direction,
\(\mathbb X_0^I=\ell\) and \(\mathbb Y_0^I=0\).
For a singleton \(I=\{h\}\) with \(h\in I_y\),
\(\mathbb X_0^I=0\) and \(\mathbb Y_0^I=h\).
All variations
of order at least two have zero initial value.

Recall that
\(\mathbb X^I=0\) if \(b(I)\ge1\). When \(b(I)=0\), differentiating the \(\mathbb X\)-equation \eqref{eq:Galerkin-systemX} gives
\begin{align}\label{eq:XtI}
\ud\mathbb X_t^I
={}&\Big[A\mathbb X_t^I+DF(\mathbb X_t)\mathbb X_t^I
+\sum_{\substack{\pi\in\Pi(I),\,|\pi|\ge2}}
D^{|\pi|}F(\mathbb X_t)
(\mathbb X_t^B:B\in\pi)\Big]\ud t\notag\\
&+\Big[D\Sigma(\mathbb X_t)\mathbb X_t^I
+\sum_{\substack{\pi\in\Pi(I),\,|\pi|\ge2}}
D^{|\pi|}\Sigma(\mathbb X_t)
(\mathbb X_t^B:B\in\pi)\Big]\ud W_t .
\end{align}
We first prove \eqref{eq:X-variation-lemma} by induction on $|I|$.
For \(|I|=1\), both sums on the right hand side of
\eqref{eq:XtI} are empty. Lemma \ref{lem:Z} then implies
\eqref{eq:X-variation-lemma}.
Suppose by induction that \eqref{eq:X-variation-lemma} holds
at every moment order \(p\ge2\)
for all nonempty direction sets of cardinality at most \(M\)
, where \(M\in\{1,2\}\).
Let \(|I|=M+1\). Every block \(B\) of a partition
\(\pi\in\Pi(I)\) with \(|\pi|\ge2\) satisfies \(|B|\le M\).
Hence H\"older's inequality, \eqref{eq:DjG1}, and the induction
hypothesis at moment order \(p|\pi|\) yield
\begin{align*}
&\|D^{|\pi|}F(\mathbb X_t)
(\mathbb X_t^B:B\in\pi)\|_{L^p(\Omega,\cH^1)}
+\|D^{|\pi|}\Sigma(\mathbb X_t)
(\mathbb X_t^B:B\in\pi)\|_
 {L^p(\Omega,\mathscr L_2(H_Q,\cH^1))}\\
&\qquad\le
C\prod_{B\in\pi}
\|\mathbb X_t^B\|_{L^{p|\pi|}(\Omega,\cH^1)}
\le C\prod_{B\in\pi}\mathcal D_B
=C\mathcal D_I .
\end{align*}
Applying Lemma \ref{lem:Z} to \eqref{eq:XtI} now proves
\eqref{eq:X-variation-lemma} for \(|I|=M+1\).

We now prove \eqref{eq:Y-variation-lemma} by induction on \(|I|\).
It is enough to consider \(b(I)\le1\), since $\mathbb Y_t^I$ vanishes for \(b(I)\ge2\).
For \(|I|=1\), let first \(I=\{\ell\}\) contain an \(x\)-direction.
The equation \eqref{eq:Galerkin-systemY} for \(\mathbb Y^I\) contains the inhomogeneous
coefficients
\[
D_z\Phi_F(\mathbb X,\mathbb Y)\mathbb X^I,\qquad
D_z\Phi_\Sigma(\mathbb X,\mathbb Y)\mathbb X^I,\qquad
\frac{T}{\sqrt{12}}D\phi_\Sigma(\mathbb X)\mathbb X^I.
\]
By H\"older's inequality,
\eqref{eq:X-variation-lemma},
\eqref{eq:Phi-coordinate-bound}, \eqref{eq:phi-coordinate-bound}, and
\eqref{eq:generic-flow-moments}, their \(L^p\)-norms
are bounded by
$C(1+\|x\|_{\cH^1}+\|y\|_{\cH})\|\ell\|_{\cH^1}.$
Lemma \ref{lem:Z}
therefore yields
\[
\sup_{t\in[0,T]}
\big(\|\mathbb X_t^I\|_{L^p(\Omega,\cH^1)}
+\|\mathbb Y_t^I\|_{L^p(\Omega,\cH)}\big)
\le C(1+\|x\|_{\cH^1}+\|y\|_{\cH})\|\ell\|_{\cH^1}.
\]
If \(I=\{h\}\) contains a \(y\)-direction, its first component is zero
and its second component solves the homogeneous linearized equation
with initial value \(h\). Therefore, Lemma \ref{lem:Z}
provides the
bound \(C\|h\|_{\cH}\). These prove
\eqref{eq:Y-variation-lemma} for \(|I|=1\).

Now suppose that \eqref{eq:Y-variation-lemma} holds, at every
moment order \(p\ge2\), for all nonempty direction sets of
cardinality at most \(M\), where
\(M\in\{1,2,3\}\).
Let \(|I|=M+1\). Direct differentiation
of the \(\mathbb Y\)-equation \eqref{eq:Galerkin-systemY} leads to
\begin{align}\label{eq:second-component-equation}
\ud\mathbb Y_t^I
={}&\Big[A\mathbb Y_t^I
+DF(\mathbb X_t)\mathbb Y_t^I
+\mathcal R_{F,I}(t)\Big]\ud t\notag\\
&+\Big[D\Sigma(\mathbb X_t)\mathbb Y_t^I
+\mathcal R_{\Sigma,I}(t)\Big]\ud W_t
+\widetilde{\mathcal R}_{\Sigma,I}(t)\ud\widetilde W_t .
\end{align}
For \(G\in\{F,\Sigma\}\), we have
\begin{align*}
\mathcal R_{G,I}(t)
&:=\mathbf 1_{\{b(I)=0\}}\Bigg[
D_z\Phi_G(\mathbb X_t,\mathbb Y_t)\mathbb X_t^I+\sum_{\substack{\pi\in\Pi(I),\,|\pi|\ge2}}
D_z^{|\pi|}\Phi_G(\mathbb X_t,\mathbb Y_t)
(\mathbb X_t^B:B\in\pi)\Bigg]\\
&\quad+\sum_{\substack{\pi\in\Pi(I),\ |\pi|\ge2,\ B_0\in\pi\\
b(B)=0\ {\rm for\ every}\ B\in\pi\setminus\{B_0\}}}
D_z^{|\pi|-1}D_\xi\Phi_G(\mathbb X_t,\mathbb Y_t)
\big((\mathbb X_t^B)_{B\in\pi\setminus\{B_0\}};
\mathbb Y_t^{B_0}\big).
\end{align*}
The first sum contains only \(z\)-directions. In every term of the
second sum, \(B_0\) supplies the unique \(\xi\)-direction. Similarly, the remainder $\widetilde{\mathcal R}_{\Sigma,I}(t)$ in \eqref{eq:second-component-equation} is given by
\begin{align*}
\widetilde{\mathcal R}_{\Sigma,I}(t)
:=\mathbf 1_{\{b(I)=0\}}\frac{T}{\sqrt{12}}\Bigg[
D\phi_\Sigma(\mathbb X_t)\mathbb X_t^I
+\sum_{\substack{\pi\in\Pi(I),\,|\pi|\ge2}}
D^{|\pi|}\phi_\Sigma(\mathbb X_t)
(\mathbb X_t^B:B\in\pi)\Bigg].
\end{align*}
Only the derivatives of $\Phi_G$ and $\phi_G$ displayed in
\eqref{eq:phi-coordinate-bound} and
\eqref{eq:Phi-coordinate-bound} occur.

We next estimate all three remainders in \eqref{eq:second-component-equation}. Fix \(p\ge2\).
Estimate \eqref{eq:generic-flow-moments} implies
\[
\|(1+\|\mathbb X_t\|_{\cH^1}+\|\mathbb Y_t\|_{\cH})\|_{L^{2p}(\Omega)}
\le C(1+\|x\|_{\cH^1}+\|y\|_{\cH}).
\]
Together with \eqref{eq:Phi-coordinate-bound},
\eqref{eq:phi-coordinate-bound}, and
\eqref{eq:X-variation-lemma}, we first obtain, when
\(b(I)=0\),
\begin{align*}
&\|D_z\Phi_F(\mathbb X_t,\mathbb Y_t)\mathbb X_t^I\|_
 {L^p(\Omega,\cH)}+\|D_z\Phi_\Sigma(\mathbb X_t,\mathbb Y_t)\mathbb X_t^I\|_
 {L^p(\Omega,\mathscr L_2(H_Q,\cH))}+
\|D\phi_\Sigma(\mathbb X_t)\mathbb X_t^I\|_
 {L^p(\Omega,\mathscr L_2(H_Q,\cH))}\\
&\le C(1+\|x\|_{\cH^1}+\|y\|_{\cH})\mathcal D_I
\le C(1+\|x\|_{\cH^1}+\|y\|_{\cH})^{a(I)}\mathcal D_I .
\end{align*}
Here \(a(I)=|I|=M+1\ge2\) at the induction step.

Let \(\pi\in\Pi(I)\) with \(|\pi|\ge2\), and put
\(\widehat p_\pi:=2p|\pi|\). If \(b(I)=0\), H\"older's inequality
and the first-component estimate \eqref{eq:X-variation-lemma} at moment order
\(\widehat p_\pi\) imply
\begin{align*}
&\|D_z^{|\pi|}\Phi_F(\mathbb X_t,\mathbb Y_t)
(\mathbb X_t^B:B\in\pi)\|_{L^p(\Omega,\cH)}\\
&\quad+
\|D_z^{|\pi|}\Phi_\Sigma(\mathbb X_t,\mathbb Y_t)
(\mathbb X_t^B:B\in\pi)\|_
 {L^p(\Omega,\mathscr L_2(H_Q,\cH))}\\
&\quad+
\|D^{|\pi|}\phi_\Sigma(\mathbb X_t)
(\mathbb X_t^B:B\in\pi)\|_
 {L^p(\Omega,\mathscr L_2(H_Q,\cH))}\\
&\le C(1+\|x\|_{\cH^1}+\|y\|_{\cH})
\prod_{B\in\pi}\mathcal D_B
\le C(1+\|x\|_{\cH^1}+\|y\|_{\cH})^{a(I)}\mathcal D_I .
\end{align*}

It remains to treat a term in which \(B_0\in\pi\) supplies the
unique \(\xi\)-direction. If every
\(B\in\pi\setminus\{B_0\}\) is a pure \(x\)-block, H\"older's
inequality, the induction hypothesis at moment order
\(\widehat p_\pi\), and
\eqref{eq:X-variation-lemma} yield
\begin{align*}
&\|D_z^{|\pi|-1}D_\xi\Phi_F(\mathbb X_t,\mathbb Y_t)
\big((\mathbb X_t^B)_{B\in\pi\setminus\{B_0\}};
\mathbb Y_t^{B_0}\big)\|_{L^p(\Omega,\cH)}\\
&\quad+
\|D_z^{|\pi|-1}D_\xi\Phi_\Sigma(\mathbb X_t,\mathbb Y_t)
\big((\mathbb X_t^B)_{B\in\pi\setminus\{B_0\}};
\mathbb Y_t^{B_0}\big)\|_
 {L^p(\Omega,\mathscr L_2(H_Q,\cH))}\\
&\le C(1+\|x\|_{\cH^1}+\|y\|_{\cH})
\|\mathbb Y_t^{B_0}\|_{L^{\widehat p_\pi}(\Omega,\cH)}
\prod_{\substack{B\in\pi,B\ne B_0}}
\|\mathbb X_t^B\|_{L^{\widehat p_\pi}(\Omega,\cH^1)}\\
&\le C(1+\|x\|_{\cH^1}+\|y\|_{\cH})^{a(B_0)+1}\mathcal D_I
\le C(1+\|x\|_{\cH^1}+\|y\|_{\cH})^{a(I)}\mathcal D_I .
\end{align*}
The last inequality holds because \(|\pi|\ge2\), and every block
other than \(B_0\) is a nonempty pure \(x\)-block. Thus
\(a(B_0)+1\le a(I)\). Combining these estimates results in, for
every {\(t\in[0,T]\)},
\begin{align}\label{eq:variation-remainder-bound}
&\|\mathcal R_{F,I}(t)\|_{L^p(\Omega,\cH)}
+\|\mathcal R_{\Sigma,I}(t)\|_
 {L^p(\Omega,\mathscr L_2(H_Q,\cH))}
+\|\widetilde{\mathcal R}_{\Sigma,I}(t)\|_
 {L^p(\Omega,\mathscr L_2(H_Q,\cH))}\\
&\le C(1+\|x\|_{\cH^1}+\|y\|_{\cH})^{a(I)}\mathcal D_I .
\end{align}
Since {\(\mathbb Y_0^I=0\)} for \(|I|\ge2\), Lemma \ref{lem:Z}
applied to \eqref{eq:second-component-equation}, along with
\eqref{eq:variation-remainder-bound}, leads to
\begin{align*}
\E\sup_{t\in[0,T]}\|\mathbb Y_t^I\|_{\cH}^p&\le C{\int_0^T}\E\Big[
\|\mathcal R_{F,I}(t)\|_{\cH}^p
+\|\mathcal R_{\Sigma,I}(t)\|_{\mathscr L_2(H_Q,\cH)}^p
+\|\widetilde{\mathcal R}_{\Sigma,I}(t)\|_
 {\mathscr L_2(H_Q,\cH)}^p\Big]\ud t\\
&\le C(1+\|x\|_{\cH^1}+\|y\|_{\cH})^{pa(I)}\mathcal D_I^p .
\end{align*}
Taking the \(p\)th root proves \eqref{eq:Y-variation-lemma} for
\(|I|=M+1\) and closes the induction.
\end{proof}
\begin{proof}[Proof of Proposition~\ref{prop:varphiY}]
The estimate follows by differentiating \eqref{eq:utx}
and applying the chain rule, H\"older's inequality,
\eqref{eq:generic-flow-moments}, and
Lemma~\ref{lem:flow-variations}.
\end{proof}
%
\section{Proof of Proposition \ref{prop:convergence-rates-sharpness}}
\label{app:sharpness}
\begin{proof} The proof of Proposition \ref{prop:convergence-rates-sharpness} is divided into four steps.


\textbf{Step 1. Verification of the assumptions and \eqref{eq:phipsi}.}
Recall the setting of Example~\ref{ex:sharpness}.
Since $Q=g\otimes g$ is a rank-one covariance operator, the
corresponding $Q$-Wiener processes may be represented as
\[
W_t=g\beta_t,
\qquad
\widetilde W_t=g\widetilde\beta_t,
\]
where $\beta$ and $\widetilde\beta$ are independent standard
one-dimensional Brownian motions.
Since $\alpha>1$ and $\kappa_j\ge1$, we have 
\[
\|g\|_{L^\infty(\mathcal O)}
\le
\sqrt2\bigg(
1+
\sum_{j=1}^{\infty}
\frac{(2\pi\kappa_j)^{1-\alpha}}{j^2}
\bigg)
\le \sqrt2\bigg(1+\sum_{j=1}^{\infty}\frac1{j^2}\bigg)<\infty,
\]
where $g$ is given by \eqref{eq:gdef}.
For the rank-one operator $Q=g\otimes g$, choose an orthonormal
eigenbasis $\{e_i\}_{i\ge1}$ of $Q$ such that
$e_1=g/\|g\|_H$ and $\{e_i\}_{i\ge2}$ is an orthonormal basis of the
orthogonal complement of $\operatorname{span}\{g\}$ in $H$.
Then $Q^{1/2}e_1=g$ and $Q^{1/2}e_i=0$ for $i\ge2$,  which together with $g\in L^\infty(\mathcal{O})$, implies Assumption~\ref{asp:F}.
Since $f\equiv0$ and $\sigma\equiv1$,
their derivatives vanish, so Assumption~\ref{asp:fg3} also holds.
To verify Assumption \ref{asp:alp}, observe that
\begin{align*}
\|g\|_{H^{\alpha-1}}^2
&=
\pi^{2(\alpha-1)}
+
\sum_{j=1}^{\infty}
(2\pi\kappa_j)^{2(\alpha-1)}
\frac{(2\pi\kappa_j)^{2(1-\alpha)}}{j^4}=
\pi^{2(\alpha-1)}
+
\sum_{j=1}^{\infty}j^{-4}
<\infty.
\end{align*}
Again using the rank-one structure of $Q$, we have
\[
\sum_{i=1}^{\infty}
\left\|
\Lambda^{(\alpha-1)/2}Q^{1/2}e_i
\right\|_H^2
=
\|g\|_{H^{\alpha-1}}^2
<\infty.
\]
Hence Assumption~\ref{asp:alp} holds.
For the test functions defined in
Example~\ref{ex:sharpness}, since
$\Lambda^{-1/2}:H^{-1}\to H$ is an isometry, it follows immediately
that \eqref{eq:phipsi} holds.

\textbf{Step 2. Weak-error reduction.}
Since
\(
Y_1^m=m(X_1^m-X_1)
\)
and the particular test function $\varphi$ in \eqref{eq:varphi} is
linear in its second argument,
\begin{align}
&\E[\varphi(X_1,Y_1^m)]
-
\E[\varphi(X_1,Y_1)]
\nonumber\\
&\qquad
=
m\left(
\E[\varphi(X_1,X_1^m)]
-
\E[\varphi(X_1,X_1)]
\right)
-
\E[\varphi(X_1,Y_1)].
\label{eq:sharp-reduction-1}
\end{align}
We first evaluate the last term $\E[\varphi(X_1,Y_1)]$.
In the present setting, the exact solution, the
exponential Euler approximation, and the limiting error process admit
the mild representations
\begin{align}
X_t
&=
\int_0^t
S_{t-s}
\begin{pmatrix}
0\\ g
\end{pmatrix}
\,\ud\beta_s,
\qquad
X_t^m
=
\int_0^t
S_{t-\underline{s}}
\begin{pmatrix}
0\\ g
\end{pmatrix}
\,\ud\beta_s,
\label{eq:rank-one-X-mild}\\
Y_t
&=
\frac12\int_0^t
S_{t-s}
\begin{pmatrix}
g\\0
\end{pmatrix}
\,\ud\beta_s
+
\frac1{\sqrt{12}}
\int_0^t
S_{t-s}
\begin{pmatrix}
g\\0
\end{pmatrix}
\,\ud\widetilde\beta_s.
\label{eq:rank-one-Y-mild}
\end{align}
Using the explicit representation \eqref{eq:St} of the wave group
$S_t$, the relevant components at $t=1$ are
\begin{align}
\left\langle[X_1]_1,\eta_1\right\rangle_H
&=
\frac1\pi
\int_0^1
\sin\bigl(\pi(1-s)\bigr)\,\ud\beta_s,
\label{eq:X1eta}\\
\left\langle
\Lambda^{-1/2}[Y_1]_2,\rho
\right\rangle_H
&=
-\frac12
\sum_{i=1}^{\infty}
\frac{(2\pi\kappa_i)^{1-\alpha}}{i^3}
\int_0^1
\sin\bigl(2\pi\kappa_i(1-s)\bigr)\,\ud\beta_s
\nonumber\\
&\quad
-\frac1{\sqrt{12}}
\sum_{i=1}^{\infty}
\frac{(2\pi\kappa_i)^{1-\alpha}}{i^3}
\int_0^1
\sin\bigl(2\pi\kappa_i(1-s)\bigr)\,
\ud\widetilde\beta_s.
\end{align}
It\^o's isometry, the independence of $\beta$ and $\widetilde\beta$,
and the change of variables $t=1-s$ yield
\begin{align}
\E[\varphi(X_1,Y_1)]
&=
-\frac1{2\pi}
\sum_{i=1}^{\infty}
\frac{(2\pi\kappa_i)^{1-\alpha}}{i^3}
\int_0^1
\sin(\pi t)\sin(2\pi\kappa_i t)\,\ud t=0,
\label{eq:sharp-lim-cov-zero}
\end{align}
where the last equality follows from the orthogonality of
$\sin(\pi \cdot)$ and $\sin(2\pi\kappa_i \cdot)$ in $L^2(0,1)$.

We next compare
$\E[\varphi(X_1,X_1^m)]$ and
$\E[\varphi(X_1^m,X_1^m)]$.
From \eqref{eq:rank-one-X-mild}, we have
\begin{equation}
\label{eq:X1mse}
B_m
:=
\left\langle
\Lambda^{-1/2}[X_1^m]_2,\rho
\right\rangle_H
=
\sum_{n=0}^{m-1}
\sum_{i=1}^{\infty}
\int_{t_n}^{t_{n+1}}
b_i\cos(2\pi\kappa_i t_n)\,\ud\beta_t,
\quad
b_i:=\frac{1}{i^3(2\pi\kappa_i)^\alpha}.
\end{equation}
Combining this representation with \eqref{eq:X1eta} and
\eqref{eq:varphi}, we obtain
\begin{equation}
\label{eq:sharp-mixed-cov}
\E[\varphi(X_1,X_1^m)]
=
\frac1\pi
\sum_{n=0}^{m-1}
\sum_{i=1}^{\infty}
b_i
\int_{t_n}^{t_{n+1}}
\cos(2\pi\kappa_i t_n)
\sin(\pi t)\,\ud t.
\end{equation}
Similarly, replacing $\sin(\pi t)$ in \eqref{eq:sharp-mixed-cov} by
$\sin(\pi\ttau)$ gives
\begin{equation}
\label{eq:varphix1x1}
\Gamma_m:=\E[\varphi(X_1^m,X_1^m)]
=
\frac{1}{m\pi}
\sum_{n=0}^{m-1}
\sin(\pi t_n)
\sum_{i=1}^{\infty}
b_i\cos(2\pi\kappa_i t_n).
\end{equation}
For every $n$,
\begin{align*}
\int_{t_n}^{t_{n+1}}\sin(\pi t)\,\ud t
=
\frac{\sin(\pi/m)}{\pi}\sin(\pi t_n)
+
\frac{1-\cos(\pi/m)}{\pi}\cos(\pi t_n).
\end{align*}
Using the identity
$
\sum_{n=0}^{m-1}
\cos\left(\frac{\pi n}{m}\right)
\cos\left(\frac{2\pi\kappa_i n}{m}\right)
=1,
$
we therefore obtain
\begin{align*}
\E[\varphi(X_1,X_1^m)]
&=
\frac{m\sin(\pi/m)}{\pi}
\E[\varphi(X_1^m,X_1^m)]
+
\frac{1-\cos(\pi/m)}{\pi^2}
\langle\Lambda^{-1/2}g,\rho\rangle_H.
\end{align*}
Since $\E[\varphi(X_1^m,X_1^m)]$ is uniformly bounded in $m$ and
$
|
\frac{m\sin(\pi/m)}{\pi}-1
|
+
|1-\cos(\pi/m)|
\le Cm^{-2},
$
\begin{equation}
\label{eq:sharp-mixed-small}
\left|
\E[\varphi(X_1,X_1^m)]
-
\E[\varphi(X_1^m,X_1^m)]
\right|
\le Cm^{-2}.
\end{equation}
Combining
\eqref{eq:sharp-reduction-1},
\eqref{eq:sharp-lim-cov-zero}, and
\eqref{eq:sharp-mixed-small}, we arrive at
\begin{gather}
\label{eq:sharp-main-reduction}
\E[\varphi(X_1,Y_1^m)]
-
\E[\varphi(X_1,Y_1)]
=
m(\Gamma_m-\Gamma_\infty)
+
O(m^{-1}),
\end{gather}
where $\Gamma_\infty
:=
\E[\varphi(X_1,X_1)]$.
For the second component $[X_1]_2$ of the exact solution $X_1$, we have
\begin{equation}\label{eq:X1mseinf}
B_\infty:=\left\langle
\Lambda^{-1/2}[X_1]_2,\rho
\right\rangle_H
=
\int_0^1
\sum_{i=1}^{\infty}
b_i\cos(2\pi\kappa_i t)\,\ud\beta_t,
\end{equation}
which, in combination with \eqref{eq:X1eta}, leads to
\begin{equation*}
\Gamma_\infty
=
\frac1\pi
\sum_{i=1}^{\infty}
b_i
\int_0^1
\sin(\pi t)\cos(2\pi\kappa_i t)\,\ud t=
-\frac1\pi
\sum_{i=1}^{\infty}
b_i
\frac{2}{\pi(4\kappa_i^2-1)}.
\end{equation*}
Comparing this identity with \eqref{eq:varphix1x1}, we define
\begin{equation}
\label{eq:Rmk}
R_m(k)
:=
\frac1m
\sum_{n=0}^{m-1}
\sin\left(\frac{\pi n}{m}\right)
\cos\left(\frac{2\pi kn}{m}\right)
+
\frac{2}{\pi(4k^2-1)},
\qquad k\ge1.
\end{equation}
Then the difference $\Gamma_m-\Gamma_\infty$ can be rewritten as
\begin{equation}
\label{eq:sharp-kernel-representation}
\Gamma_m-\Gamma_\infty
=
\frac1\pi
\sum_{i=1}^{\infty}
b_iR_m(\kappa_i).
\end{equation}

\textbf{Step 3. Proof of \eqref{eq:shp1}.}
Fix $j\ge2$ and consider
$
m\in\{\kappa_j,2\kappa_j\}.
$
We decompose \eqref{eq:sharp-kernel-representation} into the resonant
mode $i=j$, the lower-frequency modes $i<j$, and the higher-frequency
modes $i>j$.

\textbf{(Resonant mode)} For the resonant mode $i=j$, \eqref{eq:Rmk} implies
\begin{align*}
R_{\kappa_j}(\kappa_j)
&=\frac1{\kappa_j}
\cot\left(\frac{\pi}{2\kappa_j}\right)
+\frac{2}{\pi(4\kappa_j^2-1)},\\
R_{2\kappa_j}(\kappa_j)
&=-\frac1{2\kappa_j}
\tan\left(\frac{\pi}{4\kappa_j}\right)
+\frac{2}{\pi(4\kappa_j^2-1)}.
\end{align*}
Using
$
0
\le
\frac1{\sin^2\theta}-\frac1{\theta^2}
\le1$ for
$
0<\theta<\frac{\pi}{2},
$
we obtain
\(
0
\le
\frac2\pi
-
\frac1{\kappa_j}
\cot\left(\frac{\pi}{2\kappa_j}\right)
\le
\frac{\pi}{2\kappa_j^2}.
\)
It then follows that
\begin{equation}
\label{eq:sharp-resonance}
R_{\kappa_j}(\kappa_j)
=
\frac2\pi+O(\kappa_j^{-2}),
\qquad
R_{2\kappa_j}(\kappa_j)
=
O(\kappa_j^{-2}).
\end{equation}
Since
$b_j=j^{-3}(2\pi\kappa_j)^{-\alpha}$, the contribution of the
resonant mode to \eqref{eq:sharp-kernel-representation} with
$
m\in\{\kappa_j,2\kappa_j\}
$
is
\begin{equation}
\frac1\pi b_jR_{\kappa_j}(\kappa_j)
=
\frac{2}{\pi^2(2\pi)^\alpha}
\frac{1}{j^3\kappa_j^\alpha}
+
O\Big(
\frac{1}{j^3\kappa_j^{\alpha+2}}
\Big),\quad
\frac1\pi b_jR_{2\kappa_j}(\kappa_j)
=
O\Big(
\frac{1}{j^3\kappa_j^{\alpha+2}}
\Big).
\label{eq:sharp-resonant-main}
\end{equation}

\textbf{(Lower-frequency modes)}
We estimate the lower-frequency modes $i<j$.
Note that
\[
\frac1m
\sum_{n=0}^{m-1}
\sin\left(\frac{\pi n}{m}\right)
\cos\left(\frac{2\pi\kappa_i n}{m}\right)
=
\frac1{2m}
\left(
\cot\theta_i^+
-
\cot\theta_i^-
\right), \qquad \theta_i^\pm
:=
\frac{(2\kappa_i\pm1)\pi}{2m}.
\]
Since
$
\frac{\kappa_j}{\kappa_{j-1}}
=
2^{2j-1}
\ge8
$
and $m\in\{\kappa_j,2\kappa_j\}$,
\[
2\kappa_i
\le
2\kappa_{j-1}
\le
\frac{\kappa_j}{4}
\le
\frac m4,
\qquad i<j.
\]
This implies
$
0<\theta_i^-<\theta_i^+<\frac{\pi}{2}.
$
Using \eqref{eq:Rmk}, we obtain
\begin{align*}
R_m(\kappa_i)
&=
\frac1{2m}
\left[
\Big(
\cot\theta_i^+-\frac1{\theta_i^+}
\Big)
-
\Big(
\cot\theta_i^--\frac1{\theta_i^-}
\Big)
\right]=-\frac1{2m}
\int_{\theta_i^-}^{\theta_i^+}
\Big(
\frac1{\sin^2x}-\frac1{x^2}
\Big)\,\ud x.
\end{align*}
By the fact that $\sin x\le x\le\tan x$ for $0<x<\pi/2$, we have
\(
0
\le
\frac1{\sin^2x}-\frac1{x^2}
\le1.
\)
Together with
$\theta_i^+-\theta_i^-=\pi/m$, this gives
$
|R_m(\kappa_i)|
\le
\frac{\pi}{2m^2}$ for
$i<j.
$
Moreover,
\(
\sum_{i=1}^{\infty}b_i
\le
(2\pi)^{-\alpha}
\sum_{i=1}^{\infty}i^{-3}
<\infty.
\)
Using the positivity of $b_i$, there exists $C>0$ independent of $j$ and $m$ such that
\begin{equation}
\label{eq:sharp-low-modes}
\bigg|
\frac1\pi
\sum_{i=1}^{j-1}
b_iR_m(\kappa_i)
\bigg|
\le
Cm^{-2}
\le
C\kappa_j^{-2}.
\end{equation}

\textbf{(Higher-frequency modes)}
It remains to estimate the higher-frequency modes $i>j$.
From \eqref{eq:Rmk},
$
|R_m(\kappa_i)|
\le
1+\frac{2}{3\pi}
<2.
$
Since
\[
\frac{b_{i+1}}{b_i}
=
\left(\frac{i}{i+1}\right)^3
2^{-\alpha(2i+1)}
\le
2^{-3\alpha}
<1,
\]
the tail is bounded by a geometric series as follows:
\begin{align}
s_j
:=
\sum_{i=j+1}^{\infty}b_i
&\le
b_{j+1}
\sum_{\ell=0}^{\infty}2^{-3\alpha\ell}
=
\frac{1}
{(1-2^{-3\alpha})(j+1)^3(2\pi\kappa_{j+1})^\alpha}
\le
C
\frac{2^{-\alpha(2j+1)}}
{j^3\kappa_j^\alpha},
\label{eq:sharp-high-modessj}
\end{align}
where we have used
$\kappa_{j+1}=2^{2j+1}\kappa_j$.
Hence, there exists $C>0$ independent of $j$ and $m$ such that
\begin{equation}
\label{eq:sharp-high-modes}
\bigg|
\frac1\pi
\sum_{i=j+1}^{\infty}
b_iR_m(\kappa_i)
\bigg|
\le
C
\frac{2^{-\alpha(2j+1)}}
{j^3\kappa_j^\alpha}.
\end{equation}

Combining
\eqref{eq:sharp-resonant-main},
\eqref{eq:sharp-low-modes}, and
\eqref{eq:sharp-high-modes}
with \eqref{eq:sharp-kernel-representation}, we obtain
\begin{align}
\Gamma_{\kappa_j}-\Gamma_\infty
=
\frac{2}{\pi^2(2\pi)^\alpha}
\frac{1}{j^3\kappa_j^\alpha}
+r_j,
\qquad
\Gamma_{2\kappa_j}-\Gamma_\infty
=
\widetilde r_j,
\label{eq:sharp-C2kappa}
\end{align}
where the remainders $r_j$ and $\widetilde r_j$ satisfy
\[
|r_j|+|\widetilde r_j|
\le
C\bigg(
\kappa_j^{-2}
+
\frac{2^{-\alpha(2j+1)}}
{j^3\kappa_j^\alpha}
\bigg)
=
o\bigg(
\frac{1}{j^3\kappa_j^\alpha}
\bigg).
\]
Here we used $\alpha<2$ and the super-exponential growth of
$\kappa_j$.
For $m=\kappa_j$, the leading term in
\eqref{eq:sharp-main-reduction} is of order
$j^{-3}\kappa_j^{-(\alpha-1)}$, whereas the remainder is of order
$\kappa_j^{-1}$. Since $\alpha<2$ and
$\kappa_j=2^{j^2-1}$,
\[
\frac{\kappa_j^{-1}}
{j^{-3}\kappa_j^{-(\alpha-1)}}
=
j^3\kappa_j^{\alpha-2}
\rightarrow0,\quad \text{as }j\to\infty.
\]
Thus, the $O(\kappa_j^{-1})$ remainder is negligible compared with
the leading term.
Applying \eqref{eq:sharp-main-reduction} with $m=\kappa_j$ gives
\[
\E[\varphi(X_1,Y_1^{\kappa_j})]
-
\E[\varphi(X_1,Y_1)]
=
\frac{2}{\pi^2(2\pi)^\alpha}
\frac{1+o(1)}
{j^3\kappa_j^{\alpha-1}}.
\]
Hence, for every $\varepsilon>0$,
\begin{align*}
\kappa_j^{\alpha-1+\varepsilon}
\left|
\E[\varphi(X_1,Y_1^{\kappa_j})]
-
\E[\varphi(X_1,Y_1)]
\right|
\sim
\frac{2}{\pi^2(2\pi)^\alpha}
\frac{\kappa_j^\varepsilon}{j^3}\rightarrow\infty,
\qquad j\to\infty.
\end{align*}
This proves \eqref{eq:shp1}.

\textbf{Step 4. Sharpness of the Richardson rate.}
Continue to take $j\ge2$ and
$m\in\{\kappa_j,2\kappa_j\}$; in particular, $m\ge2$.
In accordance with \eqref{eq:psidef}, define
\begin{align}
Z_m
&:=
\langle[X_1^m]_1,\eta_1\rangle_H
+
\langle
\Lambda^{-1/2}[X_1^m]_2,\rho
\rangle_H,
\nonumber\\
Z_\infty
&:=
\langle[X_1]_1,\eta_1\rangle_H
+
\langle
\Lambda^{-1/2}[X_1]_2,\rho
\rangle_H.
\label{eq:ZmZ}
\end{align}
Both $Z_m$ and $Z_\infty$ are centered real-valued Gaussian random
variables. Therefore,
\begin{equation}
\label{eq:sharp-gaussian}
\E[\psi(X_1^m)]
=
\exp\left(
-\frac12\E[Z_m^2]
\right),
\qquad
\E[\psi(X_1)]
=
\exp\left(
-\frac12\E[Z_\infty^2]
\right).
\end{equation}

We first compare $\E[Z_m^2]$ with $\E[Z_\infty^2]$.
For $m\in\{\kappa_j,2\kappa_j\}$, the first component $[X_1^m]_1$ of the numerical
solution $X_1^m$ satisfies
\(
\langle[X_1^m]_1,\eta_1\rangle_H
=
\frac1\pi
\sum_{n=0}^{m-1}
\sin(\pi t_n)
(
\beta_{t_{n+1}}-\beta_{t_n}
).
\)
By It\^o's isometry,
\begin{align*}
\E\left[
\left|
\left\langle[X_1^m]_1,\eta_1\right\rangle_H
\right|^2
\right]
&=
\frac{1}{m\pi^2}
\sum_{n=0}^{m-1}
\sin^2\left(\frac{\pi n}{m}\right)=
\frac{1}{2\pi^2}.
\end{align*}
On the other hand, \eqref{eq:X1eta} gives
\begin{equation}\label{eq:sharp-first-variance}
\E\left[
\left|
\left\langle[X_1]_1,\eta_1\right\rangle_H
\right|^2
\right]
=
\frac1{\pi^2}
\int_0^1\sin^2(\pi t)\,\ud t
=
\frac1{2\pi^2}=\E\left[
\left|
\left\langle[X_1^m]_1,\eta_1\right\rangle_H
\right|^2
\right].
\end{equation}

By \eqref{eq:ZmZ}, the definitions of $\Gamma_m$ and $\Gamma_\infty$, and
\eqref{eq:sharp-first-variance}, we have
\begin{equation}
\label{eq:sharp-total-variance}
\E[Z_m^2]-\E[Z_\infty^2]
=
2(\Gamma_m-\Gamma_\infty)
+
\E[B_m^2]-\E[B_\infty^2],
\qquad
m\in\{\kappa_j,2\kappa_j\},
\end{equation}
where $B_m$ and $B_\infty$ are defined in \eqref{eq:X1mse} and
\eqref{eq:X1mseinf}, respectively.
Since the frequencies $\kappa_i$ are distinct integers, It\^o's
isometry and orthogonality in $L^2(0,1)$ imply
\begin{equation}
\label{eq:sharp-B-exact-var}
\E[B_\infty^2]
=
\frac12
\sum_{i=1}^{\infty}b_i^2.
\end{equation}

For $i>j$, the ratio $\kappa_i/\kappa_j$ is an even integer, and hence
\[
\cos\left(
\frac{2\pi\kappa_i n}{\kappa_j}
\right)
=
1,
\qquad
\cos\left(
\frac{2\pi\kappa_i n}{2\kappa_j}
\right)
=
1,
\qquad i>j.
\]
For the resonant mode $i=j$,
\[
\cos\left(
\frac{2\pi\kappa_j n}{\kappa_j}
\right)
=
1,
\qquad
\cos\left(
\frac{2\pi\kappa_j n}{2\kappa_j}
\right)
=
(-1)^n.
\]

For $i<j$, the discrete cosine modes are mutually orthogonal and are
also orthogonal to the constant mode and, when $m=2\kappa_j$, to
$(-1)^n$. Therefore, recalling the definition of $s_j$ in
\eqref{eq:sharp-high-modessj}, it follows from \eqref{eq:X1mse} that
\begin{align*}
\E[B_{\kappa_j}^2]
&=
\frac12
\sum_{i=1}^{j-1}b_i^2
+
(b_j+s_j)^2,\qquad
\E[B_{2\kappa_j}^2]
=
\frac12
\sum_{i=1}^{j-1}b_i^2
+
b_j^2+s_j^2.
\end{align*}

Combining these identities with \eqref{eq:sharp-B-exact-var}, we
obtain
\begin{align}
\E[B_{\kappa_j}^2]-\E[B_\infty^2]
&=
\frac12b_j^2
+
2b_js_j
+
s_j^2
-
\frac12
\sum_{i=j+1}^{\infty}b_i^2,
\label{eq:sharp-B-defect-kappa}\\
\E[B_{2\kappa_j}^2]-\E[B_\infty^2]
&=
\frac12b_j^2
+
s_j^2
-
\frac12
\sum_{i=j+1}^{\infty}b_i^2.
\label{eq:sharp-B-defect-2kappa}
\end{align}
The tail estimate \eqref{eq:sharp-high-modessj} gives
\(
s_j
\le
C
\frac{2^{-\alpha(2j+1)}}
{j^3\kappa_j^\alpha}
=
o(b_j).
\)
Therefore,
\begin{equation}
\label{eq:sharp-B-defect}
\left|
\E[B_{\kappa_j}^2]-\E[B_\infty^2]
\right|
+
\left|
\E[B_{2\kappa_j}^2]-\E[B_\infty^2]
\right|
\le
Cb_j^2
=
\frac{C}{j^6\kappa_j^{2\alpha}}.
\end{equation}
Combining \eqref{eq:sharp-gaussian} and
\eqref{eq:sharp-total-variance}, we obtain that for
$m\in\{\kappa_j,2\kappa_j\}$,
\begin{align}
\E[\psi(X_1^m)]
&=
\E[\psi(X_1)]
\exp\Big(
-(\Gamma_m-\Gamma_\infty)
-\frac12
\bigl(
\E[B_m^2]-\E[B_\infty^2]
\bigr)
\Big).
\label{eq:sharp-gaussian-ratio}
\end{align}

By \eqref{eq:sharp-C2kappa}, for all sufficiently large $j$,
\[
|\Gamma_m-\Gamma_\infty|
\le
\frac{C}{j^3\kappa_j^\alpha},
\qquad
m\in\{\kappa_j,2\kappa_j\}.
\]
Hence \eqref{eq:sharp-B-defect} and the expansion
$e^{-x}=1-x+O(x^2)$ imply
\begin{equation}
\label{eq:sharp-weak-linearization}
\E[\psi(X_1^m)]
-
\E[\psi(X_1)]
=
-\E[\psi(X_1)]
(\Gamma_m-\Gamma_\infty)
+
O\Big(
\frac{1}{j^6\kappa_j^{2\alpha}}
\Big),
\end{equation}
uniformly for
$m\in\{\kappa_j,2\kappa_j\}$.
Applying \eqref{eq:sharp-weak-linearization} with
$m=\kappa_j$ and $m=2\kappa_j$ gives
\begin{align*}
&2\E[\psi(X_1^{2\kappa_j})]
-
\E[\psi(X_1^{\kappa_j})]
-
\E[\psi(X_1)]
\\
&\qquad
=
-2\E[\psi(X_1)]
(\Gamma_{2\kappa_j}-\Gamma_\infty)
+
\E[\psi(X_1)]
(\Gamma_{\kappa_j}-\Gamma_\infty)
+
O\left(
\frac{1}{j^6\kappa_j^{2\alpha}}
\right).
\end{align*}
Using \eqref{eq:sharp-C2kappa}, together with
\(
\frac{1}{j^6\kappa_j^{2\alpha}}
=
o\big(
\frac{1}{j^3\kappa_j^\alpha}
\big),
\)
we conclude that
\begin{equation}
\label{eq:sharp-richardson-asymptotic}
2\E[\psi(X_1^{2\kappa_j})]
-
\E[\psi(X_1^{\kappa_j})]
-
\E[\psi(X_1)]
=
\frac{2\E[\psi(X_1)]}{\pi^2(2\pi)^\alpha}
\frac{1+o(1)}
{j^3\kappa_j^\alpha}.
\end{equation}
Since
\(
\E[\psi(X_1)]
=
\exp(
-\frac12\E[Z_\infty^2]
)
>0,
\)
there exists $j_0\in\mathbb N$ such that
\begin{equation*}
\left|
2\E[\psi(X_1^{2\kappa_j})]
-
\E[\psi(X_1^{\kappa_j})]
-
\E[\psi(X_1)]
\right|
\ge
\frac{\E[\psi(X_1)]}{\pi^2(2\pi)^\alpha}
\frac{1}{j^3\kappa_j^\alpha}
\end{equation*}
for every $j\ge j_0$.
Because $\kappa_j=2^{j^2-1}$, for every $\varepsilon>0$,
\begin{align*}
&\kappa_j^{\alpha+\varepsilon}
\left|
2\E[\psi(X_1^{2\kappa_j})]
-
\E[\psi(X_1^{\kappa_j})]
-
\E[\psi(X_1)]
\right|
\ge
c_\alpha
\frac{\kappa_j^\varepsilon}{j^3}
\rightarrow
\infty,
\qquad j\to\infty.
\end{align*}
This proves \eqref{eq:shp2} and completes the proof.
\end{proof}
\bibliographystyle{abbrv}
\bibliography{references}
\end{document}